\documentclass[12pt]{amsart}

\usepackage{amsmath,amsthm,amscd,amsfonts,wasysym,amssymb,epic,eepic,bbm,tikz-cd,mathtools,mathdots, comment,enumitem,fancyhdr,mathrsfs}
\usepackage[pagebackref,colorlinks=true,linkcolor=blue,citecolor=blue]{hyperref}
\usepackage[noabbrev]{cleveref}
\usepackage{MnSymbol}
\usepackage[all,cmtip]{xy}
\usepackage[normalem]{ulem} % strike out
\usepackage[new]{old-arrows}

\allowdisplaybreaks
\newtheorem{thm}{Theorem}[section]
\newtheorem{cor}[thm]{Corollary}
\newtheorem{conj}[thm]{Conjecture}
\newtheorem{lem}[thm]{Lemma}

\newtheorem{prop}[thm]{Proposition}

\newtheorem{question}[thm]{Question}
\theoremstyle{definition}
\newtheorem{rem}[thm]{Remark}

\newtheorem{ans}[thm]{Ansatz}

\newcounter{remarkscounter}

\numberwithin{equation}{section}
\newcommand{\A}{\mathbb{A}}
\newcommand{\GL}{\mathrm{GL}}

\newcommand{\SL}{\mathrm{SL}}
\newcommand{\Sp}{\mathrm{Sp}}
\newcommand{\Mp}{\mathrm{Mp}}
\newcommand{\ZZ}{\mathbb{Z}}

\newcommand{\Gal}{\mathrm{Gal}}

\newcommand{\QQ}{\mathbb{Q}}

\newcommand{\lto}{\longrightarrow}

\newcommand{\OO}{\mathcal{O}}
\newcommand{\CC}{\mathbb{C}}
\newcommand{\RR}{\mathbb{R}}
\newcommand{\GG}{\mathbb{G}}

\newcommand{\GO}{\mathrm{GO}}

\newcommand{\Hom}{\mathrm{Hom}}

\newcommand{\sm}{\mathrm{sm}}

\newcommand{\quash}[1]{}

\theoremstyle{definition}
\newtheorem{defn}[thm]{Definition}

\newenvironment{psmatrix}
  {\left(\begin{smallmatrix}}
  {\end{smallmatrix}\right)}

\numberwithin{equation}{subsection}

\renewcommand{\hat}{\widehat}

\newcommand{\Spec}{\mathop{\mathrm{Spec}}}

\allowdisplaybreaks

\begin{document}

\title{Modulation groups}

\begin{abstract}
Conjectures of Braverman and Kazhdan, Ng\^o and Sakellaridis have motivated the introduction of Schwartz spaces for certain spherical varieties.  We prove that under suitable assumptions these Schwartz spaces are naturally a representation of a group that we christen the modulation group.  This provides a broad generalization of the defining representation of the metaplectic group.  The example of a vector space and the zero locus of a quadric cone in an even number of variables are discussed in detail.  In both of these cases the modulation group is closely related to algebraic groups, and we propose a conjectural method of linking modulation groups to ind-algebraic groups in general.  At the end of the paper we discuss adelization and the relationship between representations of modulation groups and the Poisson summation conjecture.
\end{abstract}


\thanks{Getz is thankful for partial support provided by NSF grant DMS-2400550, and Slipper is supported by NSF grant DMS-2231514. Any opinions, findings, and conclusions or recommendations expressed in this material are those of the authors and do not necessarily reflect the views of the National Science Foundation. Gutti\'erez Terradillos is supported by a research grant (VIL54509) from VILLUM FONDEN}

\author{Jayce R. Getz}
\address{Department of Mathematics\\
Duke University\\
Durham, NC 27708}
\email{jayce.getz@duke.edu}

\author{Armando Gutiérrez Terradillos}
\address{Department of Mathematics\\
Aarhus University\\
Aarhus, Denmark}
\email{armangute@math.au.dk}

\author{Farid Hosseinijafari}
\address{Department of Mathematics\\
Duke University\\
Durham, NC 27708}
\email{farid.hj@math.duke.edu}

\author{Bryan Hu}
\address{Department of Mathematics\\ University of California San Diego\\ San Diego, CA 92093}
\email{brhu@ucsd.edu}

\author{Seewoo Lee}
\address{Department of Mathematics \\ University of California Berkeley \\ Berkeley, CA 94720}
\email{seewoo5@berkeley.edu}

\author{Aaron Slipper}
\address{Department of Mathematics\\
Duke University\\
Durham, NC 27708}
\email{aaron.slipper@duke.edu}

\author{Marie-H\'el\`ene Tom\'e}
\address{Department of Mathematics\\ Duke University\\ Durham, NC 27708}
\email{mariehelene.tome@duke.edu}

\author{HaoYun Yao}
\address{Department of Mathematics\\
Duke University\\
Durham, NC 27708}
\email{haoyun.yao@duke.edu}

\author{Alan Zhao}
\address{Department of Mathematics \\ Columbia University \\
New York, NY, 10027}
\email{asz2115@columbia.edu}
\maketitle

\setcounter{tocdepth}{1}
\tableofcontents

\section{Introduction}
 
Let $F$ be a characteristic zero local field, let $H$ be an affine algebraic group over $F,$ and let $X$ be an affine $H$-scheme such that $X^{\mathrm{sm}}(F) \neq \emptyset.$  Here $X^{\mathrm{sm}} \subseteq X$ is the smooth locus.  We then have a unitary representation
$$
L^2(X^{\mathrm{sm}}(F),\mathcal{L}^{1/2})
$$
of $H(F),$ where $\mathcal{L}^{1/2}$ is the sheaf of half-densities.  If $X^{\mathrm{sm}}(F)$ admits an $H(F)$-eigenmeasure then one can alternately form the Hilbert space using the eigenmeasure (see Remark \ref{rem:half}).

In various settings there are additional automorphisms of $L^2(X^{\mathrm{sm}}(F),\mathcal{L}^{1/2})$ that allow us to extend the action of $H(F)$ to an action of a larger abstract group.  These larger groups have been studied in special cases, but never systematically.  Our aim in this paper is to unify these special cases into a broader framework, and conjecturally relate them to the $F$-points of ind-algebraic groups. 

Our ultimate motive, however, is to use this study to add structure to certain Schwartz spaces that appear in the Poisson summation conjecture.  We give more details in \S \ref{ssec:intro:PS} below.  The constructions in this paper upgrade any suitable Poisson summation formula to a representation generalizing the global Weil representation.

\subsection{Small modulation groups} Let $\psi:F \to \CC^\times$ be a non-trivial character.
Suppose that we are given a (right) representation $V$ of $H$ and an $H$-equivariant map
$$
\omega: X \lto V.
$$
We define a left action of $V^{\vee}(F) \rtimes H(F)$ on $L^2(X^{\mathrm{sm}}(F),\mathcal{L}^{1/2})$ by
\begin{align*}
\mathcal{R}_{\omega}: V^\vee(F) \rtimes H(F) \times L^2(X^{\mathrm{sm}}(F),\mathcal{L}^{1/2}) &\lto L^2(X^{\mathrm{sm}}(F),\mathcal{L}^{1/2})\\
    ((\lambda,  h),f) &\longmapsto \psi(\lambda \circ \omega(x))f(xh).
\end{align*}

\begin{defn} The \textbf{(spectral) small modulation group} $\Psi^{\mathrm{s}}_{\omega}\{F\}$ is the image of $V^\vee(F) \rtimes H(F)$ in  $\mathrm{U}(L^2(X^{\mathrm{sm}}(F),\mathcal{L}^{1/2})).$
\end{defn}
\noindent Here the superscript ${}^{\mathrm{s}}$ is short for small, and for Hilbert spaces $W$ we write $\mathrm{U}(W)$ for the group of unitary operators on $W.$

The definition furnishes a surjection 
\begin{align*}
    \mathcal{R}_\omega:V^\vee(F)\rtimes H(F)\lto \Psi_\omega^{\mathrm{s}}\{F\}
\end{align*}
of abstract groups.  Under natural assumptions on $H$ and $X,$ we check that in fact $\Psi_{\omega}^{\mathrm{s}}\{F\}$ is the $F$-points of an affine algebraic group $\Psi_{\omega}^{\mathrm{s}}$ over $F$ in Proposition \ref{modgp:alg=s}.  We refer to $\Psi_{\omega}^{\mathrm{s}}$ as the \textbf{algebraic small modulation group}.

\subsection{The modulation group}

In many cases there exists an isometry
$$
\mathcal{F}_{X}:L^2(X^{\mathrm{sm}}(F),\mathcal{L}^{1/2}) \lto L^2(X^{\mathrm{sm}}(F),\mathcal{L}^{1/2})
$$
and isomorphisms $\iota:H \to H$ such that 
$$
\mathcal{F}_{X} \circ \mathcal{R}_{\omega}(h)=\mathcal{R}_{\omega}(\iota(h)) \circ \mathcal{F}_{X}
$$
for $h \in H(F).$  Given a character $\chi:\Psi_{\omega}^{\mathrm{s}}(F) \to \CC^\times$ we define the \textbf{modulation group}
\begin{align} \label{mod:grp}
    \Psi_{\omega}\{F\}:=\Psi_{\omega,\chi,\mathcal{F}_{X}}\{F\} =\langle \mathcal{F}_{X},(\mathcal{R}_{\omega} \otimes \chi)(\Psi_{\omega}^{\mathrm{s}}\{F\})\rangle<\mathrm{Aut}(L^2(X^{\mathrm{sm}}(F),\mathcal{L}^{1/2}).
\end{align}
The odd notation is a hint that $\Psi_{\omega}\{F\}$ is not the $F$-points of a group scheme in general.  
  
We compute the modulation group in a family of test cases in \S \ref{sec:Mod:VS} and \S \ref{sec:Mod:cones}. The reader can refer to these sections for more details on the theorems we are about to state.  
In \S \ref{sec:Mod:VS}, we consider the classic setting in which $X=V \cong \GG_a^n$ equipped with the canonical action of $\mathrm{GL}_V.$  We consider two equivariant maps 
$$
\mathrm{id}:X \lto V \quad \textrm{and} \quad \mathrm{Sym}^2:X \lto \mathrm{Sym}^2 V^\vee.
$$
For the symplectic space
$$
W:=V \oplus V^\vee,
$$
 let $\mathrm{H}_W(F)$ be the Heisenberg group, let $\mathrm{Mp}_W(F)$ be the metaplectic group  and let 
$\widetilde{\mathrm{J}}_W(F)=\mathrm{H}_W(F) \rtimes \mathrm{Mp}_W(F)$
be the corresponding Jacobi group.  We point out that we take $\mathrm{Mp}_W(F)$ to be the $8$-fold cover of $\mathrm{Sp}_{W}(F)$ constructed in \cite{GKTtheta}, as opposed to the usual double cover.

\begin{thm} \label{thm:intro:HW}
The modulation group $\Psi_{\mathrm{id}}\{F\}$ is the image of 
$$
\mathrm{H}_W(F) \rtimes \langle w,\GL_V(F)\rangle \leq \widetilde{\mathrm{J}}_W(F)
$$
under the Heisenberg-Weil representation.  Here $w \in \mathrm{Mp}_W(F)$ is defined as in \eqref{w:def}.  The modulation group $\Psi_{\mathrm{Sym}^2}\{F\}$ is the image of a certain subgroup of $\mathrm{Mp}_W(F)$ under the Heisenberg-Weil representation.
\end{thm}
\noindent This theorem shows that $\Psi_{\omega}\{F\}$ need not be the $F$-points of an algebraic group, although in both cases it is closely related to algebraic groups.  

In \S \ref{sec:Mod:cones}, we consider the case where
$X$ is the zero locus of a split non-degenerate quadratic form on a vector space $V \cong \GG_a^{2n}$ equipped with the natural action of the similitude group $\mathrm{GO}_{2n}.$  For technical reasons we extend this to an action of $\GG_m \times \mathrm{GO}_{2n}$ with $\GG_m$ acting by scaling.
In this setting we consider the inclusion $\omega:X \to V.$  
\begin{thm} \label{thm:ortho:case:intro}In the setting above,
$\Psi_{\omega}\{F\}$ is the image of $\mathrm{GO}_{2n+2}(F)$ under the minimal representation. 
\end{thm}

Theorems \ref{thm:intro:HW} and \ref{thm:ortho:case:intro} indicate that the representations
$$
\Psi_{\omega}\{F\} \lto \mathrm{U}(L^2(X^{\mathrm{sm}}(F),\mathcal{L}^{1/2}))
$$
provide an interesting generalization of minimal representations.  One can think of these representations as an enrichment of the Fourier transform on $L^2(X^{\mathrm{sm}}(F),\mathcal{L}^{1/2})$ in the same manner as the Weil representation is an enrichment of the Fourier transform on a vector space.

\subsection{Schwartz spaces} In all of the examples above there is a Schwartz space 
$$
\mathcal{S}(X(F),\mathcal{L}^{1/2})<L^2(X^{\mathrm{sm}}(F),\mathcal{L}^{1/2})
$$
that is preserved by $\Psi_{\omega}\{F\}.$  Based on conjectures of Braverman-Kazhdan \cite{BK-lifting}, Ng\^o \cite{NgoSums}, and Sakellaridis
\cite{SakellaridisSph} we expect this is true in some degree of generality. 
  
Using the examples discussed previously, we highlight properties of Schwartz spaces that to our knowledge, have not been discussed in the literature. 
In most cases we do not have conceptual explanations for these properties.  In fact, we do not even know in what generality they will hold.  

Therefore, to explain our work, we will proceed as follows.  We will label several statements as ``Ansatz.''  What we mean by this is statements that should be true in some unknown level of generality.  We will then isolate two broad families of examples and make the formal conjecture that the ansatz are true in these two families.  Briefly, these are certain reductive monoids and certain horospherical varieties.  Finally, we will prove the conjectures in the simplest working examples in each family.
Organizing the paper in this manner highlights the structural phenomenon that we believe will provide an important guide to future research.  

\subsection{Boundary terms}

Assume that $H$ acts with dense open orbit $X^{\circ}$ on $X.$
 The \textbf{boundary} of the Schwartz space is the quotient
$$
\mathcal{S}(X(F),\mathcal{L}^{1/2})/\mathcal{S}(X^\circ(F),\mathcal{L}^{1/2}).
$$
We can regard $\mathcal{S}(X^\circ(F),\mathcal{L}^{1/2})$ as understood, so the boundary isolates exactly the exotic structure of the full Schwartz space.  

Motivated by the orbit method, we introduce an action of $\Psi_{\omega}^{\mathrm{s}}$ on the cotangent bundle $T^*X^{\mathrm{sm}}$ of $X^{\mathrm{sm}}$ in \S \ref{sec:mod:alg}.   In the setting of Theorem \ref{thm:intro:HW} and Theorem \ref{thm:ortho:case:intro}
we exhibit parallel $\Psi_{\omega}^{\mathrm{s}}\{F\}$-invariant filtrations of  $\mathcal{S}(X(F),\mathcal{L}^{1/2})$ and $\Psi_{\omega}^{\mathrm{s}}$-invariant filtrations of the affine closure  $\overline{T^*X^{\mathrm{sm}}}$ of $T^*X^{\mathrm{sm}}.$  We refer the reader to \S \ref{sec:Boundary} for precise statements. 

\begin{rem}
The paper \cite{Hsu:Asymp} gives a systematic description of the boundary of the Schwartz space for certain horospherical varieties (the boundary is referred to as the asymptotics in loc.~cit.).  This work will likely be crucial for a conceptual understanding of the phenomena we observe qualitatively in \S \ref{sec:Boundary}.
\end{rem}
 
\subsection{Algebraicity}
In the settings of both Theorem \ref{thm:intro:HW} and Theorem \ref{thm:ortho:case:intro} the modulation group is closely related to an algebraic group.  Moreover, as just explained, there is an empirical link between  action of the spectral small modulation group on $\mathcal{S}(X(F),\mathcal{L}^{1/2})$ and the algebraic small modulation group on $\overline{T^*X^{\mathrm{sm}}}.$

Thus, though the modulation group is not necessarily the $F$-points of an algebraic group, we expect that it is closely related to an ind-algebraic group.  To investigate this, we restrict to the Archimedean setting.  Without loss of generality we may assume $F=\RR.$  In \S \ref{sec:DQ} we isolate assumptions under which the action of $\Psi_{\omega}(\RR)$ on $L^2(X^{\mathrm{sm}}(\RR),\mathcal{L}^{1/2})$ induces an action of $\Psi_{\omega}\{\RR\}$ on $\mathcal{D}_X,$ the algebra of algebraic differential operators on $X.$  

In Ansatz \ref{Ansatz} we give a precise mathematical formulation of the following slogan:  The action of $\Psi_{\omega}\{\RR\}$ admits a semiclassical limit, which is an action of an ind-algebraic group 
$\Psi_{\omega}$ on $\overline{T^*X^{\mathrm{sm}}}.$  
We prove this conjecture in the settings of Theorem \ref{thm:intro:HW} and \ref{thm:ortho:case:intro} in \S \ref{ssec:exam}.   

\begin{rem} \label{rem:BZSV}
Ben-Zvi, Sakellaridis, and Venkatesh have suggested that Fourier transforms on $L^2(X^{\mathrm{sm}}(F),\mathcal{L}^{1/2})$ should correspond to choices of Lagrangian subvarieties of $\overline{T^*X^{\mathrm{sm}}}.$ In cases that can be reduced to the vector spaces setting this is discussed in \cite[\S 10.9]{BZSV}.  One consequence of Ansatz \ref{Ansatz} is that, roughly, Fourier transforms pass to automorphisms of the cotangent bundle.  
\end{rem}

\subsection{Global analogues of the metaplectic representation} \label{ssec:intro:PS}

One expects that in many cases there exists Poisson summation formulae for the Schwartz spaces considered above \cite{BK-lifting,NgoSums,Ngo:Hankel,Sakellaridis:Sph:func:sph}.  In the literature these conjectural Poisson summation formulae are usually stated for a limited class of test functions as in Ansatz \ref{ans:PSC}.  To treat general functions one needs to introduce ``boundary terms'' that remain mysterious in general.  We make this more precise in \S \ref{sec:Boundary} and \S \ref{sec:PSC}.  We came to the definition of modulation groups in our study of boundary terms.  We believe that modulation groups will be helpful in illuminating the structure of these terms.

In \S \ref{sec:AR}, we define the adelic modulation group $\Psi_{\omega}\{\A_F\}.$
We moreover show that the Poisson summation formula in a strong form (i.e.~including boundary terms) 
allows one to construct functions 
$$
\Theta_f:\Psi_{\omega}\{F\} \backslash \Psi_{\omega}\{\A_F\} \lto \CC
$$
such that the space
\begin{align} \label{theta:intro}
\left\{\Theta_f: f \in \mathcal{S}\left(X(\A_F),\mathcal{L}^{1/2}\right)\right\}
\end{align}
admits a nontrivial action of $\Psi_{\omega}\{\A_F\}.$
The resulting representation is a generalization of the metaplectic representation.

We mentioned above that in the local setting modulation groups provide an enrichment of the Fourier transform.  The theory discussed above implies that adelic modulation groups provide an enrichment of Poisson summation formulae.  This generalizes the manner in which the adelic metaplectic representation provides an enrichment of the Poisson summation formula for a vector space.

\section*{Acknowledgments}
D.~Kazhdan,  Y.~Sakellaridis and J.~Wang independently suggested to Getz that the boundary terms in the summation formula for a quadric ought to be related to the action of the similitude group on the affine closure of the cotangent space of the cone.
Examining this example convinced us that introducing the small and large modulation groups was natural.  Getz thanks H.~Hahn for her constant encouragement and help with editing.
Yao thanks S-Y. Lee for help with the proof of Lemma \ref{modgp:alg=s:lem1}. This paper was written under the auspices of the Duke Research Scholars program organized by Getz and funded by NSF RTG grant DMS-2231514.

\section{Function spaces}
\quash{
Let $X$ be a scheme smooth and of finite type over the local field $F$ with $X(F) \neq \emptyset.$  If $F$ is non-Archimedean we set $\mathcal{S}(X(F)):=C_c^\infty(X(F)).$  If $F$ is Archimedean we view
$X(F)=\mathrm{Res}_{F/\mathbb{R}}X(\mathbb{R})$ as a real algebraic variety and let $\mathcal{S}(X(F))$  be the Schwartz space defined in  \cite{AG:Nash,Elazar:Shaviv}; this is a nuclear Fr\'echet space. 
Thus in either the non-Archimedean or Archimedean case
\begin{align*}
    C_c^\infty(X(F))\leq\mathcal{S}(X(F)) \leq C^\infty(X(F)).
\end{align*}}

We recall that when $F$ is Archimedean the space of \textbf{tempered functions} on $F^n$ is the space $\mathcal{T}(F^n)$ of all $f \in C^\infty(F^n)$ that are of polynomial growth in the following sense:
for all $\alpha,\alpha'\in \ZZ_{\geq 0}^n$ there exists $p_{\alpha,\alpha'} \in F[x_1,\dots,x_n]$ such that 
$\left|\frac{\partial^{|\alpha|}}{\partial^{\alpha} x}\frac{\partial^{|\alpha'|}}{\partial^{\alpha'} \overline{x}}f(x)\right| \leq \left| p_{\alpha,\alpha'}(x)\right|$ for all $x \in F^n.$  Here we omit $\alpha'$ and the corresponding terms if $F$ is real.
If $F$ is non-Archimedean,  $\mathcal{T}(F^n):=C^\infty(F^n).$  

Let $X$ be an affine scheme of finite type over $F$ with $X^{\mathrm{sm}}(F) \neq \emptyset.$  
\begin{defn} \label{defn:local}
A subspace $\mathcal{S} \leq C^\infty(X^{\mathrm{sm}}(F))$ is \textbf{local} with respect to $X$ if for any  closed immersion $X \to \GG_a^n,$ $\mathcal{S}$ is closed under multiplication by restriction of functions in $\mathcal{T}(F^n).$
\end{defn}
\quash{
\begin{defn} \label{defn:rap}
Let $X' \subseteq X^{\mathrm{sm}}$ be an open subscheme such that $X'(F) \neq \emptyset.$  An $f \in C^\infty(X'(F))$ is \textbf{rapidly decreasing with respect to $X$} if there is a relatively compact neighborhood $\Omega$ of $X(F)-X'(F)$ and a function $\widetilde{f} \in \mathcal{S}_{ES}(X(F))$ such that 
$$
|f(x)| \leq \widetilde{f}(x)
$$
for all $x \in X(F)-\Omega.$
\end{defn}

This definition is constructed to capture the notion of a function with possible singularities along $X-X'$ that is rapidly decreasing away from $X-X'.$  When $F$ is non-Archimedean a function $f \in C^\infty(X'(F))$ is rapidly decreasing with respect to $X$ if and only if its support is contained in a compact subset of $X(F).$
}

\section{Modulation groups}\label{sec:mod:groups}

\subsection{Modulations} \label{ssec:mods}
Let $F$ be 
 a local field of characteristic zero and let $L^2(F^n)$ be defined with respect to a choice of additive Haar measure.
Let $\langle\,,\,\rangle:F^n \times F^n \to F$ be a perfect pairing and let $\psi:F \to \CC^\times$ be a nontrivial additive character.  This gives rise to a Fourier transform $\mathcal{F}_{\psi}:L^2(F^n) \to L^2(F^n).$
The Fourier transform $\mathcal{F}_{\psi}$  intertwines translation $f \mapsto f(\cdot+v)$ with a modulation, namely multiplication by $\overline{\psi}(\langle v,\cdot \rangle).$
If we replace $F^n$ by a nonlinear space $X$ then translations do not exist. However, we can define an analogue of modulation in great generality. 

To explain this, suppose we are given the following data:
\begin{enumerate}[label=(A\text{\arabic*}), ref=A\text{\arabic*}]
     \item \label{action:H} an affine algebraic group $H$ over $F,$
    \item \label{action:HX} an action
    $X \times H \to X$
    of $H$ on an affine scheme $X$ of finite type over $F,$
    \item \label{action:rep} a (right) linear representation
    $
    V \times H \to V$
    with $V \cong \GG_a^n$ for some $n,$ and
    \item \label{action:equi} an $H$-equivariant morphism $\omega:X\to V.$
\end{enumerate}

In the following we denote by $R$ an $F$-algebra, used to define the points of schemes.
Let $V^\vee$ be the dual of $V,$ so that $V^\vee(R) = \Hom_{\textbf{Mod}_R}(V(R),R)$. We form the semi-direct product $V^\vee \rtimes H.$  The group law is  given at the level of points by
\begin{align} \label{strange:semi} \begin{split}
(\lambda, h)(\lambda',h')=(\lambda +\lambda' \circ h,hh').
\end{split}
\end{align}
Here $(\lambda' \circ h)(v):=\lambda'(vh)$ for $(\lambda',h,v) \in V^\vee(R) \times H(F) \times V(R).$

\quash{
\textcolor{red}{HaoYun:}
\begin{align*}
    (\lambda,h)\left((\lambda',h')(\lambda'',h'')\right) &= (\lambda,h)(\lambda'+\lambda''\circ h',h'h'')\\
    &=(\lambda+(\lambda'+\lambda''\circ h')\circ h,hh'h'')\\
    &=(\lambda+\lambda'\circ h+\lambda''\circ h'\circ h,hh'h'')
\end{align*}
\begin{align*}
    \left((\lambda,h)(\lambda',h')\right)(\lambda'',h'')&=(\lambda+\lambda'\circ h,hh')(\lambda'',h'')\\
    &=(\lambda+\lambda'\circ h+\lambda''\circ hh',hh'h'')
\end{align*}}

To proceed we will use the bundle of $1/2$ densities $\mathcal{L}^{1/2}$ on $X^{\mathrm{sm}}(F)$ and the corresponding canonical Hilbert space $L^2(X^{\mathrm{sm}}(F),\mathcal{L}^{1/2}).$ Let
\begin{align*}
\mathcal{R}:H(F) \times L^2(X^{\mathrm{sm}}(F),\mathcal{L}^{1/2}) &\lto L^2(X^{\mathrm{sm}}(F),\mathcal{L}^{1/2})\\
    (h,f) &\longmapsto (x\mapsto f(xh))
\end{align*}
denote the induced unitary representation of $H(F)$. A convenient reference is \cite[\S 3.1]{WWLi:Zeta}.

\quash{The advantage of half-densities is that there is no need to choose a measure on $X^{\mathrm{sm}}(F),$ and the action of $G(F)$ on $L^2(X^{\mathrm{sm}}(F),\mathcal{L}^{1/2})$ is unitary.}

\begin{rem} \label{rem:half}
For readers unfamiliar with this language, we point out that in many cases one can avoid considering half-densities as follows.  Assume that $X^{\mathrm{sm}}$ admits an open dense $H$-orbit $O$ such that $O(F)$ is nonempty. Then $O(F)$ is open  in $X^{\mathrm{sm}}(F)$ \cite[Proposition 3.1]{Conrad_adelic_points}, and its complement is cut off by polynomials so that it has (Borel) measure zero.  Assume moreover that $O(F)$ admits an $H(F)$-eigenmeasure $dx$; thus $d(xh)=\chi(h)dx$ for some quasi-character $\chi:H(F) \to \RR_{>0}$. One has an isometry 
\begin{align*}
\phi_{dx}:  L^2(O(F),dx) &\,\,\tilde{\lto}\,\,  L^2(X^{\mathrm{sm}}(F),\mathcal{L}^{1/2})\\
  f &\longmapsto fdx^{1/2}
\end{align*}
and $\phi_{dx}(\mathcal{R}(h)f)=\chi^{-{1/2}}(h)\mathcal{R}(h)\phi_{dx}(f).$
\end{rem}

The representation $\mathcal{R}$ extends to a unitary representation
\begin{align} \label{Romega} \begin{split}
\mathcal{R}_{\omega}: V^\vee(F) \rtimes H(F) \times L^2(X^{\mathrm{sm}}(F),\mathcal{L}^{1/2}) &\lto L^2(X^{\mathrm{sm}}(F),\mathcal{L}^{1/2})\\
    ((\lambda , h),f) &\longmapsto \psi(\lambda \circ \omega(x))f(xh) \end{split}
\end{align}
of the semi-direct product $V^\vee(F)\rtimes H(F)$. By abuse of notation we let
\begin{align*}
    \mathcal{R}_\omega: V^\vee(F) \rtimes H(F) \lto \mathrm{U}(L^2(X^{\mathrm{sm}}(F),\mathcal{L}^{1/2}))
\end{align*}
denote the corresponding homomorphism. 
The \textbf{(spectral) small modulation group} $\Psi^{\mathrm{s}}_\omega\{F\}$ is defined to be the image of $V^\vee(F)\rtimes H(F)$ in $\mathrm{U}(L^2(X^{\mathrm{sm}}(F),\mathcal{L}^{1/2}))$. In particular we have a surjection
\begin{align}\label{modgp:alg=s:themap}
    \mathcal{R}_\omega: V^\vee(F) \rtimes H(F) \lto \Psi^{\mathrm{s}}_\omega\{F\}.
\end{align}

Under fairly weak assumptions the action of $\Psi^{\mathrm{s}}_{\omega}\{F\}$ on $L^2(X^{\mathrm{sm}}(F),\mathcal{L}^{1/2})$ is irreducible.  To prove this we first prove the following:

\begin{lem} \label{lem:loc:closed}
    Let $Z$ be a separated scheme of finite type over  $F$ equipped with an action of an algebraic group $H.$  Then every $H(F)$-orbit in $Z(F)$ is locally closed.
\end{lem}
\begin{proof}
Every $H$-orbit in $Z$ is locally closed in the Zariski topology by \cite[Proposition 7.20]{Milne:AGbook}.  Thus we may assume that $Z$ is a single $H$-orbit, and that $Z(F)$ is nonempty.  Choosing $z \in Z(F),$ we have an isomorphism of $H$-schemes $H_z \backslash H \cong Z$ by loc.~cit.  Thus by \cite[Lemma 17.4.3]{Getz:Hahn}, every $H(F)$-orbit in $Z(F)$ is closed in $Z(F).$ 
\end{proof}

Choose $x_0 \in X^{\mathrm{sm}}(F).$  Then  $\omega(x_0) \in V(F)=(V^\vee)^\vee(F).$

\begin{thm}\label{IrreducibilityRV}  Assume the orbit $O(x_0)$ of $x_0$ under $H$ is open and dense in $X$ and that $H(F)$ acts transitively on $O(x_0)(F).$  Assume moreover that   $H_{x_0}=H_{\omega(x_0)}.$  Then the representation of $\Psi_{\omega}^s\{F\}$ on $L^2(X^{\mathrm{sm}}(F),\mathcal{L}^{1/2})$ is irreducible. 
\end{thm}
\begin{proof} 
We employ the Mackey machine.  The continuous linear dual of $V^\vee(F)$ is $V(F).$  We point out that $V(F)$ is second countable and $V^\vee(F) \rtimes H(F)/V^\vee(F)= H(F)$ is $\sigma$-compact.  Moreover $H(F)$-orbits in $V(F)$ are locally closed by Lemma \ref{lem:loc:closed}.  Thus $V^\vee(F)$ is Mackey compatible in $V^\vee(F) \rtimes H(F)$ in the terminology of \cite[Remark 4.26(2)]{Kaniuth:Taylor}.  

We have a character
\begin{align*}
\psi_{x_0}:V^\vee(F) &\lto F\\
\lambda &\longmapsto 
\psi(\lambda(\omega(x_0))).
\end{align*}
The stabilizer in $H(F)$ of this character is $H_{\omega(x_0)}(F),$ and the character extends to  $V^\vee(F) \rtimes H_{\omega(x_0)}(F).$ 
By \cite[Theorem 4.29]{Kaniuth:Taylor} we deduce that the unnormalized induction 
\begin{align} \label{induced}
\mathrm{ind}_{V^\vee(F) \rtimes H_{\omega(x_0)}(F) }^{\Psi_{\omega}^s(F)}\left((\lambda, h) \mapsto (\delta_{H}^{-1/2}\delta_{H_{\omega(x_0)}}^{1/2})(h)\psi_{x_0}(\lambda)\right)
\end{align}
is irreducible.  Since we assumed $H_{\omega(x_0)} = H_{x_0}$ the representation \eqref{induced} is 
isomorphic to $L^2(X^{\mathrm{sm}}(F),\mathcal{L}^{1/2}).$  We point out that the modular quasi-characters are introduced due to the fact that $L^2(X^{\mathrm{sm}}(F),\mathcal{L}^{1/2})$ is a space of half-densities, see \cite[(3.2)]{WWLi:Zeta}.
\end{proof}

For use in \S \ref{sec:DQ} we prove the following lemma:

\begin{lem} \label{lem:dense} Suppose that $x_0 \in X(F)$ and the $H$-orbit $O(x_0)$ of $x_0$ is Zariski dense in $X.$  If $H$ is connected then $O(x_0)(F)$ is Zariski dense in $X.$
\end{lem}

\begin{proof}
For clarity we write $|Y|$ for the underlying topological space of a scheme $Y.$
It suffices to show that $\overline{O(x_0)(F)},$ the closure of $O(x_0)(F)$ in $|O(x_0)|,$ is equal to $|O(x_0)|.$  Let $p:|H| \to |O(x_0)|$ be the action map induced by the choice of basepoint $x_0.$  It is surjective.  
The set $p^{-1}(\overline{O(x_0)(F)}) \subseteq |H|$  is closed and contains $H(F).$  On the other hand since $H$ is connected and $F$ is of characteristic zero, \cite[Chapter V, Corollary 18.3]{Borel:LAG} asserts that  $H(F)$ is dense in $|H|.$  Hence $p^{-1}(\overline{O(x_0)(F)}) =|H|$ and $\overline{O(x_0)(F)}=|O(x_0)|.$    
\end{proof}

\subsection{The algebraic small modulation group} \label{ssec:alg:small}
In this subsection we show that under mild assumptions $\Psi_{\omega}^{\mathrm{s}}\{F\}$ is the $F$-points of an algebraic group.  
Let $H_X$ denote the kernel of the $H$-action on $X^{\mathrm{sm}}.$ By \cite[\S II.1.3.6(c)]{DG}, it is a closed normal subgroup scheme of $H$.
Moreover, let $H_V$ be the kernel of action $H$ on $V.$  Let $Z_H$ be the center of $H$ and let $A_H$ be the maximal split subtorus of $Z_H.$
We impose the following assumptions:
\begin{enumerate}[label=(M\text{\arabic*}), ref=M\text{\arabic*}]
    \item\label{omega:span} The image of $X^{\mathrm{sm}}(F)$ under $\omega$ spans $V(F)$.
    \item\label{omega:scale} 
    In the $A_H$-isotypic decomposition
    $$
    V=\bigoplus_{\chi \in X^*(A_H)}V(\chi)
    $$
    one has $V(\chi)=0$ if $\chi$ is trivial.\quash{
    There exists a cocharacter $c:\mathbb{G}_m\to Z_H$ and an integer $k\in\mathbb{Z}_{>0}$ such that 
    \begin{align*}
        \omega(xc(t)) = t^k\omega(x)
    \end{align*}
    for all $x\in X(R)$ and $t\in R^\times$.}
    \item\label{HX:H1=0}$H^1(F,H_X)=0$.
    \item\label{HX:Fpoints} $H_X(F)=\{h\in H(F): xh=x\text{ for all }x\in X^{\mathrm{sm}}(F)\}$.
\end{enumerate}
Since $V(F)$ spans $V(R)$ as an $R$-module, it follows that \eqref{omega:span} is equivalent to 
\begin{enumerate}[label=(M\text{\arabic*$'$}), ref=M\text{\arabic*$'$}]
    \item \label{omega:span:1} For any $F$-algebra $R$,  $\omega(X^{\mathrm{sm}}(R))$ spans $V(R)$ as an $R$-module.
\end{enumerate}
In particular, under \eqref{omega:span} we have 
\begin{align} \label{HXHV}
H_X \leq H_V.
\end{align}

\begin{lem} \label{lem:M:equiv} Assume that  $\omega:X \to V$ is an immersion.
Then \eqref{omega:span} implies $H_X=H_V$ and \eqref{HX:Fpoints}.  
\end{lem}
\begin{proof} By \eqref{HXHV}, the assumption \eqref{omega:span} implies $H_X \leq H_V.$ Since immersions are injective on points we deduce $H_X=H_V.$   Moreover, 
\begin{align} \label{HVid}
H_V(F) = \{h\in H(F): vh=v\text{ for all }v\in V(F)\}
\end{align}
because $V(F)$ spans $V(R)$ as an $R$-module. 
Applying \eqref{omega:span} to \eqref{HVid} we see 
$$
H_V(F) = \{h\in H(F): xh=x\text{ for all }x\in X^{\mathrm{sm}}(F)\}.
$$
Since $H_X(F)=H_V(F)$ we deduce \eqref{HX:Fpoints}.
\end{proof}

\begin{lem}\label{modgp:alg=s:lem1} The kernel of the homomorphism $\mathcal{R}_\omega$ of \eqref{modgp:alg=s:themap} is 
\begin{align*}
     \left\{(0,h)\in V^\vee(F)\rtimes H(F): xh = x\text{ for all }x\in X^{\mathrm{sm}}(F)\right\},
\end{align*}
which is canonically isomorphic to $H_X(F).$
\end{lem}
\begin{proof} Let $(\lambda,h)\in \ker\mathcal{R}_\omega$\quash{be in the kernel}. Suppose $x\neq xh$ for some $x\in X^{\mathrm{sm}}(F)$. Then there exists $f \in L^2(X^{\mathrm{sm}}(F),\mathcal{L}^{1/2})$ such that $x \in \mathrm{supp}(f)$ while $xh \notin \mathrm{supp}(f)$. 
This implies $\mathcal{R}_\omega(\lambda,h)f\neq f$, which is a contradiction. Hence $h \in H_X(F)$ by \eqref{HX:Fpoints}.

Consequently, for any  $(x,(\lambda,h)) \in X^{\mathrm{sm}}(F) \times \ker \mathcal{R}_{\omega}$
\begin{align*}
     \psi(\lambda \circ \omega (x)) = 1.
\end{align*}
Let $V^\vee=\bigoplus_{\chi \in X^*(A_H)} V^\vee(\chi)$ be the $A_H$-isotypic decomposition of $V^\vee.$  Write $\lambda=\sum_{\chi}\lambda_{\chi}$ with respect to this decomposition.
For $a \in A_H(F)$ and $x\in X^{\mathrm{sm}}(F)$, we have
\begin{align} \label{is1}
    1 = \psi(\lambda \circ \omega (xa)) = \psi\left(\sum_{\chi}\lambda_{\chi}(\omega(xa))\right) = 
\psi\left(\sum_{\chi}\chi(a)\lambda_{\chi}(\omega(x))\right).
\end{align}
 We claim that $\lambda_{\chi}(\omega(x))=0$ for all $\chi.$
Assuming the claim,  
 we can apply \eqref{omega:span} to see that $\lambda_{\chi}=0$ for all $\chi$ and complete the proof.  

Assume to the contrary that $\lambda_{\chi_0}(\omega(x)) \neq 0$ for some $\chi_0.$  By
\eqref{omega:scale} this implies that $\chi_0$ is nontrivial.

Choose a cocharacter $\nu:\GG_m \to A_H$ such that $\chi_0 \circ \nu$ is nontrivial.  Then \eqref{is1} implies
\begin{align} \label{is12}
    1  = 
\psi\left(\sum_{\chi}t^{k_{\chi}}\lambda_{\chi}(\omega(x))\right)
\end{align}
for some integers $k_{\chi} \in \ZZ$ with $k_{\chi_0} \neq 0.$  Upon choosing a different $\nu$ and $\chi_0$ if necessary, we can assume that $k_{\chi_0}>0$ and $k_{\chi_0}>k_{\chi}$ if $\chi \neq \chi_0.$

Let $F_0=\RR$ if $F$ is Archimedean, and let $F_0=\QQ_p$ if $F$ is a non-Archimedean local field of characteristic zero and residual characteristic $p.$  
The map $(\cdot)^{k_{\chi_0}}:F^\times \to F^\times$ is open, as is the map $\mathrm{tr}_{F/F_0}:F \to F_0$.  It follows that there is a $u \in F^\times$ such that if $N>0$ one has
$$
\left|\mathrm{tr}_{F/F_0}\left(\sum_{\chi}(ut)^{k_{\chi}}\lambda_{\chi}(\omega(x))\right)\right|>N
$$
for $t \in F_0$ with $|t| \gg_N 1.$  This contradicts \eqref{is12}.  
\end{proof}

\quash{The claim is easy to check in  the Archimedean case. If $F$ is non-Archimedean, let $\varpi$ be a uniformizer of $F.$  The map  $(\cdot)^k:F^\times \to F^\times$ is open.  Thus we can find $N\in\mathbb{Z}_{\geq 1}$ such that $1+\varpi^N\OO_F\leq (1+\varpi\OO_F)^k$. Then for $n\in\mathbb{Z}$, we have $\varpi^{-nk}(1+\varpi^N\OO_F)\subseteq (\varpi^{-n}(1+\varpi\OO_F))^k\subseteq \ker\psi'.$ Since $\varpi^{-nk}\in\ker\psi'$, we see $\varpi^{-nk+N}\OO_F\leq \ker\psi'$ for all $n\in\mathbb{Z}$. The claim then follows from the observation that $F = \bigcup\limits_{n\in\mathbb{Z}}\varpi^{-nk+N}\OO_F.$}

\begin{defn} \label{def:alg-small-mod} The \textbf{algebraic small modulation group} $\Psi^{\mathrm{s},\mathrm{alg}}_{\omega}$ of $\omega:X \to V$ is 
\begin{align*}
    \Psi^{\mathrm{s},\mathrm{alg}}_{\omega} := V^\vee \rtimes (H/H_X).
\end{align*}
\end{defn}

\begin{prop}\label{modgp:alg=s} The map $\mathcal{R}_\omega$ of \eqref{modgp:alg=s:themap} induces an isomorphism $\Psi^{\mathrm{s},\mathrm{alg}}_{\omega}(F)\cong \Psi_{\omega}^{\mathrm{s}}\{F\}.$
\end{prop}

\begin{proof} By Lemma \ref{modgp:alg=s:lem1} the map $\mathcal{R}_\omega$ induces an isomorphism
\begin{align*}
    V^\vee(F)\rtimes (H(F)/H_X(F))\cong \Psi_\omega^{\mathrm{s}}\{F\}.
\end{align*}
On other other hand, by \eqref{HX:H1=0} the left hand side is $\Psi^{\mathrm{s},\mathrm{alg}}_{\omega}(F)$. 
\end{proof}

Due to the proposition we omit the superscript $\mathrm{alg}$ from $\Psi^{\mathrm{s},\mathrm{alg}}_{\omega}.$  The algebraic small modulation group is the algebraic group $\Psi^{\mathrm{s}}_{\omega}$ and the spectral small modulation group is its $F$-points $\Psi^{\mathrm{s}}_{\omega}(F).$  Thus in this setting there is no real need to use the adjectives ``algebraic'' and ``spectral.''

\subsection{The action on the cotangent bundle}  \label{sec:mod:alg}
We pause to explain how $V^\vee \rtimes H$ acts on the cotangent bundle $T^*X^{\mathrm{sm}}$ of $X^{\mathrm{sm}}.$  We will not use the results of this subsection until \S \ref{sec:DQ} and \ref{sec:Boundary}.  For this subsection $F$ is an arbitrary field.

For lack of a reference, we give an explicit description of the functor of points of the cotangent space.  This is necessary so the reader can check that the actions we construct below are well-defined.
Let $X_0$ be a smooth scheme over a field $F.$  The cotangent bundle of $X_0$ is defined as 
\begin{align*}
    T^*X_0 := \underline{\mathrm{Spec}}_{X_0}\mathrm{Sym}_{\OO_{X_0}}(\Omega_{X_0/F}^\vee).
\end{align*}
where $\Omega_{X_0/F}:=\Omega_{X_0/\Spec F}$ is the sheaf of K\"ahler differentials of $X_0$. Since $X_0$ is smooth, $\Omega_{X_0/F}$ is a finite locally free $\OO_{X_0}$-module. Using the description of the functor of points of a relative spectrum in \cite[\href{https://stacks.math.columbia.edu/tag/01LQ}{Tag 01LQ}]{stacks-project}, for $F$-algebras $R$ we have
\begin{align*}
    T^*X_0(R)
    &=\left\{(x,\phi): x\in X_0(R),\,\phi\in \Hom_{\textbf{Alg}_{\OO_{X_0}}}(\mathrm{Sym}_{\OO_{X_0}}\Omega_{X_0/F}^\vee,x_*\OO_{\mathop{\mathrm{Spec}} R})\right\}\\
    &=\left\{(x,\phi): x\in X_0(R),\,\phi\in \Hom_{\textbf{Mod}_{\OO_{X_0}}}(\Omega_{X_0/F}^\vee,x_*\OO_{\mathop{\mathrm{Spec}} R})\right\}\\
    &=\left\{(x,\phi): x\in X_0(R),\,\phi\in \Hom_{\textbf{Mod}_{\OO_{\Spec R}}}(x^*(\Omega_{X_0/F}^\vee),\OO_{\mathop{\mathrm{Spec}} R})\right\}.
\end{align*}
In the first and last equation we have used the standard identities in \cite[\href{https://stacks.math.columbia.edu/tag/0094}{Tag 0094}]{stacks-project}.

The field-valued points of the functor admits a simpler description. Let $x\in X_0(F)$. We point out that duals commute with pullbacks for finite locally free modules, so 
\[
x^*(\Omega_{X_0/F}^\vee)\cong (x^*\Omega_{X_0/F})^\vee.
\]
Hence we obtain that
\[
\Hom_{\textbf{Mod}_{\OO_{\Spec F}}}(x^*(\Omega_{X_0/F}^\vee),\OO_{\mathop{\mathrm{Spec}} F}) \cong (x^*\Omega_{X_0/F})(\Spec F)\cong T_{X_0,x}^{*},
\]
where $T_{X_0,x}^*$ is the Zariski cotangent space of $X_0$ at $x$. Here we are identifying $x \in X_0(F)$ with a closed point of the underlying topological space of $X_0.$ In other words,
\begin{align} \label{tangent:field}
    T^*X_0(F) = \left\{(x,v): x\in X_0(F),\,v\in T_{X_0,x}^*\right\}.
\end{align}
For this reason, we define
\begin{align*}
    T^*_{X_0,x} := \Hom_{\textbf{Mod}_{\OO_{\Spec R}}}(x^*(\Omega_{X_0/F}^\vee),\OO_{\mathop{\mathrm{Spec}} R})
\end{align*}
for any $x\in X_0(R)$. This is naturally an $R$-module.

\begin{lem}\label{action:H:general} Assume $X_0$ admits a right $H$-action.  Then $T^*X_0$ admits a canonical $H$-action.  It is given on points in an $F$-algebra $R$ by
\begin{align*}
T^*X_0(R) \times H(R) &\lto T^*X_0(R)\\
((x,\phi),h) &\longmapsto (xh,h^*\phi)
\end{align*}
where $h^*:T^*_{X_0,x}\to T^*_{X_0,xh}$ is given by precomposing with the isomorphism 
$$
(xh)^*(\Omega_{(X_0)_R/R}^\vee)\lto x^*(\Omega_{(X_0)_R/R}^\vee).
$$\hfill\qedsymbol
\end{lem}\quash{
\begin{proof} Let $\alpha:H(R)\to\mathrm{Aut}_{\textbf{Sch}_R}((X_0)_R)$ denote the group anti-homomorphism induced by the right action of $H$ on $X_0$. Let $\varphi\in \mathrm{Aut}_{\textbf{Sch}_R}((X_0)_R)$. By \cite[\href{https://stacks.math.columbia.edu/tag/06B6}{Tag 06B6}]{stacks-project} we have an isomorphism $(\varphi^{-1})^*\Omega_{(X_0)_R/R}\stackrel{\sim}{\to}\Omega_{(X_0)_R/R}$. Dualizing and taking relative spec, we obtain an isomorphism $(X_0)_R\times_{(X_0)_R}T^*(X_0)_R\stackrel{\sim}{\to} T^*(X_0)_R$ of schemes over $(X_0)_R$. Pulling back along $\varphi\times\mathrm{id}_{T^*(X_0)_R}$ yields an isomorphism $T^*\varphi:T^*(X_0)_R\stackrel{\sim}{\to} T^*(X_0)_R$ of $R$-schemes fitting into a commuting square
\begin{center}
    \begin{tikzcd}
        T^*(X_0)_R\arrow[d]\arrow[r,"{T^*\varphi}"]&T^*(X_0)_R\arrow[d]\\
        (X_0)_R\arrow[r,"\varphi"]&(X_0)_R.
    \end{tikzcd}
\end{center}
One checks that $\varphi\mapsto T^*\varphi$ defines a homomorphism $\mathrm{Aut}_{\textbf{Sch}_R}((X_0)_R)\to \mathrm{Aut}_{\textbf{Sch}_R}(T^*(X_0)_R)$, so that we obtain a group anti-homomorphism $H(R)\to \mathrm{Aut}_{\textbf{Sch}_R}(T^*(X_0)_R)$. Varying $R$, this defines a right $H$-action on $T^*X_0$. 

The description on points follows from the functor of points definition of a relative spec.

Note the isomorphism $(xh)^*(\Omega_{(X_0)_R/R}^\vee)\to x^*(\Omega_{(X_0)_R/R}^\vee)$ is obtained by applying $x^*$ to the isomorphism $\alpha(h)^*(\Omega_{(X_0)_R/R}^\vee)\to \Omega_{(X_0)_R/R}^\vee$, and note that $x^*\alpha(h)^* = (xh)^*$.
\end{proof}}

\begin{lem}\label{action:Vv:general}  Let $V \cong \GG_a^n$ over $F$ and let $\omega:X_0\to V$ be a morphism of $F$-schemes. There is an action of $V^\vee$ on $T^*X_0$ given on points in an $F$-algebra $R$ by
\begin{align*}
T^*X_0(R) \times V^\vee(R) &\lto T^*X_0(R)\\
((x,\phi),\lambda) &\longmapsto (x,\phi+\omega_x^*\lambda).
\end{align*}
Here $\omega_x^*:V^\vee(R) = T^*_{V,\omega(x)} \to T^*_{X_0,x}$ is given by precomposing with $$
\Omega_{(X_0)_R/R}^\vee\lto (\omega^*\Omega_{V_R/R})^\vee \cong\omega^*\Omega_{V_R/R}^\vee.
$$\qed
\end{lem}

We return to the setting of \eqref{action:H}, \eqref{action:HX}, \eqref{action:rep} and \eqref{action:equi}; the field $F$ can be arbitrary. 
\begin{lem} \label{lem:cot:action} We have an action of $V^\vee\rtimes H$ on $T^*X^{\mathrm{sm}}$ given on the level of points by
\begin{align} \label{action:on:cotan}\begin{split}
  T^*X^{\mathrm{sm}}(R) \times V^\vee(R) \rtimes H(R) &\lto T^*X^{\mathrm{sm}}(R)\\
    ((x,\phi),(\lambda,h)) &\longmapsto \left(xh,h^*\phi+\omega_{xh}^*(\lambda\circ h^{-1})\right).
\end{split}
\end{align}
Here we are using notation as in Lemma \ref{action:H:general} and Lemma \ref{action:Vv:general}. \qed
\end{lem}

\subsection{Fourier transforms}
\label{ssec:FT}
To proceed we require Fourier transforms.  We begin with the abstract setup.  Let $F$ be a local field.  
Let 
$$
 \nu:\GG_m \lto Z_H
$$
be a cocharacter.  We assume there is an automorphism $\iota:H \to H$ such that $\iota \circ \nu(x)=\nu(x)^{-1}.$
 We then assume the existence of an isometry
\begin{align} \label{Fi}
\mathcal{F}_{X}:L^2(X^{\mathrm{sm}}(F),\mathcal{L}^{1/2}) \tilde{\lto} L^2(X^{\mathrm{sm}}(F),\mathcal{L}^{1/2})
\end{align}
such that 
\begin{align} \label{Fi:equiv}
    \mathcal{F}
    _{X} \circ \mathcal{R}(h)=\mathcal{R}(\iota(h)) \circ \mathcal{F}_{X}.
\end{align}
Finally, let 
\begin{align}
    \chi:\Psi_{\omega}^{\mathrm{s}}(F) \lto \CC^\times
\end{align}
be a character.

\begin{defn}  \label{defn:mod:group}  The \textbf{modulation group} $\Psi_{\omega}\{F\}$ attached to $\omega,$ $\mathcal{F}_X,$ and $\chi$ is the subgroup
of $\mathrm{U}(L^2(X^{\mathrm{sm}}(F),\mathcal{L}^{1/2}))$ generated by $\mathcal{F}_{X}$  and $(\mathcal{R}_{\omega} \otimes \chi)(\Psi^{\mathrm{s}}_{\omega}(F)):$
\begin{align*}
\Psi_{\omega}\{F\}:=\Psi_{\omega,\mathcal{F}_{X},\chi}\{F\}    :=\langle \mathcal{F}_X,(\mathcal{R}_{\omega} \otimes \chi)(\Psi_{\omega}^{\mathrm{s}}(F))\rangle.
\end{align*}
\end{defn}

\begin{rem} \label{rem:split:assum}
In \eqref{modif:mod:group} we will give an example in which it is convenient to twist the action of the small modulation group by a character as in Definition \ref{defn:mod:group}.  
 It would be desirable to have a conceptual explanation for this phenomenon.  In any case one should view Definition \ref{defn:mod:group} as a working definition.   
 \end{rem}
\noindent In some situations it is more natural to work with a family of Fourier transforms (see \cite{BK:basic:affine}). The considerations of this paper adapt in the natural manner to this setting.

\subsection{Schwartz spaces}
\label{ssec:Schwartz}
For many purposes, the Hilbert space $L^2(X^{\mathrm{sm}}(F),\mathcal{L}^{1/2})$ is unwieldy.  Instead, one would like to work with an $H(F)$-invariant Schwartz space
$$
\mathcal{S}(X(F),\mathcal{L}^{1/2}) \leq L^2(X^{\mathrm{sm}}(F),\mathcal{L}^{1/2}) \cap C^\infty(X^{\mathrm{sm}}(F),\mathcal{L}^{1/2})
$$
satisfying the following axioms:
\begin{enumerate}[label=($\mathcal{S}$\text{\arabic*}), ref=$\mathcal{S}$\text{\arabic*}]
    \item\label{assum:preserve} $\mathcal{S}(X(F),\mathcal{L}^{1/2})$ is preserved by $\mathcal{F}_X.$
     \item \label{assum:local} $\mathcal{S}(X(F),\mathcal{L}^{1/2})$ is local with respect to $X^{\mathrm{sm}} \subseteq X$ in the sense of Definition \ref{defn:local}.
      \item \label{assum:dense} $\mathcal{S}(X(F),\mathcal{L}^{1/2})$ is dense in $L^2(X^{\mathrm{sm}}(F),\mathcal{L}^{1/2}).$
      \item \label{assum:sm} $\mathcal{S}(X^{\mathrm{sm}}(F),\mathcal{L}^{1/2}) \leq \mathcal{S}(X(F),\mathcal{L}^{1/2}).$ 
   \quash{ \item \label{assum:rd} $\mathcal{S}(X(F),\mathcal{L}^{1/2})$ is rapidly decreasing with respect to $X^{\mathrm{sm}} \subseteq X$ in the sense of Definition \ref{defn:rap}.}
\end{enumerate}
Here $\mathcal{S}(X^{\mathrm{sm}}(F),\mathcal{L}^{1/2})=C_c^\infty(X^{\mathrm{sm}}(F),\mathcal{L}^{1/2})$ when $F$ is non-Archimedean and it is defined as in \cite{AG:Nash} when $F$ is Archimedean.

Clearly \eqref{assum:sm} implies \eqref{assum:dense}.  We separate the two because \eqref{assum:dense} is easier to prove in practice.
There are additional desiderata one could impose, but we prefer to keep things as simple as possible.

\begin{lem}
    Under \eqref{assum:preserve}, \eqref{assum:local}, and \eqref{assum:dense} the Schwartz space $\mathcal{S}(X(F),\mathcal{L}^{1/2})$ is preserved by $\Psi_{\omega}\{F\}$ and the induced map
   $\Psi_{\omega}\{F\} \to \mathrm{Aut}_\CC(\mathcal{S}(X(F),\mathcal{L}^{1/2}))$
    is injective. \qed
\end{lem}

In favorable circumstances the following holds:
\begin{ans} \label{Ans:Schwartz} There exists $\mathcal{F}_X$ satisfying \eqref{Fi:equiv} and a Schwartz space $\mathcal{S}(X(F),\mathcal{L}^{1/2})$ satisfying \eqref{assum:preserve}, \eqref{assum:local}, \eqref{assum:dense}, and \eqref{assum:sm}.
\end{ans}

\begin{question} When Ansatz \ref{Ans:Schwartz} holds, is there a notion of smooth representation for $\Psi_{\omega}\{F\}$ such that $\mathcal{S}(X(F),\mathcal{L}^{1/2})$ is the space of smooth vectors in $L^2(X^{\mathrm{sm}}(F),\mathcal{L}^{1/2})$?
\end{question}

When $\Psi_{\omega}\{F\}$ is a Lie group or a locally compact totally disconnected group it is reasonable to use the usual notion of smoothness.  However, we do not know if $\Psi_{\omega}\{F\}$ is a Lie group or a locally compact totally disconnected group in general. 

In the setting of Theorems \ref{thm:intro:HW}  and \ref{thm:ortho:case:intro} the answer to this question is affirmative if we use the usual notion of smoothness (see Proposition \ref{prop:smooth1} and \eqref{cone:Sch} below).

\section{Reductive monoids}
\label{sec:RM}

\subsection{Reductive monoids}
We briefly review Ng\^o's construction of $L$-monoids from \cite[\S5]{Ngo:Hankel}.
Let $F$ be a local or global field of characteristic $0$, and let $G$ be a connected reductive group over $F$.  Let $\rho:{}^LG\to \GL_{V}(\mathbb{C})$ be a finite dimensional representation of ${}^LG.$  Here we use the Galois form of the $L$-group.

 One has $\rho|_{Z_{\widehat{G}}^\circ}=\chi_1 \oplus \cdots \oplus \chi_{t}$
for some characters $\chi_i \in X^*(Z_{\widehat{G}}^\circ),$ possibly occurring with multiplicity.  We assume that 
\begin{align} \label{chi:assum}
    \chi_i \neq 1 \textrm{ for all }i \textrm{ and the }\chi_i \textrm{ generate a strictly convex cone }C_{Z_{\widehat{G}}^\circ}(\rho) \subset X^*(Z_{\widehat{G}}^\circ)_{\RR}.
\end{align}
Here a strictly convex cone is a convex cone containing no line.

Fix an algebraic closure $\overline{F}$ of $F$, and let $G_{\overline{F}}$ be the base change of $G$ to $\overline{F}$. Let $T\leq G$ be a maximal torus and $B$ a Borel subgroup containing $T_{\overline{F}}$. 
Let $\Omega(\rho)\subseteq X^*(\widehat{T}) = X_*(T_{\overline{F}})$ be the convex span of the weights of $\rho$ and let $C(\rho)\subseteq X_*(T_{\overline{F}})_{\RR}$ denote the cone generated by $\Omega(\rho).$
As explained in the proof of \cite[Proposition 5.1]{Ngo:Hankel}, assumption \eqref{chi:assum} implies that $C(\rho)$ is strictly convex.
Let $\alpha_1,\ldots,\alpha_r\in X^*(T_{\overline{F}})$ be a set of nonzero elements lying in the positive Weyl chamber with respect to $B$ such that their $W(G_{\overline{F}},T_{\overline{F}})(\overline{F})$-orbits generate $\Omega(\rho)^\vee:=\{\alpha\in X^*(T_{\overline{F}}): \langle\alpha,\Omega(\rho)\rangle\geq 0\}$. Let $\omega_{\alpha_i}:G_{\overline{F}}\to \GL_{V_{\alpha_i}}$ denote a representation of highest weight $\alpha_i$, and let
\begin{align} \label{omega}
\omega_{\rho}:G_{\overline{F}}\lto \bigoplus_{i=1}^r \mathrm{End}_{V_{\alpha_i}}
\end{align} 
be the canonical map. We point out that $\omega_{\rho}$ depends on the choice of $\alpha_i,$ but we will not encode this into the notation.  

By our assumption that the characteristic of $F$ is zero, claim ($\text{a}'$) in the proof of \cite[Proposition 2.3.2]{Wang:monoid} implies the following:
\begin{lem} \label{lem:is:normal}
    The closure of the image of $\omega_{\rho}$ in $\bigoplus_{i=1}^r \mathrm{End}_{V_{\alpha_i}}$ is normal. \qed
\end{lem}

In \cite[\S 5]{Ngo:Hankel},  the monoid $M_{\rho,\overline{F}}$ is defined to be the normalization of the closure of the image of $\omega_{\rho}.$  By Lemma \ref{lem:is:normal}, in our setting no normalization is necessary.
In \cite[Proposition 5.1]{Ngo:Hankel} it is proved that $M_{\rho, \overline{F}}$ admits an $F$-model independent of the choice of $\alpha_1,\ldots,\alpha_r;$ hence $M_{\rho,\overline{F}}$ is independent of the choice of $\alpha_1,\dots,\alpha_r.$ 
We will proceed slightly differently to define our model so that the whole representation $\omega_{\rho}$ descends.  In any case by loc.~cit.~the unit group of $M_{\rho,\overline{F}}$ is $G_{\overline{F}}$ and the
 map $G_{\overline{F}} \to M_{\rho,\overline{F}}$ is an open immersion.

\begin{lem} \label{lem:omegarho}
 For an appropriate choice of $\alpha_1,\dots,\alpha_r$ (possibly with multiplicity) the morphism $\omega_{\rho}$ descends to a morphism
$\omega_{\rho}:G \to E_{\omega_{\rho}}:=\bigoplus_{\mathfrak{O}}E_{\mathfrak{O}}
$
where each $E_{\mathfrak{O}}$ is an irreducible representation of $G \times G$, with $E_{\mathfrak{O}} \cong E_{\mathfrak{O}'}$ if and only if $\mathfrak{O}=\mathfrak{O}'.$
\end{lem}

\begin{proof} 
There is a natural action of $\Gal(\overline{F}/F)$ on $X^*(T_{\overline{F}}).$  We consider the twisted action 
\begin{align*} 
\Gal(\overline{F}/F) \times X^*(T_{\overline{F}}) \lto X^*(T_{\overline{F}})\\
(\sigma,\nu ) \longmapsto w(\sigma \circ \nu)
\end{align*}
where $w \in W(G_{\overline{F}},T_{\overline{F}})(\overline{F})$ is the unique element such that the automorphism $\nu \mapsto w(\sigma \circ \nu)$ preserves the positive Weyl chamber with respect to $B.$  This is the same action as that given in \cite[\S 3.1]{Tits:representations}.  

As mentioned above, without changing $M_{\rho,\overline{F}}$
we may enlarge
$\{\alpha_1,\dots,\alpha_r\}.$  From this fact and the definition of $M_{\rho,\overline{F}}$ as a closure we can even enlarge $\{\alpha_1,\dots,\alpha_r\}$ to be a multiset, or in other words allow multiplicities greater than $1$ in the representation $\omega_{\rho}.$  We first enlarge $\{\alpha_1,\dots,\alpha_r\}$  so that it is a union of $\Gal(\overline{F}/F)$-orbits.  For each orbit $\mathfrak{O}$ choose $\alpha_{i} \in \mathfrak{O}$ and let $F_{i}$ be the fixed field of the stabilizer $\Gal(\overline{F}/F)_{\alpha_{i}}.$ 

By \cite[Th\'eor\`eme 7.2]{Tits:representations}, to $\alpha_i$ one can attach a division algebra $D_{i}$ over $F_i$ and a finite rank $D_i$-module $W_{\alpha_i}$ such that 
\begin{align} \label{descent}
V_{\alpha_i}':=\mathrm{Res}_{F_i/F}\mathrm{Res}_{D_i/F_i}W_{\alpha_i}
\end{align}
is an irreducible representation of $G$ and 
$(V_{\alpha_i}')_{\overline{F}} \cong V_{\alpha_i}^{[F_i:F]\dim_{F_i}D_i}
$
as representations of $G_{\overline{F}}.$
By loc.~cit., up to isomorphism over $F,$ \eqref{descent} does not depend on the choice of $\alpha_i$ in the $\Gal(\overline{F}/F)$-orbit,
 so we denote it by $V_{\mathfrak{O}}.$  Let $E_{\mathfrak{O}}:=V_{\mathfrak{O}}^{\vee} \otimes_F V_{\mathfrak{O}};$ it is an irreducible representation of $G \times G.$

Thus we choose the multiset $\{\alpha_1,\dots,\alpha_r\}$ so that it is a union of $\Gal(\overline{F}/F)$-orbits, where the multiplicity of an orbit $\mathfrak{O} \ni \alpha_i$ is $([F_i:F]\dim_{F_i}D_i)^2.$  Then $(V_{\alpha_i}^{\vee})^{[F_i:F]\dim_{F_i}D_i} \otimes V_{\alpha_i}^{[F_i:F]\dim_{F_i}D_i}$ descends to $E_{\mathfrak{O}}.$  
\end{proof}

We let $M_{\rho}$ be the closure of $\omega_{\rho}(G)$ in $E_{\omega_{\rho}}.$  We refer to it as the \textbf{$L$-monoid} attached to $\rho.$  The notation is consistent, as the base change of $M_{\rho}$ to $\overline{F}$ is $M_{\rho,\overline{F}}.$  This model may be identified with the model constructed by Ng\^o by the uniqueness statement in \cite[Proposition 5.1]{Ngo:Hankel}.

\begin{lem} One has $\mathrm{codim}(M_\rho - M_\rho^{\mathrm{sm}},M_\rho)\geq 2$. 
\end{lem}

\begin{proof} By Lemma \ref{lem:is:normal} $M_\rho$ is normal, so it is regular in codimension $1$ \cite[\href{https://stacks.math.columbia.edu/tag/0345}{Tag 0345}]{stacks-project}. Over a field of characteristic zero regular is equivalent to smooth  \cite[\href{https://stacks.math.columbia.edu/tag/0B8X}{Tag 0B8X}]{stacks-project}.
\end{proof}

\subsection{Modulation groups for reductive monoids} \label{sec:monoid-mod} 
  For simplicity we assume that we are given an isomorphism 
$$
d:G/G^{\mathrm{der}} \tilde{\lto} \GG_m
$$
and that $\rho \circ \hat{d}:\GG_m \lto \GL_V$ is the inclusion of $\GG_m$ into the center of $\GL_V.$  
Let $T \leq G$ be a maximal torus, and let
\begin{align}
\Omega(\rho)^{\perp}:=\{\alpha \in X^*(T_{\overline{F}}): \langle \alpha,\chi\rangle=0 \textrm{ for all }\chi \in \Omega(\rho)\}.
\end{align}
We assume that 
\begin{align} \label{perp0}
\Omega(\rho)^{\perp}=0.
\end{align}

We now explain how to define modulation groups in this setting.  We take
\begin{itemize}
    \item $H=G\times G,$
    \item $X = M_\rho$ with right $G\times G$-action,
    \item $V=E_{\omega_\rho}$
    with right $G\times G$-action,
    \item $\omega:=\omega_\rho:M_\rho\to E_{\omega_\rho}$ as in Lemma \ref{lem:omegarho}.
\end{itemize}  We henceforth assume that 
\begin{align} \label{H1:assump}
H^1(F,Z_G)=0.
\end{align}

\begin{lem} \label{lem:monoid:case} In the setting above, under assumptions \eqref{perp0} and \eqref{H1:assump} the assumptions \eqref{omega:span}, \eqref{omega:scale},  \eqref{HX:H1=0}, and \eqref{HX:Fpoints} hold.
\end{lem}

\begin{proof} 
The span of $\omega_{\rho}(G)$ is a subrepresentation of $E_{\omega_\rho} =\bigoplus_{\mathfrak{O}}E_{\mathfrak{O}}.$  By Lemma \ref{lem:omegarho} $E_{\mathfrak{O}} \cong E_{\mathfrak{O}'}$ if and only if $\mathfrak{O}=\mathfrak{O}',$ and it follows from the proof of loc.~cit. that the projection of $\omega_{\rho}(G)$ to each $E_{\mathfrak{O}}$ is nonzero.  This implies \eqref{omega:span}.

Let $A_G=Z_G^{\circ}$ be the greatest split torus in the center of $G.$   Consider the action of $A_G \times A_G$ on $E_{\omega_{\rho}}.$  On each $E_{\mathfrak{O}}$ it acts via a character $\chi_{\mathfrak{O}}.$  To prove \eqref{omega:scale} we must show that $\chi_{\mathfrak{O}}$ is nontrivial for all $\mathfrak{O}.$  To prove this it suffices to show that $A_{G\overline{F}}$ acts nontrivially on $V_{\alpha_i},$ where we are using notation introduced above \eqref{omega}.  Equivalently, we must show that $\alpha_i|_{A_{G\overline{F}}}$ is nontrivial.  

We have
$$
\sum_{w \in W(G_{\overline{F}},T_{\overline{F}})(\overline{F})}w.\alpha_i=rd
$$
for some $r \in \ZZ$ because the element on the left is Weyl invariant.  If $r \neq 0$ then we are done.  If $r=0,$ then $-\alpha_i=\sum_{w \neq I}w.\alpha_i.$  This implies that $ \langle-\alpha_i,\omega'\rangle=\langle\sum_{w \neq I}w.\alpha_i,\omega'\rangle \geq 0$ for all $\omega' \in \Omega(\rho).$  Since the $W(G_{\overline{F}},T_{\overline{F}})(\overline{F})$-orbit of $\alpha_i$ lies in $\Omega(\rho)^\vee,$  we deduce that $\langle \alpha_i,\omega' \rangle=0$ for all $\omega \in \Omega(\rho),$ contradicting \eqref{perp0}.  
   
To prove \eqref{HX:H1=0} we point out that in the case at hand $H_X=(G \times G)_{M_{\rho}}=\Delta(Z_G),$ the image of the center $Z_G$ under the diagonal embedding $\Delta:G \to G \times G.$   Indeed, we certainly have $\Delta(Z_G)\leq (G \times G)_{M_{\rho}},$ and it is easy to check that $(G \times G)_G\leq \Delta(Z_G).$  We now apply \eqref{H1:assump}.  Finally,  \eqref{HX:Fpoints} follows from Lemma \ref{lem:M:equiv}.
\end{proof}

We set
\begin{align}
\Psi^{\mathrm{s}}_{\rho}:=\Psi^{\mathrm{s}}_{\omega_{\rho}} \quash{= \textcolor{red}{V_{\rho}^{\vee} \otimes V_{\rho} \rtimes (G\times G)/\Delta(Z(G))}}
\end{align}
and call it the \textbf{small modulation group of $\rho.$}

We originally defined the small modulation group using the unitary representation of $\Psi_{\rho}^{\mathrm{s}}(F)$ on $L^2(M_{\rho}^{\mathrm{sm}}(F),\mathcal{L}^{1/2})).$ This representation is equivalent to the representation
\begin{align*}
\Psi_{\rho}^{\mathrm{s}}(F) \times L^2(G(F)) &\lto L^2(G(F))\\
((\lambda, (g_1,g_2)),f) &\longmapsto \left(x \mapsto \psi(\lambda\circ\omega_\rho(x))f(g_1^{-1}xg_2)\right)
\end{align*}
where we use the Haar measure to define $L^2(G(F)).$  Indeed, $M^{\mathrm{sm}}_{\rho}(F)-G(F)$ is of measure zero, and the passage between the space of half densities and Haar measures is a special case of Remark \ref{rem:half}. We use this observation in the following lemma: 

\begin{lem} The representation of $\Psi^{\mathrm{s}}_{\rho}(F)$ on $L^2(G(F))$ is irreducible. 
\end{lem}
\begin{proof}
The $G \times G$-orbit of the identity $I$ in $G \subset M_{\rho}$ is just $G,$ and it is open and dense in $M_{\rho}.$ Moreover $G(F) \times G(F)$ certainly acts transitively on $G(F),$ and $(G \times G)_I=(G \times G)_{\omega(I)}.$  Thus we can apply Theorem \ref{IrreducibilityRV}.
\end{proof}

\begin{defn} \label{defn:mod:group:monoid} Assume $G$ is split.  The \textbf{modulation group} $\Psi_{\rho}\{F\}$ of $\rho$ over $F$ is the modulation group attached to $\mathcal{F}_{\rho},$ $\omega_{\rho},$ and the trivial character $\chi:G(F) \to \CC^\times:$
\begin{align}
\Psi_{\rho}\{F\}    :=\langle \mathcal{F}_{\rho},\Psi_{\omega_{\rho}}^{\mathrm{s}}(F)\rangle.
\end{align}
\end{defn}
\noindent Here we are using Definition \ref{defn:mod:group}.  We have assumed $G$ is split because it is unclear which character $\chi$ to choose in general.

The following is a summary of some of the results of \cite{DRS:Schwartz}:

\begin{thm} If $X$ is a reductive monoid as above, then there is a Fourier transform $\mathcal{F}_{\rho}$ satisfying \eqref{Fi:equiv} and a Schwartz space satisfying \eqref{assum:preserve}, \eqref{assum:dense}, \eqref{assum:sm}.  \qed
\end{thm}

\section{Horospherical varieties}
\label{sec:HS}

Let $P \leq G$ be a parabolic subgroup of a simple and simply connected group $G.$  
Let 
\begin{align}
    X_P^\circ:=P^{\mathrm{der}} \backslash G, \quad 
    X_P:=\overline{P^{\mathrm{der}} \backslash G}.
\end{align}
These spaces are both examples of horospherical varieties.  
Let $M$ be a Levi subgroup of $P$ and let $M^{\mathrm{ab}}=M/M^{\mathrm{der}}.$  The space $X_P$ admits a canonical right action $X_P \times M^{\mathrm{ab}} \times G \to X_P.$
   
   Fourier transforms and Schwartz spaces for this general class of schemes were first considered in \cite{BK:normalized}.  To simplify our discussion, we will focus on the case where $P$ is a self-associate maximal parabolic subgroup.
\quash{When $P=B$ is a Borel subgroup then we refer to $X_B$ as basic affine space (although this monicker is often reserved for $X_B^\circ$).  It comes equipped with the usual Pl\"ucker embedding
$$
\mathrm{Pl}_{X_B}:X_B \lto \prod_{\omega}V_{\omega}
$$
where the sum is over the fundamental representations of $G$ (viewed as right representations) and $V_{\omega}$ is the space of the fundamental representation. The natural action of $G$ on $\prod_{\omega} V_{\omega}$ extends to an action of $M^{\mathrm{ab}}$ \textcolor{red}{How?}}
We have a closed immersion
\begin{align}
\mathrm{Pl}_{P}:X_P \lto V_P
\end{align}
that is a lift of the usual Pl\"ucker embedding.  We refer to \cite[Lemma 3.4]{Getz:Hsu:Leslie} for details.  Here $V_P$ is a  right representation of $G.$  It is a fundamental representation of highest weight $-\omega_{P}$ with respect to a choice of maximal torus $T \leq M.$   The character $\omega_P$ extends to $M$ and factors through $M^{\mathrm{ab}}.$  The map $\mathrm{Pl}_{P}$ intertwines the action of $M^{\mathrm{ab}}$ on $X_P$ with the action of $M^{\mathrm{ab}}$ on $V_P$ via scaling by the character $\omega_P.$

\begin{lem}
The actions of $M^{\mathrm{ab}} \times G$ on $X_P$ and $V_P$ factor through 
$$
H:=(M^{\mathrm{ab}} \times G)/(M^{\mathrm{ab}} \times G)_{X_P}.
$$
With this choice of $H,$ $X=X_P$ and $\omega=\mathrm{Pl}_P$  the assumptions \eqref{omega:span}, \eqref{omega:scale}, \eqref{HX:H1=0}, and \eqref{HX:Fpoints} are valid.  Moreover $X_P$ is normal and $\mathrm{codim}(X-X^{\mathrm{sm}},X) \geq 2.$
\end{lem}
\begin{proof}
Assumption \eqref{omega:span} follows from the fact that $V_P$ is an irreducible representation of $G.$  Since $\omega_P$ is an immersion, Lemma \ref{lem:M:equiv} (with $H=(M^{\mathrm{ab}} \times G)$) implies $(M^{\mathrm{ab}} \times G)_{X_P}=(M^{\mathrm{ab}} \times G)_{V_P}.$  The assertion on the actions follows.  
Our comments above on the action of $M^{\mathrm{ab}}$ imply \eqref{omega:scale}.
The group $H_X$ is trivial, so it is clear that \eqref{HX:H1=0} and \eqref{HX:Fpoints} hold.

Since $X_P^\circ$ is smooth, connected, and strongly quasi-affine \cite[Theorem 1.1.2]{Braverman:Gaitsgory} $X_P$ is normal \cite[Lemma 3.1]{Getz:Gu:Hsu:Leslie}.  Since $X_P-X_P^\circ$ has dimension $0$ \cite[Theorem 1 and 2]{Popov:Vinberg} the codimension statement is easy to check.
\end{proof}

\begin{thm} If $X$ is a horospherical variety $X_P$ as above and $G$ is a classical group or $G_2$ then Ansatz \ref{Ans:Schwartz} is valid.
\end{thm}

\begin{proof}
For the construction of the Schwartz space and the Fourier transform, see \cite{Getz:Hsu:Leslie} which refines earlier work of \cite{BK:normalized}.  These references use an eigenmeasure to trivialize the space of half-densities as in Remark \ref{rem:half}, so we will omit $\mathcal{L}^{1/2}$ from notation.
It is obvious from the definition that functions in the Schwartz space $\mathcal{S}(X_P(F))$ are smooth. The fact that the Schwartz space is contained in $L^2(X_P^{\circ}(F))$ is \cite[Corollary 5.8]{Getz:Hsu:Leslie}.   See \cite[5.24]{Getz:Hsu:Leslie} for condition \eqref{assum:preserve} and the existence of the Fourier transform.

If $G$ is a classical group or $G_2$ the proof of \eqref{assum:local} and \eqref{assum:sm} (and hence \eqref{assum:dense}) is contained in \cite{Hsu:Asymp}.
\end{proof}

\section{Modulation groups of vector spaces and the Weil representation} \label{sec:Mod:VS} 
Throughout this section, we assume that $F$ is a local field of characteristic zero.  Recall $\psi:F \to \CC^\times$ is a non-trivial character.  Let
\begin{itemize}
    \item $X=V\cong \mathbb{G}_a^n,$
    \item $H=\GL_V$, acting on $V$ on the right.
\end{itemize}
We wish to compute modulation groups in this setting.  This requires a Fourier transform.  We choose a perfect pairing 
\begin{align} \label{perfect:pair}
\langle\,,\,\rangle:V(F) \times V(F) \lto F
\end{align}
and use it to define a Fourier transform
\begin{align}
\mathcal{F}_{V}:=\mathcal{F}_{\langle\,,\,\rangle,\psi}:\mathcal{S}(V(F),\mathcal{L}^{1/2}) \lto \mathcal{S}(V(F),\mathcal{L}^{1/2}).
\end{align}
The notation here is the same as that of \cite[\S 2.7]{Getz:Gu:Hsu:Leslie}, which can be consulted for more details.  The Fourier transform extends an isometry of $L^2(V(F),\mathcal{L}^{1/2}).$

The modulation groups  $\Psi_{\omega}$ computed in this section are both defined using the Fourier transform $\mathcal{F}_{V},$ but with two different choices of equivariant maps $\omega.$

\subsection{The identity map} \label{ssec:ID}
For this subsection we take $\omega=\mathrm{id}:X\to V,$ the identity map.
Since $\GL_V$ acts faithfully on $V$, one sees $H_X=H_V$ is the trivial group. It is straightforward to check the conditions \eqref{omega:span}--\eqref{HX:Fpoints}.  By Definition \ref{def:alg-small-mod}
\begin{align} \label{Psiid}
    \Psi^{s}_{\operatorname{id}} = V^{\vee}\rtimes \GL_V.
\end{align}

To compute the modulation group $\Psi_{\operatorname{id}}\{F\}$, we embed $\Psi^{s}_{\operatorname{id}}$ into a Jacobi group and employ the Heisenberg-Weil representation to describe $\Psi_{\operatorname{id}}\{F\}$ explicitly.

Let $W:=V \oplus V.$
Our choice of perfect pairing \eqref{perfect:pair} induces a symplectic pairing 
\begin{align} \label{symp} \begin{split}
\langle\,,\,\rangle_\wedge:W(F) \times W(F) &\lto F\\
( (v_1,\lambda_1) , (v_2, \lambda_2) ) &\longmapsto \langle v_2,\lambda_1\rangle-\langle v_1,\lambda_2\rangle. \end{split}
\end{align}
Let $\mathrm{Sp}_W$ denote the corresponding symplectic group.  
We will also make use of the Heisenberg group $\operatorname{H}_W = W \ltimes \GG_a.$  We recall that the group law is given on points in an $F$-algebra $R$ by 
\begin{align} \label{H:gp}
    (w_1,t_1).(w_2,t_2) = (w_1 + w_2,\; t_1 + t_2+\tfrac{1}{2}\langle w_1,w_2 \rangle_\wedge), \quad w_1,w_2\in W(R),\; t_1,t_2 \in R.
\end{align}
There is a representation 
$$
\rho_{\psi,W}:\mathrm{H}_W(F) \times L^2(V(F),\mathcal{L}^{1/2}) \lto L^2(V(F),\mathcal{L}^{1/2})
$$
given as follows:
for $f \in L^2(V(F),\mathcal{L}^{1/2})$  and $x \in V(F)$, 
\begin{align} \label{action:HW}
    \rho_{\psi,W}((v,\lambda),t)f(x) &= \psi(\langle x,\lambda \rangle+\tfrac{1}{2}\langle v,\lambda \rangle+t)f(x+v).
\end{align}

Let $\GL_V$ act on $W=V \oplus V$ via its natural action on the first factor and the dual action with respect to $\langle\,,\,\rangle$ on the second factor. This action preserves the symplectic form and induces a closed immersion of groups
\begin{align} \label{m0}
    m: \GL_V \to \Sp_W
\end{align}
with image a Levi subgroup.
For $F$-algebras $R$ 
let
\begin{align} \label{sympair}
\mathrm{Sym}_{\langle\,,\,\rangle}(R):=\left\{b \in \mathrm{End}_{R}(V(R)): \langle b v,w\rangle=\langle v,bw\rangle \textrm{ for all }v,w \in V(R)\right\}.
\end{align}
We then have a closed immersion of groups
\begin{align} \label{n} \begin{split}
n:\mathrm{Sym}_{\langle\,,\,\rangle}(R) &\lto \Sp_W(R)\\
b &\longmapsto \left((v,\lambda) \mapsto (v, \lambda + bv)\right).\end{split}
\end{align}
Lastly, there exists a unique $w \in \Sp_W(F)$ that acts on $W$ by 
\begin{equation}\label{w:def}
    (v,\lambda)w = (-\lambda, v).
\end{equation}
Since $w$ represents the longest Weyl element, the collection of elements $m(g)$, $n(b)$, and $w$ generate $\Sp_W(F)$ by the Bruhat decomposition with respect to Siegel parabolic 
\begin{align} \label{PW}
P_W = m(\GL_V)n( \operatorname{Sym}_{\langle\,,\,\rangle}(V)).
\end{align}
We let $\widetilde{\GL}_V$ be the algebraic subgroup of $\mathrm{Sp}_{V}$ such that 
\begin{align} \label{widetilde:GLv}
\widetilde{\GL}_V(k)=\langle w,\GL_V(k)\rangle
\end{align}
for all fields $k/F.$

The action of $\GL_V$ on $W$ extends to an action on $\mathrm{H}_W.$ Consequently, we have the following closed immersion of group schemes:
\begin{align}\label{Vec-Weil: embedding}
  \mathrm{H}_W\rtimes\GL_V \lto \mathrm{H}_W \rtimes \Sp_W =: \mathrm{J}_W.
\end{align}
The group $\mathrm{J}_W$ is known as Jacobi group associated with the symplectic space $W.$ The representation $\rho_{\psi,W}$ does not in general extend to $\mathrm{J}_W(F),$ but it does lift to a metaplectic cover as we now recall.  It is convenient to use \cite{GKTtheta} as a reference, but we warn the reader that 
in loc.~cit.~only the non-Archimedean case is treated.  The statements we make will be valid for general local fields.    

Following \cite[\S 9.2.2]{GKTtheta}, let $\Mp_W(F) := \Mp^{(8)}_{W,\psi,V^{\vee}}(F)$ be the metaplectic 8-fold cover of $\Sp_W(F)$ attached to $W$ and the additive character $\psi$. In other words we take $Z=\mu_8$ in the notation of loc.~cit.  Let $(\omega_{\psi,W}, L^2(V(F),\mathcal{L}^{1/2}))$ be the Schr{\"{o}}dinger model of the Weil representation of $\Mp_{W}(F)$. By \cite[Theorem 9.9]{GKTtheta}, we can explicitly describe the action on a set of generators as follows:
\begin{align} \label{Weil}\begin{split}
    \omega_{\psi,W}(m(g),1)f(x) &=  f(xg),\\
    \omega_{\psi,W}(n(b),1)f(x) &= \overline{\psi}(\tfrac{1}{2}\langle bx,x\rangle)f(x),\\
    \omega_{\psi, W}(w,1)f(x) &=  \mathcal{F}_{V,\psi}(f)(x),\\
    \omega_{\psi,W}(I,z)f(x)&=zf(x). \end{split}
\end{align}
We refer to loc.~cit.~for the (standard) unexplained notation. We point out that the usual factor of $|\det g|^{1/2}$ does not appear because we are working with half densities. \quash{\textcolor{red}{Why this is the case? the sections $f(xg)$ and $|\det g|^{1/2}f(xg)$ are still different.}}  Moreover, we are using the fact that the cover $\Mp_W(F) \to \mathrm{Sp}_{W}(F)$ splits over $P(F),$ and over the subgroup generated by $w$ \cite[Corollary 2.21, Corollary 2.27]{GKTtheta}.

Let $\widetilde{\mathrm{J}}_W(F) := \mathrm{H}_W(F) \rtimes \Mp_W(F)$.  Then we have the Heisenberg-Weil representation 
\begin{align}\label{Def: Heisenberg-Weil}
    \omega^{\mathrm{J}}_{\psi} := \rho_{\psi,W} \otimes \omega_{\psi,W} : \widetilde{\mathrm{J}}_W(F) \lto \operatorname{Aut}(L^2(V(F),\mathcal{L}^{1/2})).
\end{align}
We point out that in general the map $\omega^{\mathrm{J}}_{\psi}$ is not injective.  Indeed, when $F \neq \RR$ even its restriction to the center of $\mathrm{H}_W(F)$ is not injective because $\psi$ has non-trivial kernel.

\begin{thm} \label{thm:vector-Mod-group}
The modulation group $\Psi_{\operatorname{id}}\{ F \}$ coincides with the image of 
$$ \mathrm{H}_W(F) \rtimes \langle w,\GL_V(F)\rangle <  \widetilde{\mathrm{J}}_W(F)$$
under $\omega_{\psi}^J.$
\end{thm}
\begin{proof}
We have an injective group homomorphism 
\begin{align*}
\Psi_{\mathrm{id}}^{\mathrm{s}}(F) &\lto \mathrm{H}_W(F) \rtimes \GL_V(F)\\
(\lambda, h) &\longmapsto (((0,\lambda),0) , m(h))
\end{align*}
that intertwines the actions of $\Psi_{\mathrm{id}}^{\mathrm{s}}(F)$ and $\mathrm{H}_W(F)$ on $L^2(V(F),\mathcal{L}^{1/2}).$

The formulae \eqref{Weil} imply that
\begin{align} \label{ineq} \begin{split}
        \Psi_{\operatorname{id}}\{F\} := \langle \mathcal{F}_V,V^{\vee}(F)\rtimes\GL_V(F)\rangle &= \omega^{\mathrm{J}}_{\psi}\big(\langle w , \left(\{0\} \times V (F)\right)\rtimes\GL_V(F)\rangle\big)\\
        &\leq \omega^{\mathrm{J}}_{\psi}\big(\mathrm{H}_W(F) \rtimes \langle w,\GL_V(F)\rangle\big). \end{split}
\end{align}
On the other hand, $W(F) <\langle w , \left(\{0\} \times V(F)\right)\rtimes\GL_V(F)\rangle,$ and
for any $\lambda,v \in V(F)$ 
$$
   \big( ((v,0),0).((0,\lambda),0)\big).\big(((-v,0),0).((0,-\lambda),0)\big) = ((0,0),-\langle v, \lambda\rangle).
$$
This implies that $\langle w , \left(\{0\} \times V (F)\right)\rtimes\GL_V(F)\rangle=\mathrm{H}_W(F)\rtimes \langle w,\GL_V(F)\rangle $ and hence \eqref{ineq} is an equality.
\end{proof}

\subsection{Symmetric squares of vector spaces}
Let $V_0 \cong \GG_a^n$ and let $\langle\,,\,\rangle:V_0(F) \times V_0(F) \to F$ be a perfect pairing.
We take 
\begin{align}
    X: = V_0 \quad\textrm{and}\quad H = \GL_{V_0}
\end{align}
acting on $X$ on the right.  We then take $V=\mathrm{Sym}_{\langle\,,\,\rangle}^{\vee}$ and let $\omega$ be the map given on points in an $F$-algebra $R$ by 
 \begin{align} \label{omega:2} \begin{split}
 \omega :X(R) &\lto \mathrm{Sym}_{\langle\,,\,\rangle}^\vee(R)\\
 v &\longmapsto \left( b\mapsto \tfrac{1}{2}\langle -b(v),v \rangle\right).\end{split}
 \end{align}
The group $H_X$ is trivial and verifying the conditions \eqref{omega:span}--\eqref{HX:Fpoints} is straightforward. In this setting
\begin{align}\label{sym2mod}
     \Psi^{\mathrm{s}}_{\omega}= \operatorname{Sym}_{\langle\,,\,\rangle} \rtimes \GL_{V_0}.
\end{align}

\begin{prop} \label{prop:gen1}
Let $G$ be a reductive group over a field $k,$ let $P \leq G$ be a minimal parabolic subgroup, let $M \leq P$ be a Levi subgroup, let $N_P \leq P$ be the unipotent radical, and let $N_{P^{\mathrm{op}}}$ be the unipotent radical of the parabolic subgroup $P^{\mathrm{op}}$ opposite $P$ with respect to $M.$  Then 
$$
G(k)=N_P(k)N_{P^{\mathrm{op}}}(k)P(k).
$$
\end{prop}
\begin{proof} See the comment after \cite[Corollaire 6.26]{Borel:Tits}.
\end{proof}

Let $W_0 = V_0\oplus V_0$ equipped with the symplectic form $\langle\,,\,\rangle_{\wedge}$ attached to $\langle\,,\,\rangle$ as in \eqref{symp}.  We then have a symplectic group $\mathrm{Sp}_{W_0}$ etc. as in \S \ref{ssec:ID}.
Let
\begin{align}
\mathrm{Mp}_{W_0}(F)':=\langle w,P_{W_0}(F)\rangle
\end{align}
be the subgroup generated by $w$ and $P_{W_0}(F).$  It follows from Proposition \ref{prop:gen1} that the natural map $\mathrm{Mp}_{W_0}(F)' \to \mathrm{Sp}_{W_0}(F)$ is surjective.
 
\begin{thm} \label{thm:mod:group:Weil}
One has that $\Psi_{\operatorname{Sym^2}}\{F\} = \omega_{\psi, W_0}(\Mp_{W_0}(F)').$
\end{thm}
\begin{proof}
Using notation from \eqref{m0}, \eqref{n} and \eqref{PW} we have an isomorphism 
\begin{align*}
\Psi_{\omega}^{\mathrm{s}}(R) &\tilde{\lto} P_{W_0}(R)\\
(b,  g) &\longmapsto n(b)m(g).
\end{align*}
The isomorphism intertwines the actions of $\Psi_{\omega}^{\mathrm{s}}(F)$ and $P_{W_0}(F)$ on $L^2(V_0(F),\mathcal{L}^{1/2}).$  Using 
\eqref{Weil} we see that 
    $$
        \Psi_{\operatorname{Sym^2}}\{F\} = \langle \mathcal{F}_V, \omega_{\psi, W_0}(P_{W_0}(F))\rangle = \omega_{\psi, W_0}(\langle w,P_{W_0}(F) \rangle).
    $$  
\end{proof}

\subsection{The standard representation} \label{sec:standard}
In this section we let $X=V=M_n$ and let
$$
\omega:X \lto V
$$
be the identity map.  Here we equip $X$ and $V$ with the usual action of $H=\GL_n \times \GL_n.$
Then 
$H_X \cong \GG_m$, 
\begin{align} \label{H:quot}
H/H_X = \mathrm{P}(\GL_{n}^2):= (\GL_{n} \times \GL_{n})/\operatorname{diag}(\GG_m),
\end{align}
and conditions \eqref{omega:span}--\eqref{HX:Fpoints} are satisfied.  By \eqref{H:quot} we have
\begin{align}
    \Psi_{\operatorname{id}}^{\mathrm{s}} = V^{\vee} \rtimes \mathrm{P}(\GL_{n}^2).
\end{align}

Let $W = V\oplus V^{\vee}$ and define $\mathrm{H}_W,$ etc. as in \S \ref{ssec:ID}.  The action of $\GL_{n} \times \GL_{n}$ on $X=V$ induces a closed immersion 
$$
\mathrm{P}(\GL_{n}^2) \lto \GL_{V} \stackrel{m}{\lto} P_{W}
$$
and we identify $\mathrm{P}(\GL_{n}^2)$ with its image.
Arguing as in the the proof of Theorem \ref{thm:vector-Mod-group}, we obtain the following theorem:
\begin{thm} \label{thm:Stand} One has that $\Psi_{\operatorname{id}}\{F\} =\omega^{\mathrm{J}}_{\psi} \big(\mathrm{H}_W(F) \rtimes \langle w, \mathrm{P}(\GL_{n}^2)(F) \rangle \big).$ \qed
\end{thm}

\subsection{Smooth vectors}

\begin{prop} \label{prop:smooth1} In the settings of Theorems \ref{thm:vector-Mod-group} and  \ref{thm:Stand},  the Schwartz space is the space of smooth vectors under the action of the modulation group.  
\end{prop}

\begin{proof}
Assume first that $V$ is as in Theorem \ref{thm:vector-Mod-group}.  We have that
\begin{align} \label{first:ineq}
\mathcal{S}(V(F)) \leq L^2(V(F))^{\mathrm{sm}}.
\end{align}
To prove equality, we recall that the Schwartz space
$\mathcal{S}(V(F))$ is the space of smooth vectors in $L^2(V(F))$ under the action of $H_W(F).$ Indeed, when $F=\RR$ this is \cite[p. 827]{Howe:Heis}; the argument generalizes in a standard manner to the case $F=\CC.$
  When $F$ is non-Archimedean it is immediate.  This completes the proof in the setting of Theorem \ref{thm:vector-Mod-group}.

The proof in the setting of Theorem \ref{thm:Stand} is essentially the same.
\end{proof}

\section{Modulation groups of quadric cones}
\label{sec:Mod:cones}
For the moment we take $F$ to be an arbitrary characteristic zero field so that we can develop notation that will also appear in other sections.  Let $V_{0}=\GG_a^{2j}$ for some $j \in \ZZ_{\geq 0}$ and let $\mathcal{Q}_0:V_0 \to \GG_a$ be an anisotropic form.  Let $V_k=\GG_a^{2k+2j}$ be a space with a nondegenerate quadratic form $\mathcal{Q}_k$ in the same Witt class as $\mathcal{Q}_0.$  We choose bases so that the matrix of the quadratic form $\mathcal{Q}_{k+1}$ is given inductively by
\begin{align}
   \begin{psmatrix} & & 1\\ & J_k & \\ 1 & & \end{psmatrix} 
\end{align}
where $J_k$ is the matrix of $\mathcal{Q}_k.$  Let $\langle\,,\,\rangle_k$ be the pairing defined by $\mathcal{Q}_k.$
  We let $\mathrm{O}_{V_k}$ and $\mathrm{GO}_{V_k}$ be the orthogonal group and orthogonal similitude group of $V_k$
  and let $\mathrm{GSO}_{V_k}$ be the neutral component of $\mathrm{GO}_{V_k}.$  We denote by 
  $$
  \nu:\mathrm{GO}_{V_k} \lto \GG_m
  $$
  the similitude norm.
Let 
$$
C_{k} \subset V_k
$$
be the vanishing locus of $\mathcal{Q}_k$ and let $C_k^\circ:=C_k-\{0\}.$ 
For $F$-algebras $R$ we have a group homomorphism
\begin{align} \label{m} \begin{split}
m:R^\times \times \mathrm{GO}_{V_k}(R)&\lto \mathrm{GO}_{V_{k+1}}(R)\\
(a,g) &\longmapsto \begin{psmatrix} a\nu(g) & & \\ & g & \\ & & a^{-1}\end{psmatrix}. \end{split}
\end{align}
Let 
\begin{align} \label{opn} \begin{split}
\overline{n}:V_k(R) &\lto \mathrm{O}_{V_{k+1}}(R)\\
v &\longmapsto \begin{psmatrix} 1 & & \\ -J_kv^t & I_k & \\ -\mathcal{Q}_k(v) & v & 1 \end{psmatrix},
\\
n:V_k(R) &\lto \mathrm{O}_{V_{k+1}}(R)\\
v &\longmapsto \begin{psmatrix} 1 & v & -\mathcal{Q}_k(v) \\  & I_k & -J_kv^t \\  &  & 1 \end{psmatrix}.
\end{split}
\end{align}
We let 
\begin{align} \begin{split}
Q_k(R):&=\left\{m(a,g)n(v):(a,g,v) \in R^\times \times \mathrm{O}_{V_k}(R) \times V_k(R)\right\},\\
Q_k^{\mathrm{op}}(R):&=\{m(a,g)\overline{n}(v):(a,g,v) \in R^\times \times \mathrm{O}_{V_k}(R) \times V_k(R)\}. \end{split}
\end{align}
These are parabolic subgroups of 
$\mathrm{O}_{V_{k+1}}$ opposite each other with respect to $m(\GG_m \times \mathrm{O}_{V_k}).$ We let $N_k$ (resp.~$N_k^{\mathrm{op}}$) be the unipotent radical of $Q_k$ (resp.~$Q_k^{\mathrm{op}}$).  Thus $N_k$ is the image of $n$ and $N_k^{\mathrm{op}}$ is the image of $\overline{n}.$ Similarly, we let
\begin{align}
    \widetilde{Q}_k(R):=\{m(a,g)n(v):(a,g,v) \in R^\times \times \mathrm{GO}_{V_k}(R) \times V_k(R)\}
\end{align}
which is a parabolic subgroup of $\mathrm{GO}_{V_{k+1}},$
and let $\widetilde{Q}^{\mathrm{op}}_k$ be its opposite parabolic.  Usually in the theory described below one restricts attention to orthogonal groups, but in order to construct a group action satisfying the assumptions in \S \ref{ssec:alg:small} and to make contact with reductive monoids in \S \ref{ssec:RSM} we extend the theory to similitudes.

Let
\begin{align}
\mathbb{O}_{k+1} \subset \mathfrak{o}_{V_{k+1}}
\end{align}
be the minimal nilpotent orbit.
By restricting the adjoint action, we obtain an action of $Q_{k+1}^{\mathrm{op}}$ on $\mathbb{O}_{k+1}.$

The following is proved in \cite{Tome} using the description of the minimal nilpotent orbit given in \cite{jia2024affineclosuretslnu}.

\begin{lem} \label{lem:orbit:dim} For $\dim V_{k} > 2$ there is a unique open $Q_{k+1}^{\mathrm{op}}$-orbit in $\mathbb{O}_{k+1}.$ Its complement has codimension at least $2.$ \qed
\end{lem}

One has a canonical action of $\mathrm{O}_{V_k}$ on $T^*C_k^\circ$ as in Lemma \ref{lem:cot:action}; it extends to the affine closure $\overline{T^*C_k^\circ}^{\mathrm{aff}}.$
In Proposition \ref{prop:cone:Ansatz} we require the following proposition:

\begin{prop}\label{minNilpOrbitIso} For $\dim V_k > 2$ the affine closure $\overline{T^*C_k^{\circ}}^{\mathrm{aff}}$
is isomorphic as a $\mathrm{O}_{V_k}$-scheme to the closure of  $\mathbb{O}_{k+1}$ in $\mathfrak{o}_{V_{k+1}}$.  
\end{prop}

\noindent This is part of the folklore.  We give a proof because the morphism will be helpful to us later in \S \ref{sssec:quadric}.

\begin{proof}
For $F$-algebras $R$ let $R[\epsilon]$ denote the ring of dual numbers. We have that
\begin{align} \label{id} \begin{split}
    TC_k^\circ(R)&=C_k^\circ(R[\epsilon])\\
    &=\{(c,v) \in C_k^\circ(R) \times V_k(R):\langle c,v\rangle_k=0\}. \end{split}
\end{align}
Use $\langle\,,\,\rangle_k$ to identify $V_k$ with $V_k^\vee.$  Then \eqref{id} implies $T^*C_k^\circ$ is the quotient of $C_k^\circ \times V_k$ by the action of $\GG_a$ given by
\begin{align} \label{Ga:act} \begin{split}
    C_k^\circ(R) \times V_k(R) \times R &\lto C_k^\circ(R) \times V_k(R) \\
    ((c,v),a) &\longmapsto (c,v+ac). \end{split}
\end{align}

Let $\mathbb{O}_{k+1} \subset \mathfrak{o}_{V_{k+1}}$ be the minimal nilpotent orbit.  
We have a map
\begin{align} \label{action:map} \begin{split}
   a: C^\circ_k(R) \times V_k(R) &\lto \mathbb{O}_{k+1}(R)\\
    (c,v) &\longmapsto  \overline{n}(v)\begin{psmatrix}
    0 & c & 0\\
          0 & 0   &  -J_kc^t\\
          0 & 0   & 0
\end{psmatrix}\overline{n}(v)^{-1}. \end{split}
\end{align}
Here we are using notation as in \eqref{opn}. Thus 
\begin{align}\label{invariantscone}
a(c,v)=\begin{psmatrix}
    \langle c,v \rangle_k & c & 0 \\J_kc^t\mathcal{Q}_k(v)-\langle c,v \rangle_k J_kv^t & \left((-J_kv^t)_ic_j+(J_kc^t)_iv_j  \right)_{ij} & -J_kc^t\\
    0 & -\mathcal{Q}_k(v)c+\langle v,c\rangle_k v & -\langle v,c \rangle_k
\end{psmatrix}.
\end{align}

One has $a(c,v+tc)=a(c,v)$ for all $t \in R.$  Hence $a$ descends to a morphism
$\overline{a}:T^*C_k^{\circ} \to \mathbb{O}_{k+1}.$  
We claim that this map is injective on $\overline{F}$-points.  Indeed, assume $a(c,v)=a(c',v')$ for $(c,v),(c',v') \in T^*C_k^\circ(\overline{F}).$ Then it is clear that $c=c',$ and then $a(c,v-v')=a(c,0).$  One checks directly that this implies $v-v'=tc$ for some $t \in \overline{F}.$

One has that $\dim \mathbb{O}_{k+1}=2\dim V_k-2.$  Thus by comparing dimensions the image of $a$ is the open $Q_{k+1}^{\mathrm{op}}$-orbit in $\mathbb{O}_{k+1}$ mentioned in Lemma \ref{lem:orbit:dim}.
It is the underlying topological space of a reduced open subscheme $\mathbb{O}_{k+1}' \subset \mathbb{O}_{k+1}.$
Since $T^*C_k^\circ$ and $\mathbb{O}_{k+1}$ are both smooth they are in particular integral and normal.  
The map
$\overline{a}:T^*C_k^\circ(\overline{F}) \to \mathbb{O}_{k+1}'(\overline{F})$ is bijective, and hence the map $\overline{a}: T^*C_k^\circ \to \mathbb{O}_{k+1}$ is an isomorphism onto its image by a standard argument recorded in \cite[Proposition 1.2.9]{Getz:Hahn}; the proof therein does not use the assumption that the schemes are affine.

Now the complement of $\mathbb{O}_{k+1}'$ in $\mathbb{O}_{k+1}$ is of codimension at least $2.$  Moreover $\overline{\mathbb{O}}_{k+1}$ is known to be normal and equal to $\mathbb{O}_{k+1} \sqcup \{0\}$ \cite[\S 8.6]{Jantzen:Nilpotent}.
  It follows that the closure $\overline{\mathbb{O}'}_{k+1}=\overline{\mathbb{O}}_{k+1}$ of $\mathbb{O}'_{k+1}$ in $\mathfrak{o}_{V_{k+1}}$ may be identified with the affine closure of $\mathbb{O}'_{k+1}$ \cite[Theorem 6.45]{Gortz_Wedhorn} and hence the isomorphism
$\overline{a}:T^*C_k^\circ \tilde{\to} \mathbb{O}'_{k+1}$ extends to an isomorphism $\overline{a}:\overline{T^*C_k^\circ}^{\mathrm{aff}} \tilde{\to} \overline{\mathbb{O}}_{k+1}.$
\end{proof}

\subsection{The small modulation group for quadric cones} 
We now assume $F$ is a local field of characteristic $0$.  We take $X=C_{k}$, $V=V_{k},$ and let
$$
\omega:C_{k} \lto V_{k}
$$
be the inclusion.  We take  $H=\GG_m \times \mathrm{GO}_{V_k}$
where the action is given by 
\begin{align*}
V_k(R) \times H(R) &\lto V_k(R)\\
(v,(a,h)) &\longmapsto avh.
\end{align*}
Then one has that
$$
H_{X}(R):=\{(\lambda^{-1},\lambda I_{V_k}):\lambda \in R^\times\} \cong R^\times.
$$
The hypothesis \eqref{omega:span}--\eqref{HX:Fpoints} are valid.  
\begin{rem}
The space $C_k$ is a horospherical variety in the sense of \S \ref{sec:HS}, but for convenience we have not taken the group $H$ in the current section to be the same as the group $H$ in \S \ref{sec:HS}.  The two actions encode essentially the same information.
\end{rem}

Using the pairing $\langle\,,\,\rangle_k$ we identify $V$ with its dual.  Then according to Definition \ref{def:alg-small-mod}, we have that
\[
\Psi^{\mathrm{s}}_{\omega} = V \rtimes (\GG_m \times \GO_{V_k})/H_X.
\]
Let $\omega_{C_k}$ be the top-degree differential form on $C_k^\circ$ such that 
\begin{equation}\label{DefinitionDifferentialForm}
d v_1 \wedge\cdots \wedge d v_{\mathrm{dim}(V_k)} = d(\mathcal{Q}_k(v)) \wedge \omega_{C_k}(v).
\end{equation}
The corresponding density $|\omega_{C_k}|$ on $C_k^\circ(F)$ defines an eigenmeasure under the action of $H(F).$  Explicitly, one has that
\begin{align} \label{omega:translation}
|\omega_{C_k}|(avh)=|a|^{\dim V_k-2}|\nu(h)|^{(\dim V_k-2)/2}|\omega_{C_k}|(v).
\end{align}
As in Remark \ref{rem:half} we therefore have an isomorphism
\begin{align} \begin{split}
    L^2(C_k^\circ(F)) &\tilde{\lto}L^2(C_k^\circ(F),\mathcal{L}^{1/2})\\
    f &\longmapsto f|\omega_{C_k}|^{1/2} \end{split}
\end{align}
where on the left hand side the Hilbert space is defined using the measure $|\omega_{C_k}|.$  The action of $\Psi_{\omega}^{s}(F)$ on the right hand side is intertwined with the action on the left hand side given by 
\begin{align}\label{def:Quadric:Rep}
    \mathcal{R}_{\omega,\psi}(v \rtimes m(a,h))f(c) = \psi(\langle v, c \rangle)|a|^{(\dim V_k-2)/2}\left|\nu(h)\right|^{(\dim V_k-2)/4} f(ach) &\textrm{ for }(a,h)\in F^\times \times\mathrm{GO}_{V_k}(F).
\end{align} 
To ease comparison with \cite{GurK:Cone} we work with $L^2(C_k^{\circ}(F))$ instead of $L^2(C_k^\circ(F),\mathcal{L}^{1/2}).$

\subsection{Modulation groups for quadric cones}\label{Quadric:Cones:Section}

For the remainder of this section we assume the following:
\begin{enumerate}
\item When $F$ is non-Archimedean we assume $\dim V_k \geq 6$ and that the Witt index of $V_{k+1}$ is at least $\dim V_k/2.$

\item When $F$ is Archimedean we assume it is real. 
\end{enumerate}
The first assumption is made so that we can apply the results of \cite{GurK:Cone}.  The second is due to the fact that we were unable to locate references for the minimal representation of the orthogonal group over the complex numbers.

When $F$ is non-Archimedean we let $\chi$ be the character associated by class field theory to $F(\sqrt{\det J_k})/F.$  When $F$ is real and $\mathcal{Q}_{k}$ has signature $(p,q)$ we let $\chi:F^\times \to \CC^\times$  be the character that is trivial on $\RR_{>0}$ and assigns $(-1)^{(p-q)/2}$ to $\RR_{<0}.$

\subsubsection{The minimal representation of $\mathrm{O}_{V_{k+1}}(F)$} \label{ssec:min}

We use the isomorphism $\overline{n}:V_{k} \to N_k^{\mathrm{op}}$ of \eqref{opn} to identify $V_k$ with $N_k^{\mathrm{op}}.$
For $(v,a,h) \in V_k(F) \times F^\times \times \mathrm{GO}_{V_k}(F)$ one has that 
\begin{align} \label{conj}
m(a,h)^{-1}\overline{n}(v)m(a,h) = \overline{n}(avh).
\end{align}
\quash{Use the exponential map to identify $N_k(F) \times N_k^{\mathrm{op}}(F)$ with its Lie algebra; it is a subalgebra of $\mathfrak{o}_{V_{k+1}}.$  Restricting the Killing form $\langle X,Y\rangle:=(\dim V_{k+1}-2)\mathrm{tr}\,XY$ to $N_k(F)\times N_k^{\mathrm{op}}(F)$ we see that 
\begin{align} \label{2:pair}
    \langle n(v),\overline{n}(v')\rangle=-2(\dim V_{k+1}-2)\langle v,v'\rangle_k.
\end{align} }
Let $\widetilde{\tau}$ be the representation of $\widetilde{Q}_k(F)$ on the space of smooth functions $C^{\infty}(C_k^\circ(F))$ defined by
\begin{equation}
    \label{eqn:GK-klingen-action}
     \widetilde{\tau}( n(v)m(a,g))f(c) = \psi(\langle v,c\rangle_k)\chi(a)|a|^{(\dim V_k-2)/2}|\nu(g)|^{(\dim V_k-2)/4} f(acg )
\end{equation}
for $(v,a,g,c) \in V_k(F) \times F^\times \times \mathrm{GO}_{V_k}(F) \times C_k^\circ(F).$

Set $\tau=\widetilde{\tau}|_{Q_k(F)}.$  Then $\tau$ extends to a unitary representation
$$
\tau:\mathrm{O}_{V_{k+1}}(F) \times L^2(C_k^\circ(F)) \lto L^2(C_k^{\circ}(F)),
$$
namely, the minimal representation (see \cite{Kobayashi:Mano,GurK:Cone}).  We point out that in \cite{GurK:Cone} the action of $Q_{k}(F)$ is given in terms of the Killing form instead of $\langle\,,\,\rangle_k,$ but the two forms agree up to a nonzero constant.  The constant is unimportant for the purposes of this paper.

Identify $\mathrm{O}_{V_{k+1}}(F)$ with a subgroup of $\mathrm{GO}_{V_{k+1}}(F)$ in the evident manner.  
Let
\begin{align} \label{w0}
w_0= \left(\begin{smallmatrix}
        & & 1 \\ & I_{V_{k}} & \\ 1 & & 
    \end{smallmatrix}\right)
\in \mathrm{O}_{V_{k+1}}(F).
\end{align}

\begin{lem}
The representation $\widetilde{\tau}$ extends to a unitary representation of $\mathrm{GO}_{V_{k+1}}(F)$ on $L^2(C_k^{\circ}(F)),$ still denoted $\tau,$ such that $\widetilde{\tau}|_{\mathrm{O}_{V_{k+1}}(F)}=\tau.$  
\end{lem}
\begin{proof}
Any element of $\mathrm{GO}_{V_{k+1}}(F)$ is of the form $m(a,g)h$ where $(a,g,h) \in F^\times \times \mathrm{GO}_{V_{k}}(F) \times \mathrm{O}_{V_{k+1}}(F).$  We claim that the desired extension is given by 
$$
\widetilde{\tau}(m(a,g)h):=\widetilde{\tau}(m(a,g))\tau(h).
$$
To prove that this is an extension of our original representation it suffices to show that 
\begin{align} \label{compat2}
 \tau(h)\widetilde{\tau}(m(a,g))=\widetilde{\tau}(m(a,g))\tau(m(a,g)^{-1}hm(a,g)).
\end{align}
It suffices to check \eqref{compat2} for a set of $h$ generating $\mathrm{O}_{V_{k+1}}(F)$
 as a group.  
 Thus by Proposition \ref{prop:gen1} and the Bruhat decomposition it suffices to check \eqref{compat2} if $h \in Q_{k+1}(F)$ or if $h=w_0.$  If $h \in Q_{k+1}(F)$ then \eqref{compat2} is valid because $\widetilde{\tau}|_{Q_{k+1}(F)}=\tau|_{Q_{k+1}(F)}.$  
 
 To check \eqref{compat2} when $h=w_0$ we use the formula
 \begin{align} \label{tau}
     \tau(w_0)(f)(c)=\int_{F}\Psi(t)\mathcal{R}_t(f)(c)dt.
 \end{align}
 Here $\Psi(t)$ is a certain distribution, 
 \begin{align}
 \mathcal{R}_{t}(f)(c):=\int_{x \in C_{k}^\circ(F):\langle c,x\rangle_k=t}f(x)|\omega_{c,t}|(x)
 \end{align}
 is the Radon transform \cite[(1.5)]{GurK:Cone} \cite[(5.2.2)]{Kobayashi:Mano}, and $|\omega_{c,t}|$ is a suitable family of measures
  (see \cite[\S 3.2]{GurK:Cone}).  We point out that in \cite{Kobayashi:Mano}
a different choice of $w_0$ and pairing is used, but \eqref{tau} is still valid.

Following the proof of \cite[Proposition 3.16]{GurK:Cone}, we then have that
\begin{align*}
&\tau(w_0)\widetilde{\tau}(m(a,g))f(c)\\&=
    \chi(a)|a|^{(\dim V_k-2)/2}|\nu(g)|^{(\dim V_k-2)/4}\int_{F}\Psi(t)\int_{x \in C_{k}^\circ(F):\langle c,x\rangle_k=t}f(axg)|\omega_{c,t}|(x) dt
    \\&=\chi(a)|a|^{(2-\dim V_k)/2}|\nu(g)|^{(\dim V_k-2)/4}\int_{F}\Psi(t)\int_{x \in C_{k}^\circ(F):\langle c,x\rangle_k=t}f(axg)|\omega_{a^{-1}\nu(g)^{-1}cg,t}|(axg) dt\\
    &=\chi(a)|a|^{(2-\dim V_k)/2}|\nu(g)|^{(2-\dim V_k)/4}\int_F \Psi(t)\int_{x \in C_{k}^\circ(F):\langle c,axg^{-1}\rangle_k=t}f(x)|\omega_{a^{-1}\nu(g)^{-1}cg,t}|(x)dt\\
    &=\chi(a)|a|^{(2-\dim V_k)/2}|\nu(g)|^{(2-\dim V_k)/4}\int_{F}\Psi(t)\int_{x \in C_{k}^\circ(F):\langle a^{-1}\nu(g)^{-1}cg,x\rangle_k=t}f(x)|\omega_{a^{-1}\nu(g)^{-1}cg,t}|(x)dt.
\end{align*}
Since $m(a,g)^{-1}w_0m(a,g)= \begin{psmatrix} a^{-2}\nu(g)^{-1} & & \\ & I_{2k} & \\ & & a^2\nu(g) \end{psmatrix}w_0$
this is 
$\widetilde{\tau}(m(a,g))\tau(m(a,g)^{-1}w_0m(a,g))f(c).$ This completes the proof of \eqref{compat2} when $h=w_0$ and hence the proof of the lemma.
\end{proof}

\subsubsection{The modulation group}
\label{ssec:mod:grp}
We define
\begin{align} \label{cone:Sch}
    \mathcal{S}(C_k(F))=L^2(C_k^\circ(F))^{\mathrm{sm}}
\end{align}
where we take smooth vectors with respect to the action of $\mathrm{GO}_{V_{k+1}}(F)$ via $\widetilde{\tau},$ or equivalently with respect to $\mathrm{O}_{V_{k+1}}(F)$ via $\tau.$

We define
\begin{align}\label{QuadricFouriertrans}
\mathcal{F}_{C}:=\widetilde{\tau}(w_0)
\end{align}
with $w_0$ as in \eqref{w0}.
Thus we have an isomorphism 
$$
\mathcal{F}_{C}:\mathcal{S}(C_k(F)) \lto \mathcal{S}(C_k(F)).
$$
It is reasonable to view this as a Fourier transform.  Indeed, there is a formula for it entirely analogous to the usual Fourier transform; see \cite[1.5]{GurK:Cone} or \cite[Corollary 6.9]{Getz:Hsu:Leslie}.

Consider the action  $\mathcal{R}_{\omega,\chi}:\Psi^{\mathrm{s}}_{\omega}(F) \times L^2(C_k^\circ(F)) \to L^2(C_k^\circ(F))$ given by \begin{align} \label{twisted:action}
    \mathcal{R}_{\omega,\chi}((v, m(a,h)))f(c)=\psi(\langle v, c \rangle)\chi(a)|a|^{(\dim V_k-2)/2}\left|\nu(h)\right|^{(\dim V_k-2)/4} f(ach) 
\end{align}
for $(a,h)\in F^\times \times\mathrm{GO}_{V_k}(F).$ This is a twist of the action \eqref{def:Quadric:Rep}.
The \textbf{modulation group} of the inclusion $\omega:C_k \to V_k$ is the group 
\begin{align} \label{modif:mod:group}
\Psi_{\omega}\{F\}:=\langle\mathcal{F}_{C}, \mathcal{R}_{\omega,\chi}(\Psi_{\omega}^{\mathrm{s}}(F))\rangle.
\end{align}  

\begin{thm} \label{thm:mod:cone}
One has that $
\Psi_{\omega}\{F\}=\widetilde{\tau}(\mathrm{GO}_{V_{k+1}}(F)).
$  
\end{thm}

\begin{proof}
We have an isomorphism
\begin{align*}
    \Psi_{\omega}^{\mathrm{s}}(F) &\tilde{\lto} (F^\times \times \widetilde{Q}_{k}(F))/H_X(F)\\
    (v, (a,g))& \longmapsto n(v)m(a,g)
\end{align*}
that intertwines the action \eqref{twisted:action} and $\widetilde{\tau}$ by \eqref{eqn:GK-klingen-action}.  Thus it suffices to check that $\mathrm{GO}_{V_{k+1}}(F)$ is generated by $\widetilde{Q}_k(F)$ and $w_0.$  This is a consequence of Proposition \ref{prop:gen1}.
\end{proof}

\subsection{The modulation group for the Rankin-Selberg monoid} \label{ssec:RSM}
Let $M_{\otimes}$ be the reductive monoid whose points in an $F$-algebra $R$ are given by
\begin{equation}\label{monoidtensorproddef}
    M_\otimes(R) = \{(X_1, X_2) \in M_{2}(R) \times M_{2}(R) : \det(X_1) = \det(X_2)\}.
\end{equation}
Let $G:=M_{\otimes}^\times$ be its group of units.
We refer to $M_{\otimes}$ as the Rankin-Selberg monoid. It is the reductive monoid attached to the tensor product representation
$$
\otimes:{}^LG \lto \GL_4(\CC).
$$
There is a canonical closed immersion
\begin{align}
\omega_{\otimes}:M_{\otimes}\lto M_{2} \times M_{2}.
\end{align}
We can choose data so that this is the morphism of Lemma \ref{lem:omegarho}.

We view $M_2$ and $M_2 \times M_2$ as quadratic spaces equipped with quadratic forms $X \mapsto \det X$ and  $(X,Y) \mapsto \det X-\det Y.$  Let 
$$
W:=\GG_a \oplus M_2 \oplus M_2 \oplus \GG_a
$$
equipped with the quadratic form $(a,X,Y,b) \mapsto ab+\det X-\det Y.$

We have a Fourier transform 
$$
\mathcal{F}_{M_{\otimes}}:\mathcal{S}(M_{\otimes}(F)) \lto \mathcal{S}(M_{\otimes}(F)).
$$
When we refer to modulation groups in this section, we will always mean modulation groups with respect to this transform.  

In the present case, the small modulation group is equal to the group
\[
\Psi^{\mathrm{s}}_{\otimes} = M_2 \times M_2 \rtimes \left(G \times G\right)/\Delta(Z_G).
\]
Here we have used the quadratic form to identify $M_2 \times M_2$ with its dual.
Let
\begin{align}
    (\mathrm{GSO}_{M_2} \times \mathrm{GSO}_{M_2})^\circ(R):=\{(g_1,g_2)\in \mathrm{GSO}_{M_2}(R) \times \mathrm{GSO}_{M_2}(R):\nu(g_1)=\nu(g_2)\}.
\end{align}

\begin{prop}\label{TensorProductSmallProp}
    We have that $\Psi^s_{\otimes}= M_2 \times M_2 \rtimes(\mathrm{GSO}_{M_2}\times \mathrm{GSO}_{M_2})^\circ.$    
\end{prop}
\begin{proof}
The action
$
M_2 \times \GL_2 \times \GL_2 \to M_2$
induces an isomorphism
$\GL_2 \times \GL_2/\Delta(Z_{\GL_2}) \tilde{\to}\mathrm{GSO}_{M_2}.$  The proposition follows in a straightforward manner from this observation.
\end{proof} 

\begin{thm} One has that $\Psi_{\otimes}\{F\}=\Psi_{\omega}(\mathrm{GSO}_{W}(F)).$
\end{thm}
\begin{proof} 
The quadratic space $M_2 \times M_2$ is isomorphic to $V_4$ and the quadratic space $W$ is isomorphic to $V_{5}.$  
We have a closed immersion $(\mathrm{GSO}_{M_2} \times \mathrm{GSO}_{M_2})^\circ \to \mathrm{GSO}_{W}$ given on points by 
\begin{align}
    (g_1,g_2) \longmapsto \begin{psmatrix}\nu(g_1) & & & & \\ & g_1 & & & \\ & & g_2 & \\ & & & 1 \end{psmatrix}.
\end{align}
This extends to a homomorphism $M_2 \times M_2 \rtimes (\mathrm{GSO}_{M_2} \times \mathrm{GSO}_{M_2})^\circ \to \mathrm{GSO}_{W}$ sending $M_2 \times M_2$ to $N_4.$ 
By the proof of Theorem \ref{thm:mod:cone} and Proposition \ref{TensorProductSmallProp} 
we deduce that 
\begin{align}
\Psi_{\otimes}\{F\}=\widetilde{\tau}\left(
\langle w_0,N(F) \rtimes (\mathrm{GSO}_{M_2} \times \mathrm{GSO}_{M_2})^\circ(F) \rangle\right).
\end{align}

Consider the subgroup 
\begin{align} \label{show:big}
\langle w_0,N(F) \rtimes (\mathrm{GSO}_{M_2} \times \mathrm{GSO}_{M_2})^\circ(F) \rangle \leq \mathrm{GSO}_{W}(F).
\end{align}
We claim that this inequality is in fact an equality.  Proving the claim will complete the proof of the theorem.

For $a \in F^\times,$ choose $h\in (\mathrm{GSO}_{M_2} \times \mathrm{GSO}_{M_2})^\circ(F)$  so that $a = \nu(h)$. Then
\[
    \left(\begin{smallmatrix}
        a & & \\ & I_{M_2 \times M_2} & \\ & & a^{-1}
    \end{smallmatrix}\right) = \left(\begin{smallmatrix}
        \nu(h) & & \\ & h & \\ & & 1
    \end{smallmatrix}\right) \left(\begin{smallmatrix}
        & & 1 \\ & I_{M_2 \times M_2} & \\ 1 & & 
    \end{smallmatrix}\right) \left(\begin{smallmatrix}
        \nu(h)^{-1} & & \\ & h^{-1} & \\ & & 1
    \end{smallmatrix}\right) \left(\begin{smallmatrix}
        & & 1 \\ & I_{M_2 \times M_2} & \\ 1 & & 
    \end{smallmatrix}\right).
\]
Using this observation we deduce that 
the left hand side of \eqref{show:big} contains the $F$-points of a maximal split torus of $\mathrm{GSO}_{W}(F).$  One can now deduce the theorem from \cite[Proposition 6.2(v) and 6.11(i)]{BT:abstraits}.\quash{
One checks that it also contains the $F$-points of all root spaces associated to all simple roots with respect to either the upper or lower triangular matrices.  Hence it contains the unipotent radicals of these subgroups.  We deduce the claim from Proposition \ref{prop:gen1}.

 We now prove that it contains all the root subgroups. Let us denote by $\{e_i\}_{i = 0,\cdots,5}$ a basis of $X^{*}(T)$, where $T$ is a maximal split torus of $\mathrm{GSO}_{10}$. The group $(\mathrm{GSO}_{4}\times \mathrm{GSO}_{4})^\circ$ embedded into $\mathrm{GSO}_{10}$ as above has the following set of positive roots within the root system of $\mathrm{GSO}_{10}$: \[\Phi_1^{+} = \{e_2-e_3,e_2+e_3,e_4-e_5,e_4+e_5,-e_5+e_4,e_3+e_2\}.\] 
   Furthermore, the embedding of $V_8^{\vee}$ in $\mathrm{GSO}_{10}$, together with $w_0$, give rise to the following subsets of roots (along with their negatives):
    \[\Phi_2^+ = \{e_1-e_2,e_1-e_3,e_1-e_4,e_1-e_5,e_1+e_5,e_1+e_4,e_1+e_3,e_1+e_2\}.\]
    It is straightforward to verify that all the roots of $\mathrm{GSO}_{10}$ are generated by linear combinations of $\Phi_1^+$ and $\Phi_2^+$. In fact: $e_2-e_4 = (e_1-e_4)-(e_1-e_2)$, $e_2-e_5 = (e_1-e_5)-(e_1-e_2)$, $e_2+e_5 = (e_1+e_2)-(e_1-e_5)$, $e_2+e_4 = (e_1+e_2)-(e_1-e_4)$, $(e_3-e_4) = (e_1-e_5)-(e_1-e_3)$, $e_3+e_5 = (e_1+e_3)-(e_1-e_5)$, $e_3+e_4 = (e_1+e_3)-(e_1-e_4)$. }
\end{proof}

\section{Deformation quantization and the semi-classical limit}  \label{sec:DQ}
\subsection{Differential operators}\label{subsec:diffop} Let $F$ be a field. We will require both the algebraic and analytic theory of differential operators.  We were unable to find references that explained the relationship in the generality we require, so we provide some background information.

For an $F$-scheme $X,$ we have a sheaf $\underline{\mathrm{End}_F}(\OO_X)$ of $F$-linear endomorphisms of $\OO_X$ and a sheaf $\underline{\mathrm{End}_{\OO_X}}(\OO_X)$ of $\OO_X$-linear endomorphisms of $\OO_X.$
The sheaf $\underline{\mathrm{End}_F}(\OO_X)$ is a non-commutative $\OO_X$-algebra so the usual commutator bracket $[\cdot,\cdot]$ is defined, and $\underline{\mathrm{End}_{\OO_X}}(\OO_X)$ is canonically isomorphic to $\OO_X$.

Set $\underline{\mathcal{D}_X^{\leq 0}} = \underline{\mathrm{End}_{\OO_X}}(\OO_X)$, and for $n\geq 1$ set
\begin{align*}
    \underline{\mathcal{D}_X^{\leq n}} := \left\{T\in\underline{\mathrm{End}_F}(\OO_X)\mid [T,f]\in\underline{\mathcal{D}_X^{n-1}}\text{ for all }f\in \underline{\mathcal{D}_{X}^{\leq 0}}\right\}.
\end{align*}
The sheaf of differential operators on $X$ is given by the directed union
\begin{align*}
    \underline{\mathcal{D}_X} := \bigcup_{n\geq 0}\underline{\mathcal{D}_X^{\leq n}}.
\end{align*}
This is a (non-commutative) $\OO_X$-subalgebra of $\underline{\mathrm{End}_F}(\OO_X)$, filtered by $(\underline{\mathcal{D}_X^{\leq n}})_{n\geq 0}$.

Throughout the rest of this subsection, let $X$ be a scheme of finite type over $F.$ We assume
\begin{enumerate}[label=(M\text{\arabic*}), ref=M\text{\arabic*}]\setcounter{enumi}{4}
    \item\label{X:cod} $X$ is normal and $\mathrm{codim}(X-X^{\mathrm{sm}},X) \geq 2.$
\end{enumerate}
Let $j:X^{\mathrm{sm}}\to X$ denote the inclusion.
\begin{lem}\label{M5:OX->OXsm:isom} The canonical map $\OO_X\to j_*\OO_{X^{\mathrm{sm}}}$ is an isomorphism of $\OO_X$-modules.
\end{lem}
\begin{proof}
If $U$ is any open subscheme of $X$ then it again satisfies \eqref{X:cod}.  Thus the lemma follows from \cite[Theorem 6.45]{Gortz_Wedhorn}.
\end{proof}

For any open subscheme $U$ of $X,$ note that
\begin{align*}
    j_*\underline{\mathrm{End}_F}(\OO_{X^{\mathrm{sm}}})(U) = \mathrm{End}_F(\OO_{U^{\mathrm{sm}}}) = \mathrm{End}_F((j_*\OO_{X^{\mathrm{sm}}})\vert_U).
\end{align*}
By Lemma \ref{M5:OX->OXsm:isom}, the canonical isomorphism $\OO_X\to j_*\OO_{X^{\mathrm{sm}}}$ then induces a canonical isomorphism $\mathrm{End}_F(\OO_X\vert_U)\cong \mathrm{End}_F((j_*\OO_{X^{\mathrm{sm}}})\vert_U)$. Varying $U$ yields the following corollaries:
\begin{cor}\label{M5:OX->OXsm:isom:End} The canonical map $\underline{\mathrm{End}_F}(\OO_X)\to j_*\underline{\mathrm{End}_F}(\OO_{X^{\mathrm{sm}}})$ is an isomorphism.
\hfill\qedsymbol
\end{cor}

\begin{cor}\label{M5:DX->DXsm:isom} The canonical map $\underline{\mathcal{D}_X}\to j_*\underline{\mathcal{D}_{X^{\mathrm{sm}}}}$ is an isomorphism of filtered $\OO_X$-algebras. \qed
\end{cor}\quash{
\begin{proof} The map of \eqref{M5:OX->OXsm:isom:End} restricts to injections $\underline{\mathcal{D}_X^{\leq n}}\to j_*\underline{\mathcal{D}_{X^{\mathrm{sm}}}^{\leq n}}$ for all $n\geq 0$. It remains to show they are surjective. This follows from an easy induction on $n$, the base case being Lemma \ref{M5:OX->OXsm:isom}. 
\end{proof}}

Let
\begin{align}\label{M5:DX->DXsm:isom:glob:sec}
    \mathcal{D}_X := \underline{\mathcal{D}_X}(X)
\end{align}
be the global sections of the sheaf $\underline{\mathcal{D}_{X}}.$
Similarly set $\mathcal{D}_X^{\leq n}:=\underline{\mathcal{D}_X^{\leq n}}(X)$. Then Corollary \ref{M5:DX->DXsm:isom} implies the restriction
\begin{align} \label{iso:smooth}
    \mathcal{D}_X&\lto \mathcal{D}_{X^{\mathrm{sm}}}
\end{align}
is an isomorphism of filtered $\OO_X(X)$-algebras.

For an $F$-algebra $R$, define the algebra of differential operators $\mathcal{D}_R := \bigcup\limits_{n\geq 0}\mathcal{D}_R^{\leq n}$ in the usual manner.  Then the canonical map $\mathrm{End}_F(\OO_X)\to \mathrm{End}_F(\OO_X(X))$ restricts to a map
\begin{align}\label{DX:glob:sec:1}
    \mathcal{D}_X \lto \mathcal{D}_{\OO_X(X)}
\end{align}
of filtered $F$-algebras.  The following is \cite[\href{https://stacks.math.columbia.edu/tag/0G44}{Tag 0G44}]{stacks-project}:
\begin{lem}\label{DX:glob:sec:1:lem} If $X$ is affine, then \eqref{DX:glob:sec:1} is an isomorphism. \qed
\end{lem}

Let $H$ be an affine $F$-group scheme and let $X$ be an affine $F$-scheme with a right action of $H.$ In particular we have a group anti-homomorphism $H(F)\to\mathrm{Aut}_{\textbf{Sch}_F}(X)$. Since $X$ is affine, the usual anti-equivalence gives a group anti-homomorphism $\mathrm{Aut}_{\textbf{Sch}_F}(X)\to \mathrm{Aut}_{\textbf{Alg}_F}(\OO_X(X))$. Finally, conjugation defines a group homomorphism $\mathrm{Aut}_{\textbf{Alg}_F}(\OO_X(X))\to \mathrm{Aut}_{\textbf{Mod}_F}(\mathrm{End}_F(\OO_X(X)))$. In sum, all of these compose to a group homomorphism 
\begin{align}\label{HactDX}
    H(F)\to \mathrm{Aut}_{\textbf{Mod}_F}(\mathrm{End}_F(\OO_X(X)))
\end{align}
defining a left $H(F)$-action on $\mathrm{End}_F(\OO_X(X))$. Note that $H_X(F)$ acts trivially on $\mathrm{End}_F(\OO_X(X))$ as it acts trivially on $X.$
\begin{lem}\label{HactDX:preserve:filtration} For each $n\geq 0$, $\mathcal{D}_{\OO_X(X)}^{\leq n}$ is an $H(F)$-invariant subspace of $\mathrm{End}_F(\OO_X(X))$.
\end{lem}
\begin{proof} This follows from an induction based on the identity $[h.T,h.S] = h.[T,S]$ for $T,S$ in $\mathrm{End}_F(\OO_X(X))$ and $h\in H(F)$.     
\end{proof}

Assume $X$ is affine. We equip $\mathcal{D}_X$ with the unique left $H(F)$-action such that \eqref{DX:glob:sec:1} is $H(F)$-equivariant. Similarly we equip $\mathcal{D}_{X^{\mathrm{sm}}}$ with the unique left $H(F)$-action making \eqref{iso:smooth} $H(F)$-equivariant.

We conclude this subsection by briefly explaining the passage from algebraic differential operators to analytic ones. Let $X_0$ be a smooth scheme  over $F.$ Recall the sheaf $\underline{\mathrm{Der}_F}(\OO_{X_0})$ of derivations and the following result from \cite[Th\'eor\`eme 16.11.2]{EGAIV}.
\begin{lem}\label{Dx:smooth:generation:deg1} One has that $\underline{\mathrm{Der}_F}(\OO_{X_0})\subseteq \underline{\mathcal{D}_{X_0}^{\leq 1}}$. The induced map
\begin{align*}
    \OO_{X_0}\oplus \underline{\mathrm{Der}_F}(\OO_{X_0})\lto \underline{\mathcal{D}_{X_0}^{\leq 1}}
\end{align*}
is an isomorphism of $\OO_{X_0}$-modules. Moreover,  $\underline{\mathcal{D}_{X_0}}$ is the $\OO_{X_0}$-subalgebra of $\underline{\mathrm{End}_F}(\OO_{X_0})$ generated by $\underline{\mathcal{D}_{X_0}^{\leq 1}}$.
\end{lem}
\noindent The last assertion does not imply that $\mathcal{D}_{X_0}$ is generated by $\mathcal{D}_{X_0}^{\leq 1}$. Nevertheless this is true when $X_0$ is affine \cite[Theorem 1.15]{Muhasky}.

By definition, there is a canonical isomorphism of sheaves
\begin{align*}
\underline{\mathrm{Der}_F}(\OO_{X_0})\lto \Omega_{X_0/F}^\vee
\end{align*}
and $\Omega_{X_0/F}^\vee$ is canonically isomorphic to the sheaf of sections of the tangent bundle $TX_0\to X_0$. 

Now let $F$ be an Archimdean local field. Assume $X_0(F)$ is Zariski dense in $X_0$. In particular, $U(F)=X_0(F)\cap U$ is Zariski dense in $U$ for every nonempty open subscheme $U$ of $X_0$, and $U(F)$ is an $F$-analytic manifold. By identifying $\OO_X(U)$ with $\Hom_{\textbf{Sch}_F}(U,\mathbb{G}_a)$, taking $F$-points gives rise to an $F$-algebra homomorphism
\begin{align}\label{M6:Fpoint:inj}
    \OO_{X_0}(U)\lto C^\infty(U(F),F).
\end{align}
By assumption $U(F)$ is Zariski dense is $U,$ so this is injective. Similarly, taking $F$-points gives an injection
\begin{align}
    \Omega_{X_0/F}^\vee(U) \cong \Hom_{\textbf{Sch}_{X_0}}(U,TX_0)\lto \Gamma(U(F), TX_0(F))
\end{align}
where $\Gamma(U(F), TX_0(F))$ denotes the space of sections of $TX_0(F)\to X_0(F)$ over $U(F)$. All these allow us to realize $\underline{\mathrm{Der}_F}(\OO_{X_0})$ as vector fields on $X_0(F).$
\quash{\begin{rem} We've avoided the use of sheaves of sets in the above discussion. One can view $U\mapsto \Hom_{\textbf{Sch}_F}(U,\mathbb{G}_a)$ as a sheaf over $X$ and $\OO_X$ is then isomorphic to $\Hom_{\textbf{Sch}_F}(-,\mathbb{G}_a)$. Similarly for the vector fields.
\end{rem}}

\subsection{Action of the small modulation group on $\mathcal{D}_{X\CC}$} We now return to the setting of  \S \ref{ssec:mods} in the case of an Archimedean local field $F.$ By Weil restriction there is no loss of generality in assuming $F=\mathbb{R}$.

Therefore, we are given an affine algebraic group $H$ over $\mathbb{R}$, an affine $H$-scheme $X$ of finite type over $\mathbb{R}$ and a right representation of $H$ on a finite dimensional $\mathbb{R}$-vector space $V$ together with an $H$-equivariant map $\omega:X\to V.$ We assume the data satisfy \eqref{omega:span}, \eqref{omega:scale},  \eqref{HX:H1=0}, \eqref{HX:Fpoints} and \eqref{X:cod}. Assume further that
\begin{enumerate}[label=(M\text{\arabic*}), ref=M\text{\arabic*}]\setcounter{enumi}{5}
    \item\label{dense} $X^{\mathrm{sm}}(\mathbb{R})$ is (Zariski) dense in $X^{\mathrm{sm}}$.
\end{enumerate}
In particular, $U^{\mathrm{sm}}(\mathbb{R})$ is Zariski dense in $U^{\mathrm{sm}}$ for each open subscheme $U$ of $X.$ Hence the assumption \eqref{dense} implies the map $\OO_X(U)\to C^\infty(U^{\mathrm{sm}}(\mathbb{R})):=C^\infty(U^{\mathrm{sm}}(\mathbb{R}),\mathbb{C})$ induced by \eqref{M6:Fpoint:inj} is injective. 

\begin{rem}
We point out that \eqref{dense} is often automatic. If $H$ is connected and $x \in X(\mathbb{R})$ is a point such that the orbit $O(x_0) \subset X$ is dense, then \eqref{dense} follows from Lemma \ref{lem:dense}. 
\end{rem}

Let $C^\infty_{X^{\mathrm{sm}}}$ denote the Zariski sheaf on $X^{\mathrm{sm}}$ given
 by
\begin{align*}
    C^\infty_{X^{\mathrm{sm}}}(U):= C^\infty(U(\mathbb{R}))
\end{align*}
for open subschemes $U\subseteq X^{\mathrm{sm}}$. For $v^{\vee} \in V^{\vee}(\mathbb{R})$, $\theta \in \mathcal{D}_{X^{\mathrm{sm}}}$ and open subschemes $U \subseteq X^{\mathrm{sm}}$ we define $v^\vee(\theta)\vert_U\in \mathrm{End}_\mathbb{C}C^\infty_X(U)$ 
to be the arrow making the following diagram commute:
\begin{equation}\label{vactF}
\begin{tikzcd}
C^{\infty}(U(F)) \arrow[d, "\theta"'] \arrow[rr, "{\cdot \psi(v^{\vee}\circ\omega)}"] && C^{\infty}(U(F)) \arrow[d, "v^{\vee}(\theta)|_{U}"] \\
C^{\infty}(U(F)) \arrow[rr, "{\cdot \psi(v^{\vee}\circ\omega)}"']                     && C^{\infty}(U(F))                        
\end{tikzcd}
\end{equation}
Here the horizontal maps are given by multiplication by the function $x\mapsto \psi(v^{\vee}\circ\omega(x))$. The endomorphisms $v^\vee(\theta)|_U$ together glue to 
\begin{align*}
    v^\vee(\theta) \in \mathrm{End}_\mathbb{C}C^\infty_{X^{\mathrm{sm}}}.
\end{align*}
Repeating the construction with $X^{\mathrm{sm}}$ replaced by its open subspaces, we obtain an injective morphism of sheaves of $\mathbb{R}$-algebras
\begin{align}\label{mod:Dx:Cinfty}
    v^\vee:\underline{\mathcal{D}_{X^{\mathrm{sm}}}}\lto \underline{\mathrm{End}_\mathbb{C}}C^\infty_{X^{\mathrm{sm}}}.
\end{align} 

The horizontal arrows in \eqref{vactF} are multiplication by transcendental functions. Nonetheless we have the following proposition: 
\begin{prop}\label{VactDx:preserve:filtation} For $v^{\vee} \in V^{\vee}(\mathbb{R})$, the map of \eqref{mod:Dx:Cinfty} has image in $\underline{\mathcal{D}_{X^{\mathrm{sm}}}}\otimes_\mathbb{R}\mathbb{C}$. Moreover, $v^\vee(\underline{\mathcal{D}_{X^{\mathrm{sm}}}^{\leq n}})\leq \underline{\mathcal{D}_{X^{\mathrm{sm}}}^{\leq n}}\otimes_\mathbb{R}\mathbb{C}$ for each $n\geq 0$.
\end{prop} 
\begin{proof} The problem is local in nature, so we may assume $X^{\mathrm{sm}}=\Spec A$ is affine with $A=\mathbb{R}[x_1,\ldots,x_n]/I$ for some ideal $I \leq \RR[x_1,\ldots,x_n].$ Since $A$ is smooth, by Lemma \ref{DX:glob:sec:1:lem} and Lemma \ref{Dx:smooth:generation:deg1}, it suffices to show $v^\vee(\mathrm{Der}_\mathbb{R}(A))\subseteq \mathcal{D}_{A}^{\leq 1}\otimes_\mathbb{R}\mathbb{C}$. Here $\mathrm{Der}_\mathbb{R}(A)$ is the set of $\mathbb{R}$-linear derivations on $A$.

Let $\theta\in \mathrm{Der}_\mathbb{R}(A)$. Let $f := v^\vee\circ \omega\vert_{X^{\mathrm{sm}}} \in A$ and by abuse of notation denote by $f$ a lift to $\mathbb{R}[x_1,\ldots,x_n]$. By definition 
\begin{align*}
    v^\vee(\theta) = e^{-2\pi i rf}\circ \theta\circ e^{2\pi i rf}
\end{align*}
for some $r\in\mathbb{R}^\times$. Say $\theta = \sum\limits_{j=1}^n a_j \frac{\partial}{\partial x_j}\pmod{I}$ for some $a_j\in \mathbb{R}[x_1,\ldots,x_n]$. Then
\begin{align*}
    v^\vee(\theta) = \sum_{j=1}^n a_j \left(\tfrac{\partial}{\partial x_j} +2\pi i r\tfrac{\partial}{\partial x_j}(f)\right) \pmod{I}
    = \theta + 2\pi ir\theta(f).
\end{align*}
Since $\theta(f)\in A$, it follows that $v^\vee(\theta)\in \mathcal{D}^{\leq 1}_{A}\otimes_\mathbb{R}\mathbb{C}$. 
\end{proof}

\begin{comment} \begin{proof} 
By Lemma \ref{lem:DXsmooth}, it suffices to choose a an open cover of $X^{\mathrm{sm}}$ in the Zariski topology and verify that $v^\vee(\theta)|_{U} \in \underline{\mathcal{D}_{X^{\mathrm{sm}}}}(U)$ for all $U$ in the cover.

We use local coordinate systems as in \cite[\S 1.1]{HTT}.  In more detail, we can choose a cover of $X^{\mathrm{sm}}$ by affine open subschemes so that for each $U$ in the cover one has 
\begin{align}
    \underline{\mathcal{D}_X}(U)=\langle x_1,\dots,x_r,\partial_1,\dots,\partial_r\rangle
\end{align}
where $\langle x_1,\dots,x_r\rangle= \OO_{X^{\mathrm{sm}}}(U)$ and $\partial_1,\dots,\partial_r$ are vector fields satisfying the usual commutation relations.  For $\theta \in \underline{\mathcal{D}_X}(U)$
$$
v^\vee(\theta)|_U=e^{f}\theta e^{-f}
$$
for some $f \in \CC[x_1,\dots,x_r].$  If $\theta =x_i,$ then this is just $x_i.$  On the other hand if $\theta=\partial_i,$ then this is $\partial_i+\partial_if.$  We deduce that $v^\vee(\theta)|_{U} \in \underline{\mathcal{D}_{X^{\mathrm{sm}}}}(U)$ as desired.
\quash{

---------------

\textcolor{blue}{HaoYun: We were talking if working \'etale locally. Consider the morphism of ringed space $f:(X(\RR),\OO_{X(\RR)})\to (X,\OO_X)$ where $f:X(\RR)\to X$ is the inclusion and the ring map $\OO_X\to f_*\OO_{X(\RR)}$ is obtained by taking $\RR$-points of morphisms. The way we define the action is via $\OO_X\to f_*\OO_{X(\RR)}\to f_*C^\infty_{X(\RR)}$. I don't know if $f_*C^\infty_{X(\RR)}$ is an \'etale sheaf or not.}}

\end{proof}
\end{comment}

Let
\begin{align*}
    \mathcal{D}_{X\CC} := \mathcal{D}_{X}\otimes_\mathbb{R}\mathbb{C} \cong \mathcal{D}_{X^{\mathrm{sm}}}\otimes_\mathbb{R}\mathbb{C}.
\end{align*}
Then \eqref{HactDX} and \eqref{vactF} define an action 
\begin{align} \label{small:action}
\Psi_{\omega}^{\mathrm{s}}(\mathbb{R}) \times \mathcal{D}_{X\CC} \lto \mathcal{D}_{X\CC}.
\end{align}
\quash{

\begin{rem}   
There are examples of spaces that are both reductive monoids and the affine closures of Braverman-Kazhdan spaces (see \cite{BK:normalized}).  Thus if $G=\mathrm{Res}_{\CC/\RR}G_0$ there are two potential definitions of the Schwartz space in these cases.  We do not know if they coincide.  
\end{rem}}

Assume Ansatz \ref{Ans:Schwartz} for the remainder of this subsection.

\begin{ans} \label{ans:DX}
    The action of $\mathcal{D}_{X}$ on $C^\infty(X^{\mathrm{sm}}(\mathbb{R}))$ preserves $\mathcal{S}(X(\mathbb{R}),\mathcal{L}^{1/2}).$
\end{ans}
\noindent We assume Ansatz \ref{ans:DX} as well for the remainder of the subsection.  For $\theta \in \mathcal{D}_{X\mathbb{C}}$ the \textbf{Fourier transform $\mathcal{F}_{X}(\theta)$} is the unique endomorphism on $\mathcal{S}(X(\RR),\mathcal{L}^{1/2})$ making the following diagram commute:
\begin{equation}\label{Fdef}
    \begin{tikzcd}
\mathcal{S}(X(\mathbb{R}),\mathcal{L}^{1/2}) \arrow[d, "\theta"'] \arrow[r, "\mathcal{F}_{X}"] & \mathcal{S}(X(\mathbb{R}),\mathcal{L}^{1/2}) \arrow[d, "\mathcal{F}_{X}(\theta)"] \\
\mathcal{S}(X(\mathbb{R}),\mathcal{L}^{1/2}) \arrow[r, "\mathcal{F}_{X}"']                     & \mathcal{S}(X(\mathbb{R}),\mathcal{L}^{1/2}).                               
\end{tikzcd}
\end{equation}

\begin{ans}\label{ans:alg}
    If $\theta \in \mathcal{D}_{X}$  then $\mathcal{F}_{X}(\theta)$ lies in the image of  $\mathcal{D}_{X} \to \mathrm{Aut}(\mathcal{S}(X(\mathbb{R}),\mathcal{L}^{1/2}))$.
\end{ans}

\begin{rem}
In a follow up paper to \cite{Hsu:Asymp} Hsu will prove Ansatz \ref{ans:DX} and Ansatz \ref{ans:alg} when $X$ is a horospherical variety as in \S \ref{sec:HS}.
\end{rem}

\begin{lem} \label{lem:DX:act}
Assuming Ansatz \ref{Ans:Schwartz}, \ref{ans:DX} and \ref{ans:alg} there is a unique action
\begin{align} \label{conj:action}
\Psi_{\omega}\{\mathbb{R}\} \times \mathcal{D}_{X\CC} \lto \mathcal{D}_{X\CC}
\end{align}
such that $\Psi_{\omega}^{\mathrm{s}}(\mathbb{R})$ acts via \eqref{small:action} and $\mathcal{F}_{X}$ acts via 
\eqref{Fdef}.
\end{lem}
\begin{proof}
By \eqref{assum:dense}, $\mathcal{S}(X(\mathbb{R}),\mathcal{L}^{1/2})$ is dense in $L^2(X^{\mathrm{sm}}(\mathbb{R}),\mathcal{L}^{1/2}).$  Thus an automorphism in $\Psi_{\omega}\{\mathbb{R}\},$ which is originally defined as a group of automorphisms of $L^2(X^{\mathrm{sm}}(\mathbb{R}),\mathcal{L}^{1/2}),$ is uniquely determined by its action on $\mathcal{S}(X(\mathbb{R}),\mathcal{L}^{1/2}).$  The lemma follows.
\end{proof}

\subsection{Passage to the semi-classical limit} \label{ssec:semi}

Consider the associated graded
\begin{align*}
\mathrm{gr}\,\underline{\mathcal{D}_{X}} = \bigoplus_{n\geq 0} \underline{\mathcal{D}_X^{\leq {n}}} / \underline{\mathcal{D}_X^{\leq n-1}}
\end{align*}
where we set $\underline{\mathcal{D}_X^{\leq -1}} = 0$. This is a sheaf of commutative $\OO_{X}$-algebras. Similarly one defines $\mathrm{gr}\,\underline{\mathcal{D}_{X^{\mathrm{sm}}}}$. By Corollary \ref{M5:DX->DXsm:isom}, there is a canonical isomorphism
\begin{align}\label{M5:DX->DXsm:isom:gr}
    \mathrm{gr}\,\underline{\mathcal{D}_{X}} \lto j_*\left(\mathrm{gr}\,\underline{\mathcal{D}_{X^{\mathrm{sm}}}} \right)
\end{align}
of commutative $\OO_X$-algebras.

One has a canonical identification $\Omega_{X^{\mathrm{sm}}/\mathbb{R}}^\vee \cong \underline{\mathrm{Der}_\mathbb{R}}(\OO_{X^{\mathrm{sm}}}).$  
The composition 
\begin{align*}
    \Omega_{X^{\mathrm{sm}}/\mathbb{R}}^\vee\tilde{\lto} \underline{\mathrm{Der}_\mathbb{R}}(\OO_{X^{\mathrm{sm}}}) \leq \mathrm{gr}\,\underline{\mathcal{D}_{X^{\mathrm{sm}}}}
\end{align*}
extends to an $\OO_{X^{\mathrm{sm}}}$-algebra morphism 
\begin{align}\label{DX:gr:TX}
\mathrm{Sym}_{\OO_{X^{\mathrm{sm}}}}\Omega_{X^{\mathrm{sm}}/\mathbb{R}}^\vee \lto \mathrm{gr}\,\underline{\mathcal{D}_{X^{\mathrm{sm}}}}.
\end{align}
By a local computation recorded in  \cite[\S 1.1]{HTT} \eqref{DX:gr:TX} is an isomorphism.  Strictly speaking they work over the complex numbers but the argument is valid over characteristic zero fields.  Taking global sections of \eqref{DX:gr:TX} and \eqref{M5:DX->DXsm:isom:gr} we have isomorphisms
\begin{align*}
\OO_{T^*X^{\mathrm{sm}}}(T^*X^{\mathrm{sm}}) \cong (\mathrm{gr}\,\underline{\mathcal{D}_{X^{\mathrm{sm}}}})(X^{\mathrm{sm}}) \cong (\mathrm{gr}\,\underline{\mathcal{D}_{X}})(X).
\end{align*}

The action of the small modulation group $\Psi_\omega^s$ on $T^*X^{\mathrm{sm}}$ in Lemma \ref{lem:cot:action} induces a left $\Psi_\omega^s(\mathbb{R})$-action on $\OO_{T^*X^{\sm}}(T^*X^{\sm})$. Under the identifications above, this yields an action
 \begin{align} \label{C:ext:action}
     \Psi_\omega^s(\mathbb{R}) \times (\mathrm{gr}\,\underline{\mathcal{D}_{X}})(X)\lto (\mathrm{gr}\,\underline{\mathcal{D}_{X}})(X).
 \end{align}
There is a canonical injective $\OO_X(X)$-algebra homomorphism
\begin{align}\label{grDx:glob:sec:map}
    \mathrm{gr}\,\mathcal{D}_X := \bigoplus_{n\geq 0} \mathcal{D}_X^{\leq n} / \mathcal{D}_{X}^{\leq n-1} \lto (\mathrm{gr}\,\underline{\mathcal{D}_{X}})(X).
\end{align}
The left $H(\mathbb{R})$-action on $\mathcal{D}_X$ induces an $H(\mathbb{R})$-action on $\mathrm{gr}\,\mathcal{D}_X$. With these actions, the map of \eqref{grDx:glob:sec:map} is $H(\mathbb{R})$-equivariant.

\begin{comment} The action of small modulation group $\Psi_\omega^s$ on $T^*X^{\mathrm{sm}}$ defined in \S\ref{sec:mod:alg} induces an $\Psi_\omega^s(F)$-action on $\Gamma(T^*X^{\sm},\OO_{T^*X^{\sm}})$. Under the identification \eqref{DX:gr:TX} this induces an action
 \begin{align} \label{C:ext:action}
     \Psi_\omega^s(F) \times \mathrm{gr}\,\mathcal{D}_X\to \mathrm{gr}\,\mathcal{D}_X
 \end{align} \begin{lem} The action of $H(F)$ on $\mathcal{D}_X$ induces an $H(F)$-action on $\mathrm{gr}\,\mathcal{D}_X$, and the identification \eqref{DX:gr:TX} intertwines this $H(F)$-action with \eqref{C:ext:action}.
\end{lem}
\begin{proof} The first assertion is clear from the proof of Lemma \ref{HactDX:preserve:filtration}. For the second one, let us write down explicitly the map $\mathcal{D}_{X}^{\leq 0}\oplus \mathcal{D}_{X}^{\leq 1}/\mathcal{D}_{X}^{\leq 0}\to \Gamma(T^*X^{\sm},\OO_{T^*X^{\sm}})$
\end{proof}
    
\end{comment}

\begin{rem} By Proposition \ref{VactDx:preserve:filtation}, the action \eqref{small:action} induces an $\Psi_\omega^{\mathrm{s}}(\mathbb{R})$-action on $\mathrm{gr}\,\mathcal{D}_{X\mathbb{C}}$. However, the $V^\vee(\mathbb{R})$-action is trivial, as opposed opposed to \eqref{C:ext:action}.
\end{rem}

We now prepare to explain how a conjugate of the action \eqref{C:ext:action} may 
 extend to an action of $\Psi_{\omega}\{\mathbb{R}\},$ and formulate how this action ``corresponds" to the action on $\mathcal{D}_{X\mathbb{C}}$.   
Let
\begin{align}
\sigma:\mathcal{D}_{X} \lto \textrm{gr}\,\mathcal{D}_{X}
\end{align}
be the canonical map; it is called the symbol map. It is a multiplicative, but not additive $H(\mathbb{R})$-equivariant homomorphism.
Using the symbol map the associated graded 
$\textrm{gr}\,\mathcal{D}_{X}$ inherits the structure of a commutative Poisson algebra.  The Poisson bracket is determined by $\{f, g\} := \sigma\left([\widetilde{f}, \widetilde{g}]\right)$, where $f$ and $g$ are homogeneous elements of $\textrm{gr}\, \mathcal{D}_X$ and $\widetilde{f}$ and  $\widetilde{g}$ are arbitrary lifts of $f$ and $g$  to $\mathcal{D}_X.$ 

\quash{
We have that $T^*X^{\sm}$ is a quasi-affine variety, with affine closure
\[
\overline{T^*X^{\sm}}^{\mathrm{aff}} =  \textrm{Spec} (\mathrm{gr}\,\mathcal{D}_{X}) .
\]

The action of the small modulation group $\Psi^{\mathrm{s}}_{\omega}$ on $T^*X^{\mathrm{sm}}$ given in Lemma \ref{lem:cot:action} induces an action of
 $\Psi^{\mathrm{s}}_\omega$ on $\overline{T^*X^{\mathrm{sm}}}^{\mathrm{aff}}.$  In the usual manner this induces an action of $ \Psi_{\omega}^{\mathrm{s}}(F)$ on $ \Gamma(T^*X^{\mathrm{sm}},\OO_{T^*X^{\mathrm{sm}}})=\mathrm{gr}\,\mathcal{D}_X;$ extending $\CC$-linearly we obtain 
 \begin{align} \label{C:ext:action}
     \Psi_{\omega}^{\mathrm{s}}(F) \times \mathrm{gr}\,\mathcal{D}_{X\CC} \lto \mathrm{gr}\,\mathcal{D}_{X\CC}
 \end{align}

We suggest how a conjugate of the action \eqref{C:ext:action} may 
 extend to an action of $\Psi_{\omega}\{\mathbb{R}\}$, and to more precisely formulate how this action ``corresponds" to the action on $\mathcal{D}_{X\mathbb{C}}$.   
}

Assume that $\mathfrak{x} \subset \mathcal{D}_{X\CC}$ is a complex Lie subalgebra under the commutator that is $H(\mathbb{R})$-stable. We assume moreover that $\mathfrak{x}$ admits a decomposition
\begin{align*}
    \mathfrak{x}=\bigoplus_{i}\mathfrak{x}_j
\end{align*}
into $H(\mathbb{R})$-invariant subspaces such that 
\begin{align*}
    \sigma|_{\mathfrak{x}_j}:\mathfrak{x}_j \lto \mathrm{gr}\,\mathcal{D}_{X}
\end{align*}
is $\mathbb{R}$-linear. Taking direct sum we obtain an $\mathbb{C}$-linear map
$\Sigma : \mathfrak{x}\to \mathrm{gr}\,\mathcal{D}_{X\CC}.$
Let $\mathfrak{x}' \subset \mathrm{gr}\,\mathcal{D}_{X\CC}$ be its image.  It is a Lie algebra under the Poisson bracket.  We point out that
$\Sigma\neq \sigma\vert_{\mathfrak{x}}$ in general
because $\sigma$ is not linear.

We observe that if $\mathfrak{x}$ is preserved by $\Psi^{\mathrm{s}}_{\omega}(\mathbb{R})$ then $\Psi_{\omega}^{\mathrm{s}}(\mathbb{R})$ acts
on $\mathfrak{x}'$ by transfer of structure.  It preserves the Poisson bracket on $\mathfrak{x}'.$

\begin{ans}\label{Ansatz}
Assume Ansatz \ref{Ans:Schwartz}, \ref{ans:DX} and Ansatz \ref{ans:alg}. 
There exists a complex Lie subalgebra $\mathfrak{x} \subset \mathcal{D}_{X\CC}$ as above that satisfies the following: 
    \begin{enumerate}
    \item  \label{xgen} $\mathfrak{x}$ generates $\mathcal{D}_{X\mathbb{C}}$ as an associative algebra; 
   
    \item  \label{preserve} $\mathfrak{x}$ is preserved by the action of $\Psi_{\omega}\{\RR\}$; 

\item \label{cot:gen} $\mathfrak{x}'$ generates $\mathrm{gr}\,\mathcal{D}_{X\mathbb{C}}$ as a commutative algebra;

       \item \label{ind:gp} There is an ind-group scheme $\Psi_{\omega}^{\mathrm{ind}}$ containing $\Psi_{\omega}^{\mathrm{s}}$ as a subgroup
       such that the action of $\Psi_{\omega}^{\mathrm{s}}$ on $T^*X^{\mathrm{sm}}$ extends to an action of $\Psi_{\omega}^{\mathrm{ind}}$ on $\overline{T^*X^{\mathrm{sm}}}^{\mathrm{aff}}$; 
       \item \label{compat}
       The action of $\Psi_{\omega}^{\mathrm{ind}}(\RR)$ on $\mathrm{gr}\,\mathcal{D}_{X\CC}$ preserves $\mathfrak{x}'.$  Moreover, the images of the homomorphisms
    \begin{align*}
        \Psi_{\omega}^{\mathrm{ind}}(\RR) \lto \mathrm{Aut}_{\CC}(\mathfrak{x}'),\\
        \Psi_{\omega}\{\RR\} \lto \mathrm{Aut}_{\CC}(\mathfrak{x}')
    \end{align*}
    are $\mathrm{Aut}_{\CC}(\mathfrak{x}')$-conjugate.
    \end{enumerate}    
\end{ans}

\begin{rem}
The modulation group $\Psi_{\omega}\{\mathbb{R}\}=\Psi_{\omega,\chi}\{\mathbb{R}\}$ depends on a choice of character $\chi.$  Changing $\chi$ does not affect the image of the lower automorphism in \eqref{compat}.  In particular, $\Psi_{\omega}^{\mathrm{ind}}$ can be taken to be independent of $\chi.$
\end{rem}

\begin{conj} \label{special:case:ansatz}
Ansatz \ref{Ansatz} is true if $X$ is a reductive monoid as in \S \ref{sec:monoid-mod} or if $X$ is a horospherical variety as in \S \ref{sec:HS}.
\end{conj}

\begin{rem} When $X=\overline{P^{\mathrm{der}} \backslash G}^{\mathrm{aff}}$ we suspect that something similar to the procedure in \S \ref{sssec:quadric} will yield the conjecture in certain cases. When $X=M_{\rho},$ one might try to look for the group ind-schemes in \eqref{ind:gp} using the Kac-Moody groups discussed in \cite{Shahidi:infinite}.
\end{rem}

\subsection{Examples} \label{ssec:exam} We prove Conjecture \ref{special:case:ansatz} in the cases considered in \S \ref{sec:Mod:VS} and \S \ref{sec:Mod:cones}.  Let $\alpha \in \RR^\times$ be chosen so that 
\begin{align} \label{alpha}
\psi(t)=e^{2\pi i\alpha t}.
\end{align}

\subsubsection{Affine Space} \label{ssec:affine:space}
In this subsection we take $X=V =\GG_a^n$ as in \S \ref{sec:Mod:VS} with $H=\GL_n$ acting on the right.  We use the notation of loc.~cit.
In particular, we identify $V$ with its dual $V^\vee$ using a perfect pairing as in \eqref{perfect:pair}.  For computations we pick the pairing given on points in an $\RR$-algebra $R$ by
\begin{align*}
V(R) \times V(R) &\lto R\\
((v_j),(w_j)) &\longmapsto \sum_{j=1}^nv_jw_j.
\end{align*}
We then have a symplectic space $W=V \oplus V$ with form $\langle\,,\,\rangle_{\wedge}$ defined as in \eqref{symp}.  Thus we have identifications
$$
T^*V=V \oplus V^\vee = V \oplus V=W.
$$
In this case the algebra $\mathcal{D}_{V\CC}$ is the Weyl algebra
\begin{align}
\mathcal{D}_{V\CC}=\CC\left[x_1,\dots,x_n,\frac{\partial}{\partial x_1},\dots,\frac{\partial}{\partial x_n}\right].
\end{align}

Ansatz \ref{Ans:Schwartz} holds with the usual Fourier transform attached to $\langle\,,\,\rangle$ and $\psi$ and the usual Schwartz space $\mathcal{S}(V(\RR)).$  The action of $\mathcal{D}_{V\CC}$ on $C^\infty(V(\RR))$ preserves $\mathcal{S}(V(\RR)),$ so Ansatz \ref{ans:DX} holds.  Ansatz \ref{ans:alg} holds as well.  Indeed, standard facts on the Fourier transform imply that for $p(x_i,\frac{\partial}{\partial x_1}) \in \mathcal{D}_{V\CC}$ one has that
\begin{align} \label{Fonpoly}
\mathcal{F}_{X}(p)\left(x_j,\frac{\partial}{\partial x_j}\right)=p\left(\frac{1}{2\pi i \alpha }\frac{\partial}{\partial x_j},-2\pi i \alpha x_j\right).
\end{align}
Moreover, for $A \in \GL_n(\RR)$
\begin{align} \label{A:act}
A.p\left(x_j,\frac{\partial}{\partial x_j}\right)=p\left((x_j)A,\left(\frac{\partial}{\partial x_j}\right)A^{-t}\right).
\end{align}

In the notation of Ansatz \ref{Ansatz} we choose
\begin{align} \label{x:def}
\mathfrak{x}:=\left\langle  x_1,\dots,x_n,\frac{\partial}{\partial x_1},\dots,\frac{\partial}{\partial x_n} \right\rangle \subset \mathcal{D}_{V}
\end{align}
where the brackets indicate the Lie subalgebra over $\CC$ generated by the given elements.  We have an isomorphism of Lie algebras
\begin{align} \label{LieHeis} \begin{split}
\mathrm{Lie} \,\mathrm{H}_W &\tilde{\lto} \mathfrak{x}\\
((v_j),(w_{j}),t) &\longmapsto \sum_{j=1}^nv_jx_j+\sum_{j=1}^{n}w_j\frac{\partial}{\partial x_j}+t .\end{split}
\end{align}
We set
$$
\mathfrak{x}_0:=\CC,\quad \mathfrak{x}_1:=\bigoplus_j \mathbb{C} x_j, \quad \mathfrak{x}_2:= \bigoplus_j \mathbb{C} \frac{\partial}{\partial x_j}.
$$
Then $\mathfrak{x}=\mathfrak{x}_0 \oplus \mathfrak{x}_1 \oplus \mathfrak{x}_2$ and each $\mathfrak{x}_i$ is $H(\mathbb{R})$-invariant.

\begin{prop} \label{prop:id}
If $\omega:X \to V$ is the identity map then Ansatz \ref{Ansatz} is valid with $$
\Psi_{\omega}^{\mathrm{ind}}=H_W/Z_{H_W} \rtimes \widetilde{\GL}_V \quad \textrm{and} \quad \mathfrak{x}=\mathrm{Lie} H_W.
$$
\end{prop}

\begin{proof} 
Assertion \eqref{xgen} is clear.  
Recall that we computed $\Psi_{\omega}\{\RR\}$ in Theorem \ref{thm:vector-Mod-group}.  
Using \eqref{action:HW}, the pullback of the action of $\Psi_{\omega}\{\mathbb{R}\}$ on $\mathcal{D}_{V}$ to 
$\mathrm{H}_W(\mathbb{R}) $ is given explicitly as follows:
\begin{align}
    ((v,\lambda),t).\left( \sum_{i=1}^na_ix_i+\sum_{i=1}^nb_i\frac{\partial }{\partial x_i}+c\right)=\sum_{i=1}^na_i(x_i+v_i) +\sum_{i=1}^nb_i\left(\frac{\partial}{\partial x_i}-2\pi i \alpha \lambda_i\right)+c.
\end{align}
The action of all of $\Psi_{\omega}\{\mathbb{R}\}$ is now determined by \eqref{Fonpoly} and \eqref{A:act}.  Assertion \eqref{preserve} follows.

The symbol map has image
$$
\CC[x_1,\dots,x_n,\xi_1,\dots,\xi_n]=\Gamma(V_\CC,\OO_{V_\CC}) \otimes_{\CC} \Gamma(V_{\CC},\OO_{V_\CC})=\Gamma(T^*V_{\CC},\OO_{T^*V_\CC}).
$$
One has that $\sigma(x_i)=x_i$ and $\sigma\left(\frac{d}{dx_i}\right)=\xi_i.$
Thus $\sigma|_{\mathfrak{x}_i}$ is linear for $1 \leq i \leq 3$ and \eqref{cot:gen} is valid.

Now consider the action of $\Psi_{\mathrm{Id}}^{\mathrm{ind}}:=\mathrm{H}_W/Z_{\mathrm{H}_W} \rtimes \widetilde{\GL}_V$ on $T^*V:=V \times V$ given by 
\begin{align*}
(a,b)w&=(b,-a),\\
    (a,b)((v,\lambda), g)&=((a-v)g,(b+\lambda)g^{-t})
\end{align*}
for $((a,b),((v,\lambda), g)) \in T^*V(R) \times \Psi_{\mathrm{Id}}^{\mathrm{ind}}(R).$  
We have a closed immersion given on points by 
\begin{align*}
\Psi_{\mathrm{Id}}^{\mathrm{s}}(R) &\lto \Psi_{\mathrm{Id}}^{\mathrm{ind}}(R)\\
(\lambda,g) &\longmapsto (((0,\lambda),0), g).
\end{align*}
\quash{\textcolor{red}{Is the above RHS $(0,\lambda) \rtimes g$?}} Using this closed immersion to identify $\Psi_{\mathrm{Id}}^{\mathrm{s}}$ with a subgroup of $\Psi_{\mathrm{Id}}$ we see that the action of $\Psi_{\mathrm{Id}}^{\mathrm{ind}}$ on $T^*V$ does indeed extend the action of $\Psi_{\mathrm{Id}}^{\mathrm{s}}$ on $T^*V.$  This proves \eqref{ind:gp}.  One can now check \eqref{compat} directly.
\end{proof}

\begin{prop} \label{prop:sym}
If $\omega:X \to \mathrm{Sym}_{\langle\,,\,\rangle}^{\vee}$ is the map of  \eqref{omega:2} then Ansatz \ref{Ansatz} is valid with $\Psi_{\omega}^{\mathrm{ind}}=\mathrm{Sp}_W.$
\end{prop}

\begin{proof} We have already verified \eqref{xgen} and \eqref{cot:gen} in the proof of Proposition \ref{prop:id}.  Due to our choice of pairing we may identify $J \in \mathrm{Sym}_{\langle\,,\,\rangle}(R)$ with a symmetric $n \times n$ matrix.  Then writing $a=(a_1,\dots,a_n)$ the action of $J$ on $\mathfrak{x}$ is given by
\begin{align} \label{J:act}
J.\left( \sum_{i=1}^na_ix_i+\sum_{i=1}^nb_i\frac{\partial}{\partial x_i}+c\right)=\left( \sum_{i=1}^na_ix_i+\sum_{i=1}^nb_i\left(\frac{\partial}{\partial x_i}+\pi i\alpha \frac{\partial}{ \partial x_i}\left(x^tJx \right)\right)+c\right). 
\end{align}
This implies \eqref{preserve}.

We set $\Psi_{\omega}^{\mathrm{ind}}:=\mathrm{Sp}_{W}$ and identify $\Psi_{\omega}^{\mathrm{s}}$ with the standard Siegel parabolic subgroup of $\Psi_{\omega}^{\mathrm{ind}}$ in the usual manner.  Then a conjugate of the standard representation of $\Psi_{\omega}^{\mathrm{ind}}$ on $W=V \oplus V$ extends the action of $\Psi_{\omega}^{\mathrm{s}},$ yielding \eqref{ind:gp}.  One can now use \eqref{J:act} to check \eqref{compat} directly.
\end{proof}

\begin{rem} Kontsevich and Belov-Kanel conjecture that the ind-group of all automorphisms of the Weyl Algebra and the group of Poisson-automorphisms of the commutative algebra $\RR[x_1, \ldots, x_n; \xi_1, \ldots , \xi_n]$ are isomorphic  \cite{BelovKanelKontsevich2005}. Above, we have only considered the relationship between certain subgroups of these automorphism groups.  
 It is possible that a generalization of the Kontsevich Belov-Kanel conjecture  holds for the ind-group of automorphisms of $\mathcal{D}_X$ and the Poisson-automorphisms of $\OO_{T^*X^{\mathrm{sm}}}(T^*X^{\textrm{sm}})$ for more general schemes  (e.g. reductive monoids and horospherical spaces). 
\end{rem}

\subsubsection{Quadric Cones}  \label{sssec:quadric}
We now prove Conjecture \ref{special:case:ansatz} for the cones $C_n$ of \S \ref{sec:Mod:cones}.  

For comparison with \cite{Kobayashi:Mano} it is helpful to change the quadratic form defining our orthogonal group.  We therefore use different notation from \S \ref{sec:Mod:cones}.  Thus we write
\begin{align}
    J_{p,q}:=\begin{psmatrix} I_p & \\ &-I_q \end{psmatrix}
\end{align}
and let $\mathrm{O}_{p,q}$ be the associated orthogonal group, etc.  
We write 
\begin{align*}
\epsilon_i=\begin{cases} 1 & \textrm{ if } 1 \leq i \leq p,\\ -1 & \textrm{ if } p+1 \leq i \leq p+q.
\end{cases}
\end{align*} 
We assume $p+q$ is even. Let $\mathcal{Q}_{p,q}$ be the quadratic form on $V_{p,q}:=\GG_a^{p+q}$ associated to $J_{p,q}$ and let $C_{p,q} \subset V_{p,q}$ be the vanishing locus of $\mathcal{Q}_{p,q}.$ 

We define a family of differential operators using the identification
\begin{align}
    \Gamma(C_{p,q},\OO_{C_{p,q}})=\RR[x_1,\dots,x_{p+q}]/\mathcal{Q}_{p,q}.
\end{align}
Let
$$
\Box := \sum_{i=1}^n \epsilon_i \frac{\partial^2}{\partial x_i^2}
$$
be the Laplace-Beltrami operator. 

Consider the following operators:
\begin{enumerate}
    \item \label{xi1} $x_i,$ $1 \leq i \leq p+q$
    \item \label{xij} $
     X_{ij} := \epsilon_i\epsilon_j x_i \frac{\partial}{\partial x_j} - x_j\frac{\partial}{\partial x_i}$, $1 \leq i <j \leq p+q$
     \item \label{E} $E= \sum_i x_i \frac{\partial}{\partial x_i}$
     \item \label{Pi} $
P_i : = \epsilon_i x_i \Box - (2E + p+q -2) \frac{\partial}{\partial x_i},$ $1 \leq i \leq p+q$
\end{enumerate}
These all define elements of $\mathcal{D}_C$ \cite[\S 1.1]{Kobayashi:Mano}.

Let 
$$
\mathfrak{x}=\mathfrak{x}_2 \oplus \mathfrak{x}_1 \oplus \mathfrak{x}_0
$$
where 
\begin{align} \label{xi2} \begin{split}
    \mathfrak{x}_2:&=\langle P_1,\dots,P_{p+q}\rangle,\\
    \mathfrak{x}_1:&=\langle E+
\tfrac{p+q-2}{2},\{X_{ij}:1 \leq i <j \leq p+q\} \rangle,\\
    \mathfrak{x}_0:&=\langle x_1,\dots,x_{p+q} \rangle \end{split}
\end{align}
and the brackets denote the $\CC$-span.  The commutators of these elements in $\mathcal{D}_C$ are computed in \cite[\S 2.4]{Kobayashi:Mano}.  Let $\mathcal{U}(\mathfrak{g})$ be the universal enveloping algebra of the complex Lie algebra $\mathfrak{g}.$

\begin{prop} \label{prop:gen}
The space $\mathfrak{x}$ is a Lie subalgebra of $\mathcal{D}_{C\CC}$ under the bracket and is isomorphic to
$(\mathfrak{o}_{p+1,q+1})_{\CC}.$  
The induced map
$\mathcal{U}((\mathfrak{o}_{p+1,q+1})_{\CC}) \to \mathcal{D}_C$
is surjective.  
\end{prop}
\begin{proof} 
We observe that all nondegenerate quadratic forms on a complex vector space are equivalent. Thus the embedding $\mathfrak{x} \subset \mathcal{D}_{C\CC}$ is equivalent to the embedding constructed in \cite[Example]{Goncharov1982Weil}.  
With this in mind the main result of \cite{LSSMinimalNilpotent} implies the proposition.
\end{proof}

As in \S \ref{ssec:mod:grp}, $\mathcal{S}(C_{p,q}(\RR))=L^2(C_{p,q}^\circ(\RR))^{\mathrm{sm}},$ where the superscript indicates vectors that are smooth under the action of $\mathrm{O}_{p+1,q+1}(\RR).$  

\begin{lem} \label{lem:compat}
    An $\mathrm{O}_{p+1,q+1}(\CC)$-conjugate of the action of $\mathfrak{o}_{p+1,q+1}$ on $\mathcal{S}(C_{p,q}(\RR))$ via $\mathcal{D}_C$  coincides with the infinitesimal action of the minimal representation. 
\end{lem}
\begin{proof}
By \cite[(2.3.15), (2.3.19)]{Kobayashi:Mano} the action of $x_i$ and $P_i$ can be realized using a $\mathrm{O}_{p,q}(\CC)$-conjugate of the infinitesimal action.  Since these operators generate the Lie algebra by \cite[Lemma 2.4.8]{Kobayashi:Mano} we deduce the lemma.
\end{proof}

\begin{prop}\label{QuadDiffOp}
    Ansatz \ref{ans:DX} and \ref{ans:alg} are valid for $\mathcal{D}_{C}.$
    One has that
    \begin{align*}
    \mathcal{F}_{C_{p,q}} \circ E \circ \mathcal{F}_{C_{p,q}}&=-(E+p+q-2), & \mathcal{F}_{C_{p,q}} \circ 4x_i \circ \mathcal{F}_{C_{p,q}}&=P_i,\\ \mathcal{F}_{C_{p,q}} \circ P_i \circ \mathcal{F}_{C_{p,q}}&=4x_i, &\mathcal{F}_{C_{p,q}} \circ X_{ij} \circ \mathcal{F}_{C_{p,q}}&=X_{ij}.
    \end{align*}
\end{prop}

\begin{proof} The first assertion follows immediately from Lemma \ref{lem:compat} and the surjectivity statement in Proposition \ref{prop:gen}.
The first three equalities at the end of the proposition are \cite[Theorem 2.5.2(3)]{Kobayashi:Mano}. 
The last identity follows from the fact that $X_{ij}$ is the differential of a particular element of $\mathrm{O}_{p,q}(\RR)$ \cite[\S 1.1]{Kobayashi:Mano}, and hence commutes with $\mathcal{F}_{C_{p,q}}.$
\end{proof}

Let $\omega:C_{p,q}\to V_{p,q}$ be the canonical embedding and let $\Psi_{\omega}^{\mathrm{ind}}=\mathrm{O}_{p+1,q+1}$.  This group acts by conjugation on the minimal nilpotent orbit $\mathbb{O}_{p+1,q+1} \subset \mathfrak{o}_{p+1,q+1}$ and hence on its closure $\overline{\mathbb{O}}_{p+1,q+1}$ in $\mathfrak{o}_{p+1,q+1}.$  We therefore obtain an action on $\overline{T^*C_{p,q}^\circ}$ using Proposition \ref{minNilpOrbitIso}.

\begin{prop} \label{prop:cone:Ansatz}
 For the canonical embedding $\omega: C_{p,q} \to V_{p,q}$ Ansatz \ref{Ansatz} holds with $\Psi_{\omega}^{\mathrm{ind}}$ and $\mathfrak{x}$ as above.
\end{prop}

\begin{proof}
Statement \eqref{xgen}  is part of Proposition \ref{prop:gen}, and \eqref{preserve} is a consequence of Lemma \ref{lem:compat}.

The symbol map is $\CC$-linear on the subspaces $\mathfrak{x}_i.$  Let us describe it explicitly.  As in the proof of Proposition \ref{minNilpOrbitIso}
$$
T^*C_{p,q}^\circ=(C_{p,q}^\circ \times V_{p,q})/\GG_a
$$
where the implied action of $\GG_a$ is given in \eqref{Ga:act}.  Thus we can identify functions on $T^*C^\circ_{p,q}$ with functions on $C_{p,q}^\circ \times V_{p,q}$ that are invariant under the action of $\GG_a.$
 Under this identification, the principal symbol of the operator $\partial/\partial x_i$ is $v_i$, the corresponding coordinate function on $V_{p,q}$.  Therefore
\begin{align}
\sigma(P_i)&=\epsilon_i x_i\mathcal{Q}_{p,q}(v)-2(x\cdot v) v_i \label{symb1}\\
\sigma\left(E+ {\frac{p+q-2}{2}}\right)&=x\cdot v \label{symb2}\\
\sigma(X_{ij}) &= \epsilon_i\epsilon_j x_i v_j - x_j v_i \label{symb3}\\
\sigma(x_i)&=x_i \label{symb4},
\end{align}
\noindent where $\cdot$ denotes the standard dot product. We let $\mathfrak{x}'$ denote the $\RR$-span of (\ref{symb1})-(\ref{symb4}). As a Lie algebra under the Poisson bracket, we have by construction that $\mathfrak{x}'\cong \mathfrak{x}$, so by Proposition \ref{prop:gen}, $\mathfrak{x}'\cong \mathfrak{o}_{p+1,q+1}$. 

Under the embedding $T^*C_{p,q}^\circ(\RR) \hookrightarrow \mathbb{O}_{p+1,q+1}(\RR)$, written explicitly in (\ref{invariantscone}) (with respect to a different basis) we find that (\ref{symb4}) gives the $i$th coordinate of $c$,  (\ref{symb3}) gives $\left((-J_nv)_ic_j+(J_n c^t)_iv_j\right)_{ij}$, (\ref{symb2}) gives $\langle c, v\rangle$, and (\ref{symb1}) gives the $i$th coordinate of $J_nc^t\mathcal{Q}_n(v)-\langle c,v \rangle_n J_nv^t$. Thus the functions (\ref{symb1})-(\ref{symb4}) are the pullback of the standard matrix coordinates of $\mathfrak{o}_{p+1, q+1}$ to $T^*C_{p,q}^{\circ}$ under the map $\overline{a}:T^*C_{p,q}^{\circ} \to \mathbb{O}_{p+1,q+1}$ constructed in the proof of Proposition \ref{minNilpOrbitIso}.
Since $\overline{a}$ is an isomorphism on affinizations, and the matrix coordinates of $\mathfrak{o}_{p+1, q+1}$ generate the ring of regular functions on the closed subset $\overline{\mathbb{O}}_{p+1,q+1} \subset \mathfrak{o}_{p+1,q+1}$,  we see that the principal symbols \eqref{symb1}-\eqref{symb4} generate all regular functions on $\overline{T^*C^\circ_{p,q}}$. This proves (\ref{cot:gen}). 

The ind-group scheme $\Psi_{\omega}^{\mathrm{ind}}$ in (\ref{ind:gp}) is $\textrm{O}_{p+1,q+1}$, which acts on $\overline{T^*C^{\circ}} \cong \overline{\mathbb{O}}_{p+1,q+1}$ via the manifest conjugation action on the minimal nilpotent orbit. Using \eqref{action:map} we see that the small modulation group corresponds to the action of  $Q_i^{\mathrm{op}}$, up to a conjugation due to the fact that we are using a different quadratic form.  

The action of $\mathrm{O}_{p+1,q+1}$ on $\overline{T^*C_{p,q}^{\circ}}$ induces an action of $\mathrm{O}_{p+1,q+1}$ on $\mathfrak{x}'$ that is just matrix conjugation. On the other hand up to conjugation by an element of $\mathrm{O}_{p+1,q+1}(\CC)$ the action of $\Psi_{\omega}\{\RR\}$ on $\mathcal{D}_C$ is induced by the conjugation action on $\mathfrak{x}=(\mathfrak{o}_{p+1,q+1})_{\CC}$
by Lemma \ref{lem:compat}.  Assertion \eqref{compat} follows.
\end{proof}
\quash{
\begin{rem}
    Observe that the Grothendieck filtration on the differential operators $\mathcal{D}_{C_{p,q}}$ is \textit{not} compatible with the PBW filtration under the map $\mathcal{U}(\mathfrak{o}_{V_{n+2}}) \to \mathcal{D}_C$ in Proposition \ref{QuadDiffOp}. 
    
    We note that
    $\overline{\mathbb{O}_{\min}} \subset \mathfrak{o}_{V_{n+2}}^*$ implies that we have a surjective map:

    \[
    \textrm{gr}_{\textrm{PBW}}\, \mathcal{U}(\mathfrak{o}_{V_{n+2}}) =  \textrm{Sym}^\bullet (\mathfrak{o}_{V_{n+2}})\twoheadrightarrow \Gamma(T^*C,\OO_{T^*C}).
    \]

    \noindent where the grading is taken with respect to the PBW filtration. So, somewhat surprisingly, this is \textit{incompatible} with the graded isomorphism 

    \[
    \textrm{gr}_{\textrm{Groth}}\mathcal{D}_C \simeq \Gamma(T^*C^{\circ},\OO_{T^*C^\circ}).
    \]
\end{rem}}

\section{The boundary of the Schwartz space and small modulation groups}
\label{sec:Boundary}

Assume that we are in the situation of \S \ref{ssec:Schwartz}, with a Schwartz space satisfying \eqref{assum:preserve}, \eqref{assum:local}, and \eqref{assum:sm}. We assume that $X$ admits an open dense $H$-orbit $X^\circ \subset X.$

\begin{defn} The \textbf{boundary} of the Schwartz space $\mathcal{S}(X(F),\mathcal{L}^{1/2})$ is the quotient 
$$
\mathcal{S}(X(F),\mathcal{L}^{1/2})/\mathcal{S}(X^{\circ}(F),\mathcal{L}^{1/2}).
$$
\end{defn}
Kazhdan suggested to Getz that the quotient $\mathcal{S}(X(F),\mathcal{L}^{1/2})/\mathcal{S}(X^{\circ}(F),\mathcal{L}^{1/2})$ is the local avatar of boundary terms in the Poisson summation formula; we refer to \S \ref{sec:PSC} below for more details.

In \S \ref{sec:mod:alg} we constructed an action of $\Psi_{\omega}^{\mathrm{s}}$ on $T^*X^{\mathrm{sm}}.$  It extends to the affine closure $\overline{T^*X^{\mathrm{sm}}}^{\mathrm{aff}}.$
We pose the following imprecise conjecture:
\begin{conj}
Assume $\omega$ is a closed immersion.  Then there is a correspondence between $\Psi_{\omega}^{\mathrm{s}}$-orbits on $\overline{T^*X^{\mathrm{sm}}}^{\mathrm{aff}}$ and subquotients of $\mathcal{S}(X(F),\mathcal{L}^{1/2})$ as a $\Psi_{\omega}^{\mathrm{s}}(F)$-module.
\end{conj}
  We do not have a conjectural formulation of the correspondence at the moment.  We content ourselves with discussing a pair of examples.

\subsection{Vector spaces} \label{ssec:standard:rep:loc:bd}
We consider  the setting of  \S \ref{ssec:ID}.  Thus let $\GL_V$ act on $V.$ As in \eqref{Psiid}  $\Psi_{\mathrm{id}}^{\mathrm{s}}=V^\vee \rtimes \GL_V.$

One can directly verify the following two lemmas:
\begin{lem} \label{lem:cotangent:decomp:st}
    One has a decomposition
    $$
    T^*V:=V \oplus V^\vee=V^\circ \oplus V^\vee \bigsqcup \{0\} \oplus V^\vee
    $$
    into $\mathrm{\Psi}_{\mathrm{id}}^{\mathrm{s}}$-orbits. \qed
\end{lem}

\begin{lem} \label{lem:boundary:decomp:st}
One has an exact sequence 
$$
0 \lto \mathcal{S}(V^\circ(F))  \lto \mathcal{S}(V(F)) \lto \CC \lto 0
$$
of $\Psi_{\mathrm{id}}^{\mathrm{s}}(F)$-modules. \qed
\end{lem}

This example is suggestive, but the reader may complain that if we wanted to find a geometric analogue of the decomposition in Lemma \ref{lem:boundary:decomp:st} then one could just use the orbits of $\GL_V$ on $V.$  This na\"ive approach does not work in  the singular case, as we will demonstrate by example in the next subsection.

\subsection{Quadric cones}
We now turn to the setting of \S \ref{sec:Mod:cones}.  We assume for simplicity that the quadratic form $\mathcal{Q}_n$ is split.  Thus $C_n$ is the zero locus of $\mathcal{Q}_n$ inside the space $\GG_a^{2n},$ and $\omega:C_n \to \GG_a^{2n}$ is the inclusion.  

The following was proved in \cite{Tome}: 

\begin{thm}\label{thm:quadconedecomp}
   For $n>1$ there is a $\Psi_\omega^{\mathrm{s}}$-equivariant decomposition  into four subschemes
    \begin{align*} 
    \overline{T^*(C_{n}^\circ)}^{\mathrm{aff}} &= T^*(C_{n}^\circ)   \sqcup  (\overline{T^*(C_{n-1}^\circ)}^{\mathrm{aff}}-\{0\}) \times \mathbb{G}_a^{2} \sqcup \{0\} \sqcup C_n^\circ. 
    \end{align*}
\qed
\end{thm}
\noindent Thus $\overline{T^*(C_{n}^\circ)}^{\mathrm{aff}}$ has a nested structure.  

The Schwartz space admits an analogous decomposition \cite[Theorem 1.2]{GurK:Cone}:
\begin{thm}
Assume $F$ is non-Archimedean and that $n \geq 3.$
One has an exact sequence
$$
0 \lto \mathcal{S}(C_n^\circ(F)) \lto \mathcal{S}(C_n(F)) \lto \mathcal{S}(C_{n-1}(F)) \oplus \CC \lto 0.
$$ 
It is $\Psi_{\omega}^{\mathrm{s}}(F)$-equivariant up to twisting by appropriate quasi-characters.
\qed
\end{thm}
We suspect that $\mathcal{S}(C_n^\circ(F))$ should correspond to $T^*(C_n^\circ),$ the constant quotient $\CC$ should correspond to $C_{n}^\circ,$ and $\mathcal{S}(C_{n-1}(F))$ should correspond to the remaining terms.
  Unfortunately we do not yet know what we mean by ``correspond.''

\section{Poisson summation} \label{sec:PSC}

In this section we discuss the global theory, so $F$ is a number field.  We place ourselves in the general setting of \S \ref{sec:mod:groups} and assume \eqref{omega:span}, \eqref{omega:scale}, \eqref{HX:H1=0}, \eqref{HX:Fpoints}. 

We assume for simplicity that there is an open dense $H$-orbit $X^\circ \subset X.$  We assume moreover that $X^\circ$ admits a nowhere vanishing section of the canonical bundle that transforms under the action of $H$ by a character. 
For each place $v$ the nonvanishing section gives rise to an $H(F_v)$-eigenmeasure $dx_v$ on $X^\circ(F_v).$  Using this we identify the local Schwartz spaces $\mathcal{S}(X(F_v),\mathcal{L}^{1/2})$ with spaces of functions $\mathcal{S}(X(F_v))<L^2(X(F_v)):=L^2(X(F_v),dx_v)$ as in Remark \ref{rem:half}.

We assume  Ansatz \ref{Ans:Schwartz}.  Thus we have local Fourier transforms 
$$
\mathcal{F}_{X}:=\mathcal{F}_{X_{F_v}}:\mathcal{S}(X(F_v)) \lto \mathcal{S}(X(F_v)).
$$
Let $S$ be a finite set of places of $F$ including the infinite places.  
By an $\OO_F^S$-model of a scheme $Z$ of finite type over $F$ we mean a flat $\OO_F^S$-scheme of finite type whose generic fiber is $Z.$  By abuse of notation, we continue to denote the model by $Z.$  We choose $\OO_F^S$-models of all of the schemes appearing in \S\ref{sec:mod:groups} and assume that the relevant morphisms all extend over $\OO_F^S.$  Thus, in particular, we have an action
$
X \times H \to X
$
over $\OO_F^S.$  We extend the open $H$-orbit $X^\circ \subset X$ over $\OO_F^S$ by taking the schematic closure in $X.$   

For $v \not \in S$ we assume the existence of basic functions
\begin{align}
    b_{X_{F_v}} \in \mathcal{S}(X(F_v))^{H(\OO_{F_v})}
\end{align}
such that 
\begin{enumerate}[label=($b$\text{\arabic*}), ref=$b$\text{\arabic*}]
\item \label{b:supp} $X^\circ(\OO_{F_v}) \subsetneq \mathrm{supp}(b_{X_{F_v}}) \subseteq X(\OO_{F_v}),$
    \item \label{b:norm}$b_{X_{F_v}}|_{X^\circ(\OO_{F_v})}=1,$
    \item\label{assum:bx:preserve} $\mathcal{F}_{X}(b_{X_{F_v}})=b_{X_{F_v}}.$
\end{enumerate}
We assume that for $v|\infty$ the space $\mathcal{S}(X(F_v))$ is a Fr\'echet space.
We define 
\begin{align}
\mathcal{S}(X(\A_F)):=\widehat{\otimes}_{v|\infty}\mathcal{S}(X(F_v)) \otimes \otimes_{v\nmid \infty}'\mathcal{S}(X(F_v))
\end{align}
where the hat denotes the completed projective tensor product and the restricted tensor product is with respect to the basic functions $b_{X_{F_v}}.$  The tensor product of the local Fourier transforms yields a map $\mathcal{F}_X :\mathcal{S}(X(\A_F)) \to \mathcal{S}(X(\A_F))$ by \eqref{assum:bx:preserve}. Throughout this section we assume that 
$$
\sum_{\gamma \in X^\circ(F)}|f(\gamma)|<\infty
$$
for all $f \in \mathcal{S}(X(\A_F)).$ 

\begin{ans}[The Poisson summation formula] \label{ans:PSC}
    Let $v_1$ and $v_2$ be places of $F$ (not necessarily distinct). 
 If $f=f_{v_1}f_{v_2}f^{v_1v_2} \in \mathcal{S}(X(\A_F))$ where $f_{v_1} \in \mathcal{S}(X^\circ(F_{v_1}))$  and $\mathcal{F}_{X}(f_{v_2}) \in \mathcal{S}(X^\circ(F_{v_2}))$ then
    \begin{align*}
        \sum_{\gamma \in X^\circ(F)}f(\gamma)=\sum_{\gamma \in X^\circ(F)}\mathcal{F}_{X}(f)(\gamma).
    \end{align*}
\end{ans}

\begin{conj}[Braverman and Kazhdan] \label{conj:PSC} Ansatz \ref{ans:PSC} is true when $X$ is a reductive monoid as in \S \ref{sec:RM}.
\end{conj}

As stated in \cite{BK-lifting},  this conjecture implies the functional equation of the Langlands $L$-function attached to $\rho.$  Combining this with the converse theorem \cite{Cogdell:PS:ConverseII}, we see that Conjecture \ref{conj:PSC} implies much of Langlands functoriality.  

The assumption on $f$ in Ansatz \ref{ans:PSC} is unnatural.  For every $f \in \mathcal{S}(X(\A_F))$ one should have that
\begin{align} \label{full:PSC}
    \sum_{\gamma \in X^\circ(F)}f(\gamma)+\mathrm{BT}_{X}(f)=\sum_{\gamma \in X^\circ(F)}\mathcal{F}_{X}(f)(\gamma)+\mathrm{BT}_{X}(\mathcal{F}_{X}(f))
\end{align}
where $\mathrm{BT}_{X}:\mathcal{S}(X(\A_F)) \to \CC$ is a linear functional that is reasonably explicit and related to the geometry of $X.$ 
We refer to the linear functionals $\mathrm{BT}_{X}$ as \textbf{boundary terms}.
These linear functionals remain mysterious at the moment, and it is an important problem to describe them geometrically.  We refer loosely to an identity of the form \eqref{full:PSC} with an explicit, geometric description of the boundary terms $\mathrm{BT}_{X}(f)$ as a \textbf{full Poisson summation formula}.  The full Poisson summation formula is known for horospherical varieties as in \S \ref{sec:HS}  when $G$ is a classical group or $G_2$ \cite{BK:normalized,Choie:Getz,Hsu:Asymp}.
   Apart from the case of matrices, the only case where we have a full Poisson summation formula with a completely geometric description of the boundary terms is the case of the cone considered in \S \ref{sec:Mod:cones}.  We discuss this in \S \ref{ssec:BT:cones} below.

Let us explain the relationship between global boundary terms and the boundary of the Schwartz space, as explained to Getz by Kazhdan.  Let $v$ be a place of $F.$ 
 Assuming Ansatz \ref{ans:PSC}, for any fixed $f^{v} \in \mathcal{S}(X(\A_F^v))$ one has a well-defined functional
\begin{align*}
\mathcal{S}(X(F_v))/\mathcal{S}(X^\circ(F_v)) \cap \mathcal{F}_{X}^{-1}(\mathcal{S}(X^\circ(F_v))) & \lto \CC\\
f_v &\longmapsto    \sum_{\gamma \in X^\circ(F)}f_vf^v(\gamma)-\sum_{\gamma \in X^\circ(F)}\mathcal{F}_{X}(f_vf^v).
\end{align*}
Studying this functional is equivalent to studying the functional 
$$
f_v \longmapsto \mathrm{BT}_X(\mathcal{F}_{X}(f_vf^v))-\mathrm{BT}_X(f_vf^v).
$$
On the other hand we have a series of inclusions
$$
\mathcal{S}(X(F_v)) \geq \mathcal{S}(X^\circ(F_v)) \geq \mathcal{S}(X^\circ(F_v)) \cap \mathcal{F}_{X}^{-1}(\mathcal{S}(X^\circ(F_v)).
$$
Thus understanding the structure of the boundary $\mathcal{S}(X(F_v)) / \mathcal{S}(X^\circ(F_v))$  would at least shed some light on 
the boundary terms.

\subsection{Boundary terms in our two examples}

\subsubsection{Vector spaces}
If $X=V$ then one defines 
$$
\mathrm{BT}_{V}(f):=f(0).
$$
The full Poisson summation formula \eqref{full:PSC} is just the usual Poisson summation formula.

\subsubsection{Quadric cones} \label{ssec:BT:cones}

We briefly recall the main theorem of \cite{Getz:Quadric}.  
In loc.~cit.~the first author defined operators 
\begin{align} \label{I} \begin{split}
I:\mathcal{S}(V_i(\A_F) \oplus \A_F^2) &\lto C^\infty(C_i^\circ(\A_F)) \end{split}
\end{align}
for $i>0,$ together with operators
\begin{align*}
c_i:\mathcal{S}(V_i(\A_F) \oplus \A_F^2) &\lto \CC,\\
d_{i,i'}:\mathcal{S}(V_{i}(\A_F) \oplus \A_F^2)& \lto \mathcal{S}(V_{i'}(\A_F) \oplus \A_F^2)
\end{align*}
for $i > i' \geq 0.$ Briefly, $c_i(f)$ is the regularized value of $I(f)$ at $0 \in C_i(F).$  On the other hand $d_{i,i-1}$ is given by a partial Fourier transform and then restriction to the complement of a hyperbolic plane.  One then sets $d_{i,i'}=d_{i'+1,i'} \circ  \dots \circ d_{i,i-1}.$
By convention, $d_{i,i}$ is the identity.  
We let
\begin{align*}
\mathcal{F}_{\wedge}:\mathcal{S}(\A_F^2) \lto \mathcal{S}(\A_F^2)
\end{align*}
be the usual $\SL_2(\A_F)$-equivariant Fourier transform.  

The main theorem of \cite{Getz:Quadric} follows:
\begin{thm} \label{thm:main:cones} For $f \in \mathcal{S}(V_n(\A_F) \oplus \A_F^2)$ the sum
\begin{align}
\sum_{\xi \in C^{\circ}_n(F)}&I(f)(\xi)+c_n(f) \label{not:bd2} \\
&+\sum_{i=1}^{n-1}\left(c_i(d_{n,i}(f))+\sum_{\xi \in C^{\circ}_i(F)}I(d_{n,i}(f))(\xi)\right)+ \kappa d_{n,0}(f)(0_{V_0},0,0) \label{bd2}
\end{align}
is invariant under $f \mapsto  (1_{\mathcal{S}(V_i(\A_F))} \otimes \mathcal{F}_{\wedge})(f)$.
Here $\kappa$ is a suitable constant. \qed
\end{thm}

There is some description of the boundary terms in the setting of Braverman-Kazhdan spaces and generalized Schubert varieties in \cite{Choie:Getz}, but Theorem \ref{thm:main:cones} is the only case beyond vector spaces where one has a manifestly geometric expression for the boundary terms.  

There is a qualitative connection between the various terms in Theorem \ref{thm:main:cones} and the subschemes in Theorem \ref{thm:quadconedecomp}.  The sum in \eqref{not:bd2} corresponds to $T^*C_n^\circ,$ $c_n(f)$ corresponds to $C_n^\circ,$ and \eqref{bd2} corresponds to the remaining terms.  Just as in the local setting, we do not yet have a conjectural understanding of what we should mean by ``corresponds.''

\section{Automorphic representations of modulation groups} \label{sec:AR}

We continue to work under the assumptions of the previous section.  Choose a character $\chi:\Psi_{\omega}^{\mathrm{s}}(\A_F) \to \CC^\times$ trivial on $\Psi_{\omega}^{\mathrm{s}}(F).$  We assume that $\chi|_{\Psi_{\omega}^{\mathrm{s}}(\widehat{\OO}_F^S)}=1.$

For $v \not \in S$ let
\begin{align}
K_v:=\langle \mathcal{F}_{X},(\mathcal{R}_{\omega}\otimes \chi)(\Psi_{\omega}^{\mathrm{s}}(\OO_{F_v}))\rangle\leq \mathrm{Aut}(L^2(X(F_v))).
\end{align}
We then define the restricted direct product
\begin{align}
    \Psi_{\omega}\{\A_F\}:=\sideset{}{'}\prod_v \Psi_{\omega}\{F_v\}
\end{align}
with respect to the $K_v.$   We also have a subgroup
\begin{align}
    \Psi_\omega\{F\}=\langle \mathcal{F}_{X},\Psi_{\omega}^{\mathrm{s}}(F)\rangle \leq \Psi_{\omega}\{\A_F\}.
\end{align}
Here the implicit map from $\Psi_\omega^{\mathrm{s}}\{F\}$ into $\Psi_\omega^{\mathrm{s}}\{\A_F\}$ is the diagonal embedding.  

We point out that for almost all $v$ the group $K_v$ acts trivially on the basic function $b_{X_{F_v}} \in \mathcal{S}(X(F_v)).$  Thus  we have a representation
\begin{align}
\mathcal{R}_{\omega}:\Psi_{\omega}\{\A_F\} \times \mathcal{S}(X(\A_F)) \lto \mathcal{S}(X(\A_F)).
\end{align}

\begin{prop} \label{prop:auto:func}
 Assume that one has a $\Psi_{\omega}^{\mathrm{s}}(F)$-invariant linear functional
 $$
\mathrm{BT}_X:\mathcal{S}(X(\A_F)) \lto \CC.
 $$
For $(h,f) \in \Psi_\omega(\A_F) \times \mathcal{S}(X(\A_F)) $ consider the function
\begin{align}
\Theta_f(h):=\sum_{\gamma \in X^\circ(F)}\mathcal{R}_\omega(h)f(\gamma)+\mathrm{BT}_X(\mathcal{R}_{\omega}(h)(f)).
\end{align}
 Then
$\Theta_f(h)$ is left $\Psi_{\omega}\{F\}$-invariant if and only if
\begin{align} \label{full:PS}
    \sum_{\gamma \in X^\circ(F)}f(\gamma)+\mathrm{BT}_X(f)=\sum_{\gamma \in X^\circ(F)}\mathcal{F}_{X}(f)(\gamma)+\mathrm{BT}_X(\mathcal{F}_{X}(f)),
\end{align}
that is, the full Poisson summation conjecture holds.
\end{prop}
\begin{proof}
If $\Theta_f(h)$ is $\Psi_\omega\{F\}$-invariant then  \eqref{full:PS} is just the identity $\Theta_f(I)=\Theta_f(\mathcal{F}_X).$

Conversely, suppose that \eqref{full:PS} holds.  Then $\Theta_f(h)$ is invariant under the subgroup generated by $H(F)$ and $\mathcal{F}_{X}$, and the term $\mathrm{BT}_X(\mathcal{R}_{\omega}(h)f)$ is invariant under $\Psi_{\omega}^{\mathrm{s}}$ by assumption.  The sum $\sum_{\gamma \in X^\circ(F)}\mathcal{R}_{\omega}(h)(f)(\gamma)$ is obviously invariant under $H(F),$ and it is invariant under the action of $V^\vee(F)$ because $\psi(\lambda(\omega(\gamma)))=1$ for $(\lambda,\gamma) \in V^\vee(F) \times X^\circ(F).$
\end{proof}

Thus assuming the full Poisson summation conjecture it is reasonable to regard
$
\Theta_f(h)$ as an automorphic function on $\Psi_{\omega}\{F\} \backslash \Psi_{\omega}\{\A_F\},$ and it is reasonable to view  
$$
\{\Theta_f(g):f \in \mathcal{S}(M_{\rho}(\A_F))\}
$$
as an automorphic representation of $\Psi_{\omega}\{\A_F\}.$  Of course the notion of an automorphic representation of $\Psi_{\omega}\{\A_F\}$ has yet to be defined.

\bibliography{refs.bib}{}
\bibliographystyle{alpha}

\end{document}